\documentclass[11pt]{amsart}
\usepackage{amsmath,amssymb,amsthm,mathtools,xypic}
\usepackage[shortlabels]{enumitem}
\usepackage[normalem]{ulem}
\usepackage[hmargin=35mm, vmargin=30mm]{geometry}
\usepackage{xy} %for diagrams
\usepackage{mathrsfs} %for script fonts
\usepackage{xspace}
\usepackage[pagebackref=true, colorlinks]{hyperref}

\title{Locally compact piecewise full groups of homeomorphisms}

\author{Alejandra Garrido} 
\address{Facultad de Matem\'aticas, Universidad Complutense de Madrid, and ICMAT, Madrid, SPAIN}
\email{alejandra.garrido@ucm.es; alejandra.garrido@icmat.es}

\author{Colin D. Reid}
\address{School of Information and Physical Sciences, University of Newcastle Australia}
\email{colin@reidit.net}

\thanks{The research project behind this article started while both authors were employed at the University of Newcastle, supported by the Australian Research Council grant FL170100032 (CI: George Willis).  We thank our erstwhile colleagues at Newcastle for productive discussions, especially David Robertson with whom we discussed the project quite extensively and gave significant insights on some potential sources of examples.
}
\thanks{The project has also benefited from financial support by Spain’s Ministry of Science and Innovation, grants [PID2020-114032GB-I00] and [CEX2019-000904-S]}

\newtheorem{thm}{Theorem}[section]

\newtheorem{prop}[thm]{Proposition}
\newtheorem{lem}[thm]{Lemma}
\newtheorem{cor}[thm]{Corollary}
\newtheorem{que}{Question}
\newtheorem{prob}[que]{Problem}

\theoremstyle{definition}
\newtheorem{defn}[thm]{Definition}
\newtheorem{rmk}[thm]{Remark}

\newtheorem{ex}[thm]{Example}

\newcommand{\Zb}{\mathbb{Z}}
\newcommand{\Nb}{\mathbb{N}}

\newcommand{\Qb}{\mathbb{Q}}

\newcommand{\bZ}{\mathbb{Z}}

\newcommand{\mc}[1]{\mathcal{#1}}

\newcommand{\mf}[1]{\mathfrak{#1}}
\newcommand{\ms}[1]{\mathscr{#1}} %requires mathrsfs

\newcommand{\tdlc}{t.d.l.c.\@\xspace}

\newcommand{\Sym}{\mathrm{Sym}}
\newcommand{\Alt}{\mathrm{Alt}}
\newcommand{\Aut}{\mathrm{Aut}}
\newcommand{\AAut}{\mathrm{AAut}}

\newcommand{\Homeo}{\mathrm{Homeo}}
\newcommand{\PHomeo}{\mathrm{PHomeo}}
\newcommand{\PAut}{\mathrm{PAut}}
\newcommand{\Mon}{\mathrm{M}}
\newcommand{\TMon}{\overline{\mathrm{M}}}
\newcommand{\rist}{\mathrm{rist}}

\DeclareMathOperator{\Comm}{Comm}

\DeclareMathOperator{\QZ}{QZ}

\newcommand{\rest}[2]{{#1}\!\!\restriction_{#2}}

\newcommand{\dom}{\mathrm{dom}}
\newcommand{\ran}{\mathrm{ran}}

\newcommand{\N}{\mathrm{N}}
\newcommand{\CC}{\mathrm{C}}
\newcommand{\Full}{\mathrm{F}}
\newcommand{\Sy}{\mathrm{S}}
\newcommand{\Al}{\mathrm{A}}
\newcommand{\Der}{\mathrm{D}}

\newcommand{\ldlat}{\mathcal{LD}}

\newcommand{\inv}{^{-1}}

\newcommand{\id}{\mathrm{id}}
\newcommand{\ad}{\mathrm{ad}}

\newcommand{\defbold}{\textbf}

\newcommand{\triv}{\{1\}}

\newcommand{\grp}[1]{\langle #1 \rangle}

\newcommand{\ol}[1]{\overline{#1}}

\usepackage[usenames, dvipsnames]{color}

\begin{document}

\begin{abstract}
	We study when a piecewise full group (a.k.a. topological full group) of homeomorphisms of the Cantor space $X$ can be given a non-discrete totally disconnected locally compact (\tdlc) topology and give a criterion for the alternating full group (in the sense of Nekrashevych's group $\Al(G)$) to be compactly generated.
	As a result, starting from qualitative criteria, we obtain a large class of \tdlc groups such that the derived group is non-discrete, compactly generated, open and simple, putting previous constructions of Neretin, R\"{o}ver and Lederle in a more systematic context.  We also show some notable properties of Neretin's groups apply to this class in general.
	General consequences are derived for the theory of simple \tdlc groups, prime among them the universal role that alternating full groups play in the class of simple \tdlc groups that are non-discrete, compactly generated and locally decomposable.
	
	Some of the theory is developed in the setting of topological inverse monoids of partial homeomorphisms of $X$.  In particular, we obtain a sufficient condition to extend the topology to a monoid equipped with all restrictions with respect to compact open subsets of $X$ and all joins of compatible pairs of elements.  The compact generation criterion is also naturally expressed in this context.
\end{abstract}

%\keywords{almost simple group, simple derived group, commutator group, totally disconnected locally compact group, action on Cantor space, Cantor dynamics, micro-supported action}
\subjclass[2020]{22D99, %Locally compact groups and their algebras (None of the above, but in this section),
	22D05, %General properties and structure of locally compact groups 
	22F50; %groups as automorphisms of other structures
	20E32, %simple groups
	22A15, %Structure of topological semigroups
	37C85, %	Dynamics induced by group actions other than Z	and R	, and C
	20E08, %Groups acting on trees 
	37B05%Dynamical systems involving transformations and group actions with special properties (minimality, distality, proximality, expansivity, etc.) 
}

\maketitle

\tableofcontents

\addtocontents{toc}{\protect\setcounter{tocdepth}{1}}

\section{Introduction}

\subsection{Background and motivation}

\begin{defn}
Let $X$ be a compact zero-dimensional topological space and let $G \le \Homeo(X)$. 
 The \defbold{piecewise full group} $\Full(G;X)$ of $G$ is the group of all homeomorphisms $h$ of $X$ such that there exists a clopen partition $\{\alpha_i \mid i \in I\}$ of $X$ and corresponding elements $(g_i)_{i \in I}$ of $G$ such that $\rest{h}{\alpha_i} = \rest{g_i}{\alpha_i}$ for all $i$.
 The action of $G$ is \defbold{piecewise full} on $X$ if $G = \Full(G;X)$.
\end{defn}

Originally introduced as a tool in topological dynamics (\cite{Putnam}, \cite{GlasnerWeiss}, \cite{GPS99}), piecewise full groups (also known as topological full groups\footnote{We prefer to avoid the term `topological full groups' in this article to avoid ambiguity in the use of the word `topological', because we will be discussing `topological groups', that is, groups equipped with a topology compatible with the group operations.}) have in recent years become an important source of examples of groups with interesting properties, especially simple groups.  An important example is the theorem of Juschenko--Monod \cite{JuschenkoMonod} that the piecewise full group of the cyclic group generated by a minimal homeomorphism of the Cantor space is amenable; this gave rise to the first examples of finitely generated infinite amenable simple groups.  Subsequent work has produced finitely generated, simple groups of intermediate word growth which are moreover torsion groups (\cite{Nekr_annals}).  On the other hand, Thompson's group $V$, and more generally the Higman--Thompson groups, are important examples of groups of independent interest (this time non-amenable) that can be realised as piecewise full groups over the Cantor space.  See for example \cite{Katzlinger} for a survey of piecewise full groups and their history.

Going beyond countable groups, piecewise full groups have also provided interesting examples of non-discrete totally disconnected locally compact (\tdlc) topological groups; in particular, examples with a normal subgroup that falls into the class $\ms{S}$ of \tdlc groups that are non-discrete, compactly generated and topologically simple.  Interest in this class has increased in the last few years, thanks to reduction theorems for \tdlc groups to this class (\cite{CapraceMonod}) and the development of a structure theory for groups in $\ms{S}$ (\cite{CRW-Part2}).

The most well-known example is Neretin's group $\AAut(T)$ of almost automorphisms of the regular locally finite tree $T$ (of some degree $d \ge 3$).  It was shown by Kapoudjian \cite{Kapoudjian} that this group is abstractly simple and compactly generated: more precisely, $\AAut(T)$ is generated by the group $\Aut(T)$ of automorphisms of the tree, together with an appropriate Higman--Thompson group (which is finitely generated).  Neretin's groups have since been shown to have remarkable properties of their own.  For example, Caprace--Le Boudec--Matte Bon \cite{CLBMB} showed that $\AAut(T)$ acts freely on its Furstenberg boundary; this is a strong non-amenability property.  It has also been shown that $\AAut(T)$ has no lattice (\cite{BCGM}) and even lacks non-trivial invariant random subgroups (\cite{Zheng}).  Generalisations of Neretin's groups were studied more recently by Lederle (\cite{Lederle}, \cite{Lederle-V}), producing more examples of groups in $\ms{S}$.  Other constructions in a similar spirit have been considered starting from a branch group acting on a rooted tree, for example by R\"{o}ver \cite{Roever}.

Despite the results obtained for specific constructions of piecewise full \tdlc groups acting on the Cantor space $X$, at the time of writing there has not been a systematic treatment of the theory of such groups from the perspective of topological group theory, and their general relationship to the class $\ms{S}$.  The goal of this article is to begin the development of such a theory.  In the process, we will not only obtain results on piecewise full groups, but also obtain some structural features of the class $\ms{S}$ modulo local isomorphism, where two \tdlc groups $G$ and $H$ are \defbold{locally isomorphic} if there is an open subgroup of $G$ isomorphic to an open subgroup of $H$.

\subsection{Overall strategy}

Before giving a detailed breakdown of results, it is useful to give a high-level overview of our approach.  We are presented with the following two challenges:
\begin{itemize}
\item The locally compact group topology of the Neretin and Lederle groups is not the compact-open topology.  Indeed, there do not exist any non-discrete locally compact piecewise full groups of homeomorphisms of the Cantor space with the compact-open topology (\cite{GR_uniformdiscrete}).  So for the general theory, we want to consider groups equipped with a given locally compact topology, and what properties of this topology are implied by having piecewise full action on the Cantor space.
\item Much of the existing literature on piecewise full groups has developed not purely at the level of group theory, but via a topological groupoid $\mc{G}$ associated to the action (for example \cite{MatuiHomology}, \cite{Nekra}); one can then define the piecewise full group as the group of global bisections of the action groupoid.  The groupoid is constructed to be \'{e}tale, which effectively means it locally `remembers' the topology of the Cantor space.   However, the \'{e}tale property at the groupoid level translates to a discrete topology on the associated piecewise full group, and it is awkward to introduce a group topology in this setting.
\end{itemize}

In order to tackle the second issue, and with consequences for how we handle the first issue, for part of the article we pass to the setting of inverse monoids of partial homeomorphisms of the Cantor space $X$ whose domain and range are clopen.  In this setting, the analogue of being `piecewise full' is what we call a Boolean inverse monoid $B$.  One then recovers a piecewise full group as the group of units of $B$.  So far, this is equivalent to the \'{e}tale groupoid approach (where $B$ is the set of clopen bisections).

The novelty comes in how we equip $B$ with its own topology.  Our framework is to suppose we already have a topological inverse monoid $M$ of homeomorphisms (which could in fact be a group of homeomorphisms), which we `complete' to form its Boolean inverse monoid, and then the question is whether the topology of $M$ can be extended to $B$.  We find a property, `locally decomposable', that is sufficient to accomplish this (a generalisation of locally decomposable group actions in the sense of \cite{CRW-Part1}).  Analogously, we recall from \cite{CRW-Part2} that one of the basic cases of groups $G$ in $\ms{S}$ is the locally decomposable case; write $\ms{S}_{\ldlat}$ for the groups $G$ in $\ms{S}$ that have a faithful locally decomposable action on the Cantor space.

The treatment of inverse monoids may be of independent interest, but for group theory, the payoff is that we generalise a \emph{finite} generation theorem due to Nekrashevych (originally expressing in terms of \'{e}tale groupoids) to a \emph{compact} generation theorem.  This sets the stage for the second half of the paper, where we return to the setting of topological group theory and specifically \tdlc group theory.  By invoking some prior results (including from the companion article \cite{GR_compressible}), we show that the piecewise full \tdlc groups $F$ we have been focused on are actually $\ms{S}_{\ldlat}$-by-(discrete abelian), with fully compressible action on the Cantor space (that is, given any two non-empty clopen subsets of $X$, one can be mapped inside the other).

For convenience of discussion, we introduce a class $\ms{A}$ of \tdlc groups, which consists of all groups $\Der(F) \le G \le F$ where $F$ is as in the last paragraph.  For the last stage of this article, there are two goals: first to prove some properties of simple groups in $\ms{A}$, and second to characterise the \emph{local} isomorphism classes occurring in $\ms{A}$, in terms of more qualitative properties of \tdlc groups.  In particular, we find that every group in $\ms{S}_{\ldlat}$ is an open subgroup of a group in $\ms{A} \cap \ms{S}_{\ldlat}$.

\subsection{Main results}

For concreteness, throughout this introduction we will take $X$ to be the Cantor space, although for many of the results it is sufficient to take a compact zero-dimensional topological space.  We also assume that groups and monoids are given as acting on $X$ by homeomorphisms, respectively as partial homeomorphisms with clopen domain and range.

Given $G \le \Homeo(X)$, one can always define the piecewise full group $\Full(G)$ as a subgroup of $\Homeo(X)$, but as noted, the compact-open topology is not a suitable choice of topology on $\Full(G)$ for our goals.  The basic example of Neretin's group, where $\Aut(T)$ is embedded as an open subgroup of $\AAut(T)$, suggests an alternative approach.  For general reasons, given a group $A$, a subgroup $B$ and a group topology $\tau_B$ on $B$, there is at most one group topology $\tau_A$ on $A$ such that the natural inclusion of $B$ into $A$ is continuous and open.  If it exists, we refer to $\tau_A$ as the extension of $\tau_B$ to $A$.

The right condition for the topology of $G$ to extend to $\Full(G)$ is that the action should be \defbold{locally decomposable}: for every clopen partition $\mc{P}$ of $X$, the product of the rigid stabilisers $\prod_{Y \in \mc{P}}\rist_G(Y)$ is open in $G$ (see Definition \ref{defn:locally_decomposable_group}).

The main results of Section~\ref{sec:groups} on extending group topologies can be summarised as follows.
\begin{thm}[Proposition \ref{prop:Polishfull_implies_locdec}, Corollary \ref{cor:piecewise_full_extension}]\label{thm:intro_localdec}
	Let $G\leq\Homeo(X)$.
	\begin{enumerate}
		\item If $G$ is \tdlc and locally decomposable  then $\Full(G)$ can be given a unique \tdlc group topology for which $G$ is open. 
		\item If $\Full(G)$ is \tdlc and second-countable then $\Full(G)$ is locally decomposable, so $G$ is also locally decomposable if it is open. 
	\end{enumerate}
\end{thm}

In fact, we can take a somewhat more general perspective, going via monoids of partial homeomorphisms, as set out in Section~\ref{sec:inverse_monoids}.  (This more general perspective is only critical in stating and proving Theorem~\ref{thm:intro_expansive_generation} below; for later results, the reader is free to only think in terms of subgroups of $\Homeo(X)$.)   Let $X$ be a compact zero-dimensional space as before.  
The construction of the piecewise full group $\Full(G)$ of a group $G\leq\Homeo(X)$ can be thought of as having two steps.

First, we take the set $M$ of all restrictions of elements of $G$ to clopen subsets of $X$ (equivalently, to compact and open subsets,  since $X$ is compact and zero-dimensional).  These restrictions belong to the set $\PHomeo_c(X)$ of \defbold{clopen partial homeomorphisms} of $X$, that is, homeomorphisms $f$ such that the domain $\dom(f)$ and the range $\ran(f)$ are clopen subsets of $X$.
The set $\PHomeo_c(X)$ is an inverse monoid where the inverse $f^*$ of $f$ is just the inverse of $f$ as a function\footnote{We use $f^*$ rather than $f^{-1}$ to emphasise that this is an inverse in the sense of semigroup theory, not a group-theoretic inverse.}, and composition of $f,g \in  \PHomeo_c(X)$
is given by
\[
fg: g^{-1}(\dom(f) \cap \ran(g)) \rightarrow f(\dom(f) \cap \ran(g)); \quad x \mapsto f(g(x)).
\]
It is then easy to check that $M$ is an inverse monoid of $\PHomeo_c(X)$.  In fact, it is the inverse submonoid generated by $G$ together with the idempotents of $\PHomeo_c(X)$ (which are just the identity maps on clopen subsets of $X$).  For the rest of the construction, we can actually forget the group $G$ and just use the inverse monoid $M$.  In fact, there is no need to assume $M$ arises as the set of restrictions of some subgroup of $\Homeo(X)$; we could just take $M$ to be any inverse submonoid of $\PHomeo_c(X)$.

Second, given $f,g \in \PHomeo_c(X)$ that are compatible with each other (meaning there is no point $x \in X$ such that $f(x)$ and $g(x)$ are defined and different, and similarly for $f^*,g^*$), then there is a natural `join' $f \vee g$, which is the unique element $h \in \PHomeo_c(X)$ with domain $\dom(f) \cup \dom(g)$ such that both $f$ and $g$ occur as restrictions of $h$.  (More generally, we say a subset $M$ of $\PHomeo_c(X)$ is \defbold{finitely joinable} if $M$ contains the join of any compatible pair of elements of $M$.) The set of joins of restrictions of finitely many elements of $M$ themselves form an inverse monoid, which we call the \defbold{Boolean completion} $\mc{BI}(M)$ of $M$.  Finally, $\Full(M)$ ($=\Full(G)$) is just the group of units of $\mc{BI}(M)$.

We do not develop the theory of topological inverse monoids in full generality, however there is a suitable notion of a `locally decomposable' topology on an inverse submonoid $M$ of $\PHomeo_c(X)$ (Definition~\ref{defn:locally_decomposable}); in the case $M \le \Homeo(X)$ this agrees with the group-theoretic definition.  Starting from a locally decomposable inverse submonoid $M$ of $\PHomeo_c(X)$, we show the topology extends in a natural way to a locally decomposable topology on $\mc{BI}(M)$, which then passes to the subgroup $\Full(M)$:

\begin{thm}[Corollary \ref{cor:piecewise_full_monoid_extension}]\label{thm:intro_BIM}
	Let $M$ be an inverse submonoid of $\PHomeo_c(X)$ equipped with a locally decomposable topology. 
	Then there is a unique Boolean inverse monoid topology (see Definition~\ref{defn:Boolean_inverse:topological}) on $\mc{BI}(M)$ that extends that of $M$. 
\end{thm}

\begin{rmk}
For readers familiar with the \'{e}tale groupoids approach to piecewise full groups, any inverse submonoid $M$ of $\PHomeo_c(X)$ naturally gives rise to an \'{e}tale groupoid $\mc{M}$, and then $\mc{BI}(M)$ is just the set of clopen bisections of $\mc{M}$.  However, we do not need to consider the groupoid directly in this article.  We find the inverse monoid setting more convenient as we want to endow these objects with a topology and \'{e}tale groupoids by definition already come with a topology.

Another approach that has been pursued extensively in the literature (\cite{MatteBon} is a pertinent example) is to use pseudogroups.  The definition of the pseudogroup $\Psi(M)$ associated to $M$ would differ from the that of the Boolean inverse monoid in that $\Psi(M)$ incorporates joins of arbitrary cardinality and restrictions to any open set.  Thus in general, $\mc{BI}(M)$ is a submonoid of $\Psi(M)$.  Again, our motivation for working with $\mc{BI}(M)$ rather than $\Psi(M)$ is that extending the topology to $\mc{BI}(M)$ is more intuitive.  Since we are working over a zero-dimensional space, in practice the Boolean inverse monoid does not carry any less information than the pseudogroup.
\end{rmk}

We now turn to the alternating full group $\Al(M)$ introduced by Nekrashevych (\cite{Nekra}), which is a subgroup of $\Full(M)$.  The precise definition is given in Section~\ref{ssec:alternating_defn}, but intuitively, $\Al(M)$ is the group generated by the `$3$-cycles' of $\Full(M)$, where a $k$-cycle is an element $g$ of order $k$ supported on the union of $k$ disjoint clopen sets $Y_0,\dots,Y_{k-1}$, such that $gY_i = Y_{i+1}$ (subscripts read modulo $k$).  (Analogously, Nekrashevych's symmetric full group $\Sy(M)$ is generated by the `$2$-cycles'.)  By \cite[Theorem~4.1]{Nekra}, if $X$ is perfect and $\Full(M)$ acts minimally on $X$, then $\Al(M)$ is simple.  Moreover, \cite[Theorem~5.10]{Nekra} gives a sufficient condition for $\Al(M)$ to be finitely generated.  We generalise the latter into a sufficient condition for the closure of $\Al(M)$ in $\Full(M)$ to be \emph{compactly} generated:

\begin{thm}[Theorem \ref{thm:expansive_generation}]\label{thm:intro_expansive_generation}
	Let $\mc{BI}(M)$ be as in Theorem~\ref{thm:intro_BIM}.  If $\mc{BI}(M)$ is compactly generated and every orbit on $X$ has at least 5 points, then $\ol{\Al(M)}$ is compactly generated. 
\end{thm}

Compact generation of $\mc{BI}(M)$ means that there is a compact subset $S$ of $\mc{BI}(M)$, such that every element is a finite join of products (under composition) of elements of $S$.  This notion of compact generation is closely related to the action of $M$ on $X$ being \emph{expansive} (there exists a clopen partition $\mc{P}$ of $X$ such that every pair of distinct points of $X$ can be separated by translates of $\mc{P}$).

In the case that $M=G$ is a \tdlc group acting minimally on $X$, we in fact obtain:
\begin{cor}[Corollary \ref{cor:group_expansive_iff_alternating_compactly_gen}]\label{cor:intro_expansive_generation}
	Let $G\leq \Homeo(X)$ be a locally decomposable \tdlc group acting minimally on $X$. 
	The following are equivalent:
	\begin{enumerate}[(i)]
		\item $G$ contains an open compactly generated subgroup $H$ that acts expansively on $X$ and such that $\Full(H)=\Full(G)$,
%		\item $\Al(G)$ is compactly generated,
		\item $\overline{\Al(G)}$ is compactly generated.
	\end{enumerate}		
\end{cor}

For the rest of this introduction, we focus on the setting of \tdlc groups.  Thus we have a \tdlc group $G$ acting faithfully and locally decomposably on $X$; the topology then extends uniquely to a \tdlc group topology on $\Full(G)$; and then there is a subgroup $\Al(G)$, which under certain circumstances is simple and/or compactly generated.  At this level of generality, it is not even clear if $\Al(G)$ will be closed in $\Full(G)$.  However, in the context where the equivalent conditions of Corollary~\ref{cor:group_expansive_iff_alternating_compactly_gen} are satisfied, we have a much stronger restriction:
\begin{thm}[{See Section \ref{sec:A_is_open}}]\label{thm:intro_A_is_open}
	Suppose $G\leq \Homeo(X)$ is \tdlc, locally decomposable and acts mimimally, and that $\overline{\Al(G)}$ is compactly generated.  Then the action of $\Al(G)$ on $X$ is fully compressible.  Moreover, $\Al(G)$ is abstractly simple, open and of countable index in $\Full(G)$, and we have $\Al(G) = \Der(\Full(G))$.
\end{thm}
An action of a group $G$ on $X$ is called \defbold{fully compressible} (also known as extremely proximal) if for any two proper nonempty clopen subsets $Y, Z \subset X$, there exists $g \in G$ such that $gY \subset Z$.  The fact that we have a fully compressible action, which is an application of results in \cite{CRW-Part2}, is the key to proving the other conclusions of Theorem~\ref{thm:intro_A_is_open}.  Some variations on the theme of fully compressible actions are explored in a companion article, \cite{GR_compressible}; that article also contains some of the ingredients for the proof of Theorem~\ref{thm:intro_A_is_open}.

For the rest of this introduction, let us consider the class $\ms{A}$ of \tdlc groups defined as follows: we say $G \in \ms{A}$ if there is a non-discrete \tdlc group $F$, with faithful, minimal, locally decomposable, piecewise full action on $X$, such that $\Al(F)$ is a compactly generated open subgroup of $F$, and such that $\Al(F) \le G \le F$.  (For all groups $G \in \ms{A}$, the action of $G$ on $X$ is uniquely specified by the group structure and $G$ carries a unique Polish group topology, which is locally compact; see Theorem~\ref{thm:A_properties}.)
We recall also the class $\ms{S}_{\ldlat}$ of \tdlc groups that are non-discrete, compactly generated, simple and admit a faithful locally decomposable action on $X$.
Thus in the conclusion of Theorem~\ref{thm:intro_A_is_open}, $\Full(G)$ belongs to the class $\ms{A}$, while $\Al(G)$ belongs to $\ms{A} \cap \ms{S}_{\ldlat}$.

We can prove some special properties of the groups in $\ms{A}$, beyond what is known (or even true) more generally in $\ms{S}_{\ldlat}$; we highlight two results here.  The first is a straightforward application of \cite[Theorem~1.3]{CLBMB}, showing that we are in a very different dynamical situation from that of the Juschenko--Monod theorem.

\begin{thm}[{Theorem~\ref{thm:CLBMB}}]\label{thm:intro_CLBMB}
	Let $G \in \ms{A}$.  Then $G$ acts freely on its Furstenberg boundary.  Equivalently, every $G$-conjugacy class of closed relatively amenable subgroups of $G$ contains the trivial group in its $\mathrm{Sub}(G)$-closure.
\end{thm}

The second is a strong restriction of actions of groups $G \in \ms{A}$ on hyperbolic spaces, which builds on recent work of Balasubramanya, Fournier-Facio and Genevois \cite{BFG}.  This obstacle to actions on hyperbolic spaces is notable given that many other examples of groups in $\ms{S}_{\ldlat}$ (such as the subgroup of index $2$ of $\Aut(T)$) have been constructed as groups acting geometrically densely on a tree.

\begin{thm}[See~\ref{sec:hyperbolic}]\label{thm:intro_hyperbolic}
Let $G \in \ms{A}$.  Let $Y$ be a hyperbolic space, and suppose $G$ acts on $Y$ by isometries.  Then the action of $G$ is not of general type: that is, any two loxodromic elements have a common limit point on the visual boundary.  If moreover $G$ is perfect (for instance, taking $G = \Der(H)$ for any $H \in \ms{A}$), then the action of $G$ has no loxodromic elements; in particular, $G$ is one-ended.
\end{thm}

On the other hand, if we ask which \emph{local} isomorphism classes of \tdlc group include groups in $\ms{A}$, we find they can be characterised in ways that do not directly invoke the piecewise full property.  (Recall that two \tdlc groups $G$ and $H$ are \defbold{locally isomorphic} if there exist open subgroups $U \le G$ and $V \le H$ such that $U \cong V$.)

A \tdlc group $G$ is \defbold{expansive} if there is some identity neighbourhood $U$ in $G$ such that $\bigcap_{g \in G}gUg\inv = \triv$; we say $G$ is \defbold{regionally expansive} if some compactly generated open subgroup of $G$ is expansive.  A \defbold{robustly monolithic} group is a \tdlc group $G$ with a closed normal subgroup $M$, such that $M$ is non-discrete, regionally expansive and topologically simple, with trivial centraliser in $G$.  For example, it is clear from the results so far that every $G \in \ms{A}$ is robustly monolithic.  
The class $\ms{R}$ of robustly monolithic groups was introduced in \cite{CRW-DenseLC}, where it was shown that groups $G \in \ms{R}$ have some structural properties generalising those of groups in $\ms{S}$.  It is currently unknown if every group in $\ms{R}$ is locally isomorphic to a group in $\ms{S}$.  However, if we write $\ms{R}_{\ldlat}$ for the groups in $\ms{R}$ with faithful locally decomposable action, we find the following.

\begin{thm}[{See Theorem~\ref{thm:locally_in_A}}]\label{thm:intro_locally_in_A}
Let $G$ be a compactly generated \tdlc group, and suppose that no non-trivial element of $G$ has open centraliser.  Let $Y$ be the Stone space of $\ldlat(G)$ (see Definition~\ref{defn:ldlat}).  Then the following are equivalent:
\begin{enumerate}[(i)]
\item $G$ is an open subgroup of a group in $\ms{A}$;
\item $G$ is locally isomorphic to a group in $\ms{S}_{\ldlat}$;
\item $G$ is locally isomorphic to a group in $\ms{R}_{\ldlat}$;
\item $G$ acts faithfully on $Y$, and there are compactly generated \tdlc groups $A$ and $B$ locally isomorphic to $G$, such that $A$ is expansive and has no non-trivial discrete normal subgroup, while $B$ admits a faithful minimal micro-supported action on a compact zero-dimensional space.
\end{enumerate}
\end{thm}

\begin{cor}[See Corollary~\ref{cor:direct_powers}]
Let $G$ be a compactly generated group in $\ms{R}$, and write $G^n$ for the direct product of $1 < n < \infty$ copies of $G$.

If $G$ has a faithful locally decomposable action, then $G \wr \Sym(n)$ is an open subgroup of a group in $\ms{A}$, and hence $G^n$ is locally isomorphic to a group in $\ms{S}$.  Otherwise, $G^n$ is not locally isomorphic to any group in $\ms{R}$.
\end{cor}

The local isomorphism class of a \tdlc group is witnessed by a profinite group.  An accessible source of candidates for such profinite groups, for which there is significant existing literature to draw on, is the class of profinite branch groups: these are compact groups acting continuously on rooted locally finite trees, such that the action on the boundary is transitive and locally decomposable.  For these groups, some of the conditions of Theorem~\ref{thm:intro_locally_in_A}(iv) are automatically satisfied, so we have the following.

\begin{cor}
Let $U$ be a profinite branch group.  Then the following are equivalent:
\begin{enumerate}[(i)]
\item $U$ is an open subgroup of a group in $\ms{A}$;
\item $U$ is locally isomorphic to a group in $\ms{S}_{\ldlat}$;
\item $U$ is locally isomorphic to a group in $\ms{R}_{\ldlat}$;
\item There is a compactly generated \tdlc group $A$ locally isomorphic to $U$, such that $A$ is expansive and has no non-trivial discrete normal subgroup.
\end{enumerate}
\end{cor}

We will discuss in Section~\ref{sec:examples} some examples of piecewise full groups arising from group actions in trees, including starting from a profinite branch group.  These examples do not stray far from the prior literature, and are mainly given to illustrate the concepts of the present article in a familiar setting.  The authors plan to construct more interesting new examples in a future article.

\subsection{Structure of the article}

In Section~\ref{sec:groups} we establish the importance of local decomposability for piecewise full groups, as summarised in Theorem~\ref{thm:intro_localdec}.  This section is not formally required for the rest of the article, as analogous results will be proved in greater generality in Section~\ref{sec:inverse_monoids}.  However, we have included Section~\ref{sec:groups} in order to give a purely group-theoretic explanation of why Polish piecewise full groups must be locally decomposable.

In Section~\ref{sec:inverse_monoids}, we construct the topological inverse monoid $\mc{BI}(M;X)$ described in Theorem~\ref{thm:intro_BIM} and establish some basic properties of it; this construction then allows us, in Section~\ref{sec:compact_gen}, to adapt the proof of Nekrashevych's finite generation theorem ( \cite[Theorem~5.10]{Nekra}) to prove our Theorem~\ref{thm:intro_expansive_generation} and Corollary~\ref{cor:intro_expansive_generation}.

In Section~\ref{sec:A_is_open} we collect together some preliminaries on compressible actions (including from \cite{GR_compressible}) and their role in the structure of \tdlc groups, leading to the conclusion summarised in Theorem~\ref{thm:intro_A_is_open}.

Section~\ref{sec:examples} is an examples section to illustrate the theory up to this point, with a focus on groups with a compact open branch subgroup.  We also recall the group $\ms{L}(G)$ of germs of automorphisms of a \tdlc group $G$, which will be useful in Section~\ref{sec:A and S}.

Section~\ref{sec:alternatable} is a short section exploring the structure of automorphisms of $\Full(G;X)$ and $\Al(G;X)$ when $G$ acts minimally on a compact zero-dimensional space $X$.  The group topology plays no role in this section and the novel results are not critical for later sections.

In Section~\ref{sec: alternatable almost simple} we introduce the class $\ms{A}$ of simple-by-abelian \tdlc groups and prove some properties common to all groups in this class, including Theorems~\ref{thm:intro_CLBMB} and~\ref{thm:intro_hyperbolic}.

In Section~\ref{sec:A and S} we consider which local isomorphism classes of \tdlc groups include groups in $\ms{A}$, as summarised above by Theorem~\ref{thm:intro_locally_in_A} and its corollaries.

The article finishes with a list of open questions (Section~\ref{sec:questions}).

%\subsection{Acknowledgements}
%
%The research project behind this article started while both authors were employed at the University of Newcastle, supported by the Australian Research Council grant FL170100032 (CI: George Willis).  We thank our erstwhile colleagues at Newcastle for productive discussions, especially David Robertson with whom we discussed the project quite extensively and gave significant insights on some potential sources of examples.

\section{Topologies on piecewise full groups of group actions}\label{sec:groups}

		In this section, we will show that piecewise full Polish groups are locally decomposable, and that conversely, given a faithful locally decomposable action of a Polish group, the topology extends to the piecewise full group. 
		 These results will not be formally necessary for the rest of the article, as we will come to analogous conclusions in greater generality in Section~\ref{sec:inverse_monoids},  for locally decomposable inverse monoids, in preparation for Section~\ref{sec:compact_gen}. 
		  However, since Section~\ref{sec:inverse_monoids} is more technical and the primary focus of this article is group theory, it may be useful for the reader to first see the direct group-theoretic arguments.

\subsection{Extending the topology from a subgroup}\label{sec:group_extend_topology}
	
	The main goal of this section is to find a suitable group topology for the piecewise full group of a group action.  We begin with some generalities on when it is possible to extend the topology from a subgroup.  Specifically, given a group $G$ and $H \le G$ such that $H$ is equipped with a group topology $\tau_H$, we say a group topology $\tau_G$ on $G$ is an \defbold{extension} of $\tau_H$ if $H \in \tau_G$ and $\tau_H$ is the restriction of $\tau_G$ to $H$.  It is easy to see that the extension is unique if it exists: $\tau_G$ must be generated by translates of open subsets of $H$.  So from now on, we will simply say that $\tau_H$ \defbold{extends to $G$} to mean that the extension exists as a group topology of $G$.
	
	As a guiding example, consider Neretin's group $\AAut(T)$ of almost automorphisms of a regular tree $T$. 
	It is well-known (see e.g. \cite{lc:book_Neretin}) that $\AAut(T)$ admits a non-discrete locally compact group topology, which is finer than that inherited from the compact-open topology on $\Homeo(\partial T)$. 
	This topology is `induced up' from the topology on $\Aut(T_r)$, the group of tree automorphisms that fix a vertex $r$.
	The group $\Aut(T_r)$ is compact; its basic identity neighbourhoods are the stabilisers $\{A_n: n\in \mathbb{N}\}$ of balls of radius $n$ centred at $r$.
	This is in fact the subspace topology from the compact-open topology on $\Homeo(\partial T)$.
	The reason why this topology on $\Aut(T_r)$ can be extended to $\AAut(T)$ is because the latter group commensurates the former: more precisely, for every almost automorphism $h\in\AAut(T)$ the conjugate $h\Aut(T_r)h^{-1}$ contains some $A_n$.
	The following general fact implies that the open subgroups of the commensurated subgroup $\Aut(T_r)$ yield a neighbourhood basis for a group topology on $\AAut(T)$, which extends that of $\Aut(T_r)$.
	
	\begin{prop}{\cite[III, p.3-4, Prop. 1]{Bourbaki_top14_07}}\label{prop:bourbaki_gentop}
		Let $\mc{B}\subseteq$ be a filter base of a group $G$ (a non-empty set of non-empty subsets of $G$ such that the intersection of any two members of $\mc{B}$ contains a member of $\mc{B}$). 
		Suppose that $\mc{B}$ satisfies:
		\begin{enumerate}
			\item for every $U\in\mc{B}$ there is $V\in \mc{B}$ such that $V.V\subseteq U$;
			\item for every $U\in\mc{B}$ there is $V\in \mc{B}$ such that $V\inv\subseteq U$;
			\item for every $g\in G, U\in\mc{B}$ there is $V\in\mc{B}$ such that $V\subseteq aUa\inv$. 
		\end{enumerate}
		Then there is a unique group topology on $G$ for which $\mc{B}$ is a base of neighbourhoods of the identity. 
	\end{prop}
	
	Inspired by condition (3) of Proposition~\ref{prop:bourbaki_gentop}, we say that a group $G$ \defbold{preserves the opens} of a topological group $H$ if for every $g\in G$, and every identity neighbourhood $O$ in $H$ there is an identity neighbourhood $U$ in $H$ such that $U \subseteq gOg^{-1} \cap g^{-1}Og$. 
	If $H\leq G$, the \defbold{open preserver} of $H$ in $G$ is the set $\mathrm{OP}_G(H)$ of all elements of $G$ that preserve the opens of $H$.  Given Proposition~\ref{prop:bourbaki_gentop}, the following is then easily verified.
	
	\begin{cor}\label{cor:bourbaki_gentop}
		Let $G$ be a group, let $H$ be a subgroup of $G$ equipped with a group topology $\tau_H$, and let $K = \mathrm{OP}_G(H,\tau_H)$.  Then $K$ is a subgroup of $G$ and $\tau_H$ extends to $K$.
	\end{cor}
	
	The open preserver has some familiar alternative interpretations for certain families of groups with sufficiently many open subgroups.  Before making a precise statement, we recall some more topological group theory.
	
	A \defbold{Polish} group is a topological group that is separable and completely metrisable.  For this article we will mainly be interested in totally disconnected locally compact Polish groups;  note that for a locally compact group, being Polish is the same as being second-countable (e.g. \cite[Theorem 5.3]{Kechris}). 

	By the Baire category theorem, Polish groups have the \defbold{countable index property}: every closed subgroup of countable index is open. Conversely, since Polish groups are separable, every open subgroup has countable index.  The groups we are interested in will turn out to have at most one Polish group topology, using the following well-known results from the literature.
	
	\begin{thm}[{\cite[Theorem~3.3]{Mackey}}]\label{thm:Mackey}
		Let $S$ be a standard Borel space and let $\mc{C}$ be a countable collection of Borel subsets of $S$ that separate points.  Then $\mc{C}$ generates the Borel structure of $S$.
	\end{thm}
	
	The following is a standard consequence of Pettis' theorem \cite{Pettis}.
	
	\begin{thm}[{See \cite[Proposition~8.22, Theorem~9.10]{Kechris}}]\label{thm:Pettis}
		Let $G$ and $H$ be Polish groups and let $\varphi: G \rightarrow H$ be a homomorphism.  If $\varphi$ is Borel measurable, then it is continuous.
	\end{thm}
	
	In particular, if two Polish group topologies $\tau_1,\tau_2$ on a group $G$ generate the same Borel structure, then the identity map is continuous with continuous inverse, so $\tau_1=\tau_2$.  Combining with Theorem~\ref{thm:Mackey} yields the following.
	
	\begin{cor}\label{cor:Polish_determined}
		Let $G$ be a group and let $\mc{C}$ be a countable collection of subsets of $S$ that separate points.  Then there is at most one Polish group topology on $G$ such that the elements of $\mc{C}$ are all Borel sets.
	\end{cor}
	
	We can now state some alternative interpretations of `preserving the opens' for the groups we will be interested in.
	
	\begin{lem}\label{lem:open_subgroup_extend}
		Let $G$ be a group, let $H$ be a subgroup of $G$ and let $g \in G$.  Suppose that $H$ is equipped with a Polish or locally compact group topology $\tau_H$ in which open subgroups separate points.  Then $g$ preserves the opens of $H$ if and only if, for every open subgroup $K$ of $H$, the intersections $H \cap gKg\inv$ and $H \cap g\inv Kg$ belong to $\tau_H$.
	\end{lem}
	
	\begin{proof}
		Let $K$ be an open subgroup of $H$.  Note that $H \cap gKg\inv$ is a subgroup of $H$, so it is an identity neighbourhood if and only if it is open.  Thus if $g \in \mathrm{OP}_G(H)$ then $H \cap gKg\inv, H \cap g\inv Kg \in \tau_H$.  For the rest of the proof we may suppose $H \cap gKg\inv, H \cap g\inv Kg \in \tau_H$ for all $K \le H$ open.
		
		If $\tau_H$ is locally compact, then since open subgroups separate points, $H$ is a \tdlc group.  Thus in fact open subgroups form a base of neighbourhoods of the identity, by Van Dantzig's theorem, and we conclude that $g$ preserves the opens of $H$.
		
		Suppose instead that $\tau_H$ is Polish.  Since $H$ is a Lindel\"{o}f space (e.g. \cite[Theorem~16.9]{Willard}), in fact countably many cosets of open subgroups suffice to separate the points of $H$.  By our hypothesis, the subgroups $g^{-1}Hg \cap H$ and $H \cap gHg^{-1}$ are both open in $H$.  By Theorem~\ref{thm:Mackey}, conjugation by $g$ induces a Borel isomorphism $c_g$ from $g^{-1}Hg \cap H$ to $H \cap gHg^{-1}$; in fact by Theorem~\ref{thm:Pettis}, $c_g$ is a homeomorphism.  Again we conclude that $g$ preserves the opens of $H$.
	\end{proof}
	
	\begin{defn}\label{def:relative_commensurator}
	Given a group $G$ and a subgroup $H$, the \defbold{commensurator} $\Comm_G(H)$ is defined as the set of $g \in G$ such that $H \cap gHg\inv$ has finite index in both $H$ and $gHg\inv$.
	\end{defn}	

	It is easily verified that $\Comm_G(H)$ is always a subgroup of $G$ containing $H$.

	\begin{lem}\label{lem:compact_open_subgroup_extend}
		Let $G$ be a topological group and let $H$ be a compact subgroup of $G$.  Then $\mathrm{OP}_G(H) = \Comm_G(H)$.
	\end{lem}
	
	\begin{proof}
		Let $g \in G$.  Since $H$ is a compact group, any subgroup of nonempty interior has finite index.  In particular, if $g$ preserves the opens of $H$, then the intersection $H \cap gHg\inv$, being an identity neighbourhood in $H$, must have finite index, and similarly $H \cap gHg\inv$ has finite index in $gHg\inv$, so $g \in \Comm_G(H)$.
		
		Conversely, suppose $g \in \Comm_G(H)$.  Then conjugation by $g$ yields a topological group isomorphism from $H_1 = g\inv H g \cap H$ to $H_2 = H \cap gHg\inv$.  Since $H$ is compact, it is closed in $G$, so $H_1$ and $H_2$ are also closed; by assumption, both $H_1$ and $H_2$ have finite index in $H$, so in fact $H_1$ and $H_2$ are open subgroups of $H$.  Thus $g$ preserves the opens of $H$ by Lemma~\ref{lem:open_subgroup_extend}.
	\end{proof}

	\subsection{Local decomposability and piecewise full groups}
	
	We now focus on the context of actions on zero-dimensional compact spaces and introduce some notation.

\textbf{Notation.}
Let $G$ be a group acting by homeomorphisms on a topological space $X$ and let $Y$ be a subspace of $X$. 
The \defbold{setwise stabiliser} in $G$ of $Y$ is denoted by $G_Y=\{g\in G \mid g(Y)=Y\}$.
The \defbold{rigid stabiliser} $\rist_G(Y)$ is the group of all elements of $G$ that fix $X \smallsetminus Y$ pointwise. 
We say $G$ is \defbold{micro-supported} if $\rist_G(Y)$ is non-trivial for every nonempty open set $Y$ in $X$.
Note that if $Y,Z$ are disjoint subsets of $X$ then $\rist_G(Y)\cap \rist_G(Z)=\{1\}$. 
A set $\mc{S}$ of subsets of $X$ is \defbold{pairwise disjoint} if $Y \cap Z = \emptyset$ for every distinct pair of elements $Y,Z \in \mc{S}$. 
Write $Y^{\perp}$ for the complement of $Y$ in $X$ (the ambient space will be clear in context).
Given a group $G$ and $g,h\in G$, we define $[g,h] = ghg\inv h\inv$ and write $\Der(G)$ for the commutator subgroup, that is, $\Der(G) = \grp{[g,h] \mid g,h \in G}$.

\

The reason why $\AAut(T)$ commensurates $\Aut(T_r)$ is that the open subgroups $A_n$ of $\Aut(T_r)$ \emph{decompose} as  $A_n=\displaystyle{\prod_{d(r,v)=n} \rist_{\Aut(T_r)}(\partial T_{(r,v)})}$ where $T_{(r,v)}$ is the subtree spanned by vertices $w$ such that the path from $r$ to $w$ passes through $v$. 
This means that conjugating one of these stabilisers by an element of $\AAut T$, produces another direct product of the same form:
For every $A_n$ and $g\in\AAut(T)$, we can suppose without loss of generality that $g$ is represented by a forest isomorphism $T\setminus B\rightarrow T\setminus C$, where $B, C\subset T$ are finite subtrees, $B$ contains the ball of radius $n$ around $r$ and $C$ is contained in some ball of radius $m$ around $r$.
Then 
\begin{multline*}
	g^{-1}A_mg=g^{-1}\left(\prod_{d(r,v)=m}\rist_{\Aut(T_r)}(\partial T_{(r,v)})\right)g \\
	=\prod_{d(r,v)=m}\rist_{\Aut(T_r)}(g^{-1}(\partial T_{(r,v)})) 
	\leq \prod_{u\in B}\rist_{\Aut(T_r)}(\partial T_{(r,u)}) \leq A_n.
\end{multline*}

The boundary of the tree is the Cantor space, which is the unique (up to homeomorphism) second countable, compact, zero-dimensional space without isolated points. 
We say that a topological space is \defbold{zero-dimensional} if it separates points\footnote{Some authors allow a zero-dimensional space to fail to separate points, but it is convenient for our purposes to exclude this possibility.} (is $T_0$) and has a base of its topology consisting of clopen sets. 
This implies that it is Hausdorff.
Generalising from $\partial T$ to a compact zero-dimensional space, 
the structure of the open subgroups $A_n$ of $\AAut(T)$ leads to the following:

\begin{defn}\label{defn:locally_decomposable_group}
	Let $X$ be a compact zero-dimensional space and $G$ a topological group of homeomorphisms of $X$. 
	The action of $G$ on $X$ is \defbold{locally decomposable} if for every partition $\mc{P}$ of $X$ into clopen subsets, the subgroup 
		\[
		\grp{\rist_G(U) \mid U \in \mc{P}} = \prod_{U\in\mc{P}}\rist_G(U)
		\]
		is open in $G$ and carries the product topology.
\end{defn}

Given any clopen subset $U$ of $X$, we have $\rist_G(U)\times \rist_G(X\setminus U)\leq G_U$. 
So if $G$ is locally decomposable, then every $G_U$ is open and so the action  $G\times X\rightarrow X$ is continuous with respect to the product topology on $G\times X$.
In particular, the topology on $G$ refines the compact-open topology on  $\Homeo(X)$, the coarsest one that makes the action of $\Homeo(X)$ on the compact Hausdorff space $X$ jointly continuous.

We note that the property of being locally decomposable is indeed local, that is, it passes to and from open identity neighbourhoods.
	
	\begin{lem}\label{lem:locally_locally_decomposable}
		Let $X$ be a compact zero-dimensional space and $G$ a topological group of homeomorphisms of $X$.  Let $O$ be an open identity neighbourhood in $G$.  Then the action of $G$ is locally decomposable if and only if for every clopen partition $\mc{P} = \{U_1,\dots,U_n\}$ of $X$, the set
		\[
		\prod^n_{i=1}(\rist_G(U_i) \cap O) := \{g_1g_2 \dots g_n \mid g_i \in \rist_G(U_i) \cap O\}
		\]
		is open and carries the product topology.
	\end{lem}
	
	\begin{proof}
		Fix a clopen partition $\mc{P} = \{U_1,\dots,U_n\}$ of $X$.
		
		If the action of $G$ is locally decomposable, then certainly the product of the sets $(\rist_G(U_i) \cap O)$ is open and carries the product topology.
		
		Conversely, suppose $\prod^n_{i=1}(\rist_G(U_i) \cap O)$ is open and carries the product topology, and let $H = \prod^n_{i=1}\rist_G(U_i)$ be the direct product of the groups $\rist_G(U_i)$, equipped with the product topology.  Then the natural group homomorphism $\phi: H \rightarrow G$ induced by multiplication in $G$ is injective, and has the property that it restricts to a homeomorphism between neighbourhoods of the identity.  By continuity of multiplication, it follows that $\phi$ is an open embedding, in other words $\phi(H)$ is open in $G$ and carries the product topology.
	\end{proof}

The discussion of $\AAut(T)$ makes it plausible that the locally decomposable condition is sufficient to extend the group topology to the piecewise full group.

We start by verifying that rigid stabilisers of full groups are indeed big enough subgroups.

\begin{lem}\label{lem:full_implies_stabs_are_rsts}
	If $G\leq \Homeo(X)$ is piecewise full then, for any clopen partition $\mc{P}$ of $X$, the part-wise stabiliser $K:=\bigcap_{Y\in\mc{P}} G_{Y}$ of the partition  decomposes (as a group) as the direct product $K=\prod_{Y \in \mc{P}}\rist_G(Y)$. 
\end{lem}

\begin{proof}
	One inclusion is clear. 
	To see the other one, let $g\in \bigcap_{Y\in\mc{P}} G_{Y}$. 
	For each $Y \in \mc{P}$ the map $g_Y:X\rightarrow X$ defined by $\rest{g_Y}{Y}=\rest{g}{Y}, \rest{g_Y}{X\setminus Y}=\rest{\mathrm{id}}{X\setminus Y}$ is a well-defined homeomorphism of $X$ that is moreover in $\Full(G)=G$.  Indeed, $g_Y$ is in $\rist_G(Y)$.
	The product of the $g_Y$ as $Y$ ranges through $\mc{P}$ is $g$, showing that $g \in \prod_{Y \in \mc{P}}\rist_G(Y)$.
\end{proof}

We now check that being locally decomposable is indeed a sufficient condition to extend the topology to the piecewise full group.

\begin{prop}\label{prop:full_contains_intersection_rists}
	Let $X$ be a zero-dimensional compact space and $H,G\leq\Homeo(X)$, where $H$ is equipped with a group topology $\tau_H$.  Let $g \in \Full(G)$ and $1 \in O \in \tau_H$.  Then there is a clopen partition $\mc{P}=\{U_1,\dots,U_n\}$ of $X$, an identity neighbourhood $O'$ in $H$ and elements $g_1,\dots, g_n\in G$ such that $\bigcap^n_{i=1}g_iO^*g\inv_i \subseteq gOg\inv$, where
		\[
		O^* := (\rist_H(U_1) \cap O')(\rist_H(U_2) \cap O') \dots (\rist_H(U_n) \cap O').
		\]

	In particular:
	\begin{enumerate}[(i)]
		\item If $G$ preserves the opens of $H$ and $H$ is locally decomposable, then $\Full(G)$ preserves the opens of $H$.
		\item If $G$ admits a group topology such that $H$ is open and locally decomposable, then it can be uniquely extended to a group topology on $\Full(G)$, which is locally decomposable.
		\item If $H$ is locally decomposable, the open preserver $\mathrm{OP}_{\Homeo(X)}(H)$ of $H$ in $\Homeo(X)$ is piecewise full and locally decomposable.
	\end{enumerate}
\end{prop}

\begin{proof}
	Let $g\in \Full(G)$. By definition, there is a clopen partition $X=\sqcup_{i=1}^n U_i$ and $g_1,\dots,g_n\in G$ such that 
	$\rest{g}{U_i} =\rest{g_i}{U_i}$ for $i=1,\dots,n$.
	
	By continuity of multiplication we can take $1 \in O' \in \tau_H$ such that $(O')^n \subseteq O$.  To see that $\bigcap^n_{i=1}g_iO^*g\inv_i \subseteq gOg\inv$, take $k$ in the intersection. Then $g\inv_i kg_i \in O^*$ for $i=1,\dots, n$; say $g\inv_i kg_i=k_i=k_{i1} \dots k_{in}$, where $k_{ij} \in \rist_H(U_i) \cap O'$.
		
		For $1 \le i \le n$ and $x \in U_i$, we have
		\[
		g\inv kg(x) = g\inv (g_ik_ig_i^{-1})g(x)=g\inv g_ik_ig_i\inv g_i(x)=g\inv g_ik_{ii}(x)=g_i\inv g_ik_{ii}(x)=k_{ii}(x).
		\]
		Thus $g\inv k g$ acts on $U_i$ as an element of $\rist_H(U_i) \cap O'$ and hence $g\inv k g \in O^*$; the choice of $O'$ ensures that $O^* \subseteq O$.  Hence $k \in gOg\inv$ as required.	
		
		If $H$ is locally decomposable and $G$ preserves its opens, then $O^*$ is an identity neighbourhood in $H$ and therefore so is the intersection $\bigcap^n_{i=1}g_i O^* g_i\inv$, showing that $gO^*g^{-1}$ contains an open subgroup of $H$.  Thus $\Full(G)$ preserves the opens of $H$.

	The second item follows from the first one and Corollary~\ref{cor:bourbaki_gentop}, because $G$ preserves its own opens. 
	
	The last item follows from the first, taking $G=\mathrm{OP}_{\Homeo(X)}(H)$, and the same argument as the second item. 
\end{proof}

\subsection{Uniqueness of the topology}
	
	As we have seen, local decomposability is sufficient to extend the topology from $G$ to $\Full(G;X)$.
	It turns out that it is also necessary, at least when assuming that the action is on the Cantor space and that the group topology is also Polish (as is the case for $\AAut(T)$; the countably many cosets of the $\{A_n: n\in\mathbb{N} \}$ form a base for the topology).

\begin{lem}\label{lem:microsupported:centralisers}
	Let $X$ be a  zero-dimensional compact space
	and $G\leq \Homeo(X)$ have a micro-supported action. 
	\begin{enumerate}[(i)]
		\item If $Y$ is a closed subset of $X$ then $\rist_G(Y) = \CC_G(\rist_G(Y^{\perp}))$. 
		\item If $Y$ is regular closed, then $G_Y = \N_G(\rist_G(Y))$, so $G_Y = \N_G(\CC_G(\rist_G(Y^{\perp})))$.
		\item If $G$ carries a Hausdorff group topology and $Y$ is regular closed, then $G_Y$ is a closed subgroup.
	\end{enumerate}
\end{lem}

\begin{proof}
	Clearly, $\rist_G(Y) \le \CC_G(\rist_G(Y^{\perp}))$.
	Conversely, suppose $g \in G$ is such that $gx \neq x$ for some $x \in Y^{\perp}$.  
	Then there is a compact open neighbourhood $U$ of $x$ such that $U$ and $g(U)$ are disjoint open subsets of $Y^{\perp}$. 
	As $\rist_G(U)$ is non-trivial,  $g\rist_G(U)g\inv \neq \rist_G(U)$; since $\rist_G(U) \le \rist_G(Y^{\perp})$, it follows that $g$ does not centralise $\rist_G(Y^{\perp})$. 
	Thus $\rist_G(Y) = \CC_G(\rist_G(Y^{\perp}))$.
	
	Suppose that $Y$ is regular closed, that is, $Y$ is closed and has dense interior.  Certainly $G_Y$ normalises $\rist_G(Y)$.  Conversely, suppose $g \in G$ normalises $\rist_G(Y)$, let $x$ be an interior point of $Y$, and let $U$ be a compact open neighbourhood of $x$ contained in $Y$.  Then $\rist_G(U) \le \rist_G(Y)$, so also $g\rist_G(U)g\inv \le \rist_G(Y)$.
	But then $g\rist_G(U)g\inv = \rist_G(gU)$, so $\rist_G(gU) \le \rist_G(Y)$, in other words, the support of $\rist_G(gU)$ is contained in $Y$. 
	Since $G$ is micro-supported on $X$, it follows that $Y$ contains a dense subset of $gU$; since $Y$ and $gU$ are closed, in fact $gU \subseteq Y$, and then since $gU$ is open, in fact $gU$ lies in the interior of $Y$. 
	Thus $\N_G(\rist_G(Y))$ preserves the interior of $Y$; since $Y = \ol{\mathrm{int}(Y)}$, it follows that in fact $\N_G(\rist_G(Y))$ preserves $Y$. 
	Thus
	\[
	G_Y = \N_G(\rist_G(Y)) = \N_G(\CC_G(\rist_G(Y^{\perp}))). 
	\]
	
	In any Hausdorff group topology, centralisers are closed, and normalisers of closed subgroups are also closed. 
	It follows that $G_Y=\N_G(\CC_G(\rist_G(Y^{\perp})))$ is a closed subgroup of $G$.
\end{proof}

The above lemma does not account for all piecewise full actions, since a piecewise full action can fail to be micro-supported.
For instance, if the action of $G$ on $X$ fixes $Y$ pointwise then $\rist_G(Y)=\rist_{\Full(G)}(Y)=1$. 
To avoid this situation, we impose an extra condition. 
The action of a group $G$ by homeomorphisms on a space $X$ is \defbold{non-degenerate} if for every non-empty open subset $Y$ of $X$ there are $y\in Y$ and $g\in G$ such that $g(y)\in Y\setminus\{y\}$.

If $X$ has isolated points then it cannot support a non-degenerate action. 
So we must consider perfect spaces. 

\begin{lem}\label{lem:microsupported}
	Let $X$ be a perfect, zero-dimensional compact space and $G\leq \Homeo(X)$. 
	The following are equivalent:
	\begin{enumerate}[(i)]
		\item $G$ is non-degenerate;
		\item $\Full(G)$ is non-degenerate;
		\item $\Full(G)$ is micro-supported.
	\end{enumerate}
\end{lem}

\begin{proof}
	Since $F:=\Full(G)$ has the same orbits as $G$, we see that (i) and (ii) are equivalent.
	
	Suppose that $F$ is non-degenerate and let $Y$ be an open subset of $X$. 
	Then there exists $y\in Y$ and $f\in F$ such that $f(y)\in Y$ and $f(y)\neq y$. 
	Since $X$ is Hausdorff and $Y$ is open, there is an open neighbourhood $Z$ of $y$ such that $Z$ and $f(Z)$ are disjoint open subsets of $Y$. 

	Since $F$ is piecewise full, it contains the non-trivial homeomorphism $$g(x)=\begin{cases} f(x), & x\in Z,\\
		f^{-1}(x), & x\in f(Z),\\
		x, &\text{ else}
	\end{cases}$$
	which lies  in $\rist_F(Y)$.

	Finally, if $\rist_F(Y)$ is non-trivial for every open subset $Y$ of $X$, then the action of $F$ is non-degenerate. 
\end{proof}

We note next that if $G$ is a non-degenerate piecewise full group, then any Polish group topology of $G$ must make the action locally decomposable.

\begin{prop}\label{prop:Polishfull_implies_locdec}
	Let $X$ be a perfect, zero-dimensional compact space and $G\leq\Homeo(X)$ a non-degenerate piecewise full group equipped with a Polish group topology $\tau$. 
	Then the action of $G$ is jointly continuous and locally decomposable with respect to $\tau$.
\end{prop}
\begin{proof}
	Fix a clopen partition $\mc{P}$ of $X$.  
	We suppose $G$ is equipped with a Polish group topology $\tau$. 
	By Lemma~\ref{lem:microsupported:centralisers} the subgroup $\displaystyle{G_{\mc{P}}=\bigcap_{U\in\mc{P}} G_{U}}$ is open in $G$. 
	Ignoring the group topology,  Lemma~\ref{lem:full_implies_stabs_are_rsts} says that the natural homomorphism 
	$$\varphi: \displaystyle{\prod_{U\in\mc{P}}\rist_G(U)} \rightarrow G_{\mc{P}}$$ 
	is bijective.
	By Lemmas~\ref{lem:microsupported} and \ref{lem:microsupported:centralisers} the action is micro-supported, so $\rist_{G}(U)$ is closed in $G$ and therefore Polish, for every clopen subset $U$ of $X$. 
	This implies that $ \displaystyle{\prod_{U\in\mc{P}}\rist_G(U)}$  is a Polish group when given the product topology. 
	Now, multiplication in $G$ is continuous, so $\varphi$ is continuous, hence a homeomorphism by Theorem~\ref{thm:Pettis}.
	So  $\displaystyle{\prod_{U\in\mc{P}}\rist_G(U)}$ (with the product topology) is an open subgroup of $G$. 
\end{proof}

We now observe the uniqueness of Polish group topology under some mild assumptions.

\begin{cor}\label{cor:unique_Polish}
	Let $G$ be a group and let $X$ be a perfect, zero-dimensional compact space.  Suppose one of the following:
	\begin{enumerate}[(i)]
		\item $G$ admits a faithful micro-supported action on $X$ and $X$ is the Cantor space;
		\item $G$ admits a faithful non-degenerate piecewise full action on $X$.
	\end{enumerate}
	Then $G$ has at most one Polish group topology, and if this topology exists then the action of $G$ on $X$ in either case is continuous.
\end{cor}

\begin{proof}
	We may suppose $G$ admits some Polish group topology $\tau$. 
	 Let $Y$ be a clopen subset of $X$.  
	 In case (i), $G_Y$ is closed by Lemma~\ref{lem:microsupported:centralisers} and has countable index since there are only countably many clopen subsets of $X$, so $G_Y \in \tau$ by the countable index property. 
	  In case (ii), the action is locally decomposable by Proposition~\ref{prop:Polishfull_implies_locdec}, so again $G_Y \in \tau$.
		
		Since the action is faithful, the cosets of the subgroups $G_Y$ separate points in $G$; since $G$ is Lindel\"{o}f, the cosets of countably many open subgroups suffice.  Thus there is a countable set $\mc{S}$ of clopens of $X$ such that $\mc{C} = \{gG_Y \mid Y \in \mc{S}, g \in G\}$ separates points of $G$.  The set $\mc{C}$ is again countable, and by the previous paragraph, the elements of $\mc{C}$ are open in any Polish group topology on $G$, ensuring the action is continuous.  By Corollary~\ref{cor:Polish_determined}, it follows that the Polish topology of $G$ is unique.
\end{proof}

Consider now the special case that $H\leq \Homeo(X)$ is compact and acts continuously on $X$.
This means that its topology is finer than the subspace topology in $\Homeo(X)$. 
Since  no compact Hausdorff topologies can properly refine each other, 
the topology of $H$ is actually the subspace topology of $\Homeo(X)$. 
In this case, $H$ is totally disconnected and compact, that is, profinite.
So every open subgroup of $H$ has finite index and every closed subgroup of finite index is open. 
In this case, we can apply Lemma~\ref{lem:compact_open_subgroup_extend} to obtain:

\begin{cor}\label{cor:piecewise_full_extension}
	Let $X$ be a compact zero-dimensional space and $H\leq \Homeo(X)$ be a compact group. 
	If $H$ is locally decomposable, then: 
	\begin{enumerate}[(i)]
		\item $C:= \Comm_{\Homeo(X)}(H)$ is piecewise full, and admits a locally compact group topology extended from $H$;
		\item for any $H\leq G\leq C$, the full group $\Full(G)$ is an open subgroup of $C$, so $\Full(G)$ in turn admits a locally compact and locally decomposable group topology.
	\end{enumerate}
\end{cor}

\section{Locally decomposable topological inverse monoids}\label{sec:inverse_monoids}

Let $X$ be a compact, zero-dimensional space. 
As explained in the introduction, the construction of the piecewise full group $\Full(G)$ of a group $G\leq \Homeo(X)$ naturally leads to considering inverse submonoids of 
$$\PHomeo_c(X):=\{\text{homeomorphisms } f:U\rightarrow V \mid U,V \subseteq X \text{ clopen}\}$$
with the extra \emph{join} operation between partial homeomorphisms that have disjoint domains and ranges. 
These three operations yield a \emph{Boolean inverse monoid} of partial homeomorphisms of $X$ (Definition \ref{defn:Boolean_inverse}). 
The piecewise full group $\Full(G)$ is then simply the \defbold{group of units} of this Boolean inverse monoid: those elements $m$ that are fully, not just partially, invertible, meaning that composition with the partial inverse yields the identity function on all of $X$.

We will study this larger algebraic object in the next sections. 
First we need to find a suitable topology to put on it and conditions that make its group of units locally compact. 

In Section \ref{sec:compact_gen} we will examine the conditions on generation of these inverse monoids (and even the right notion of `generation') that ensure that a certain subgroup of the group of units is compactly generated. 

\begin{rmk}
	Compact zero-dimensional spaces are sometimes called Stone spaces because of \defbold{Stone duality}, the categorical correspondence between these spaces and  Boolean algebras (\cite{Stone_duality}, \cite[IV.4]{univ_alg_book}).
	Given a Boolean algebra $\mc{A}$, there is a corresponding compact zero-dimensional topological space $X = \mf{S}(\mc{A})$, called the \defbold{Stone space} of $\mc{A}$. 
	The points of $X$ are the ultrafilters of $\mc{A}$; the basic open sets, which are in fact clopen, are the sets of the form $X_a = \{x \in X \mid a \in x\}$. 
	There is then a natural Boolean algebra isomorphism from $\mc{A}$ to $\mc{CO}(X)$, the algebra of compact open subsets of $X$, given by sending $a$ to $X_a$. 
	Conversely, any compact zero-dimensional space $X$ is naturally isomorphic to the Stone space of $\mc{CO}(X)$, by sending $x \in X$ to the set of $a \in \mc{CO}(X)$ containing $x$. 
	
	It is convenient to identify a Boolean algebra or a compact zero-dimensional space with its double dual, so for instance if $\mc{A}$ is a Boolean algebra and $X = \mf{S}(\mc{A})$ its Stone space, we will regard ``$a \in x$'' and ``$x \in a$'' as synonymous for pairs $(a,x) \in \mc{A} \times X$, and similarly in the case that $X$ is a compact zero-dimensional space and $\mc{A} = \mc{CO}(X)$.
	Given an element $a$ of a Boolean algebra $\mc{A}$, write $a^{\perp}$ for the complement of $a$ in $\mc{A}$.
	
	%We need to be careful here: Stone duality is contravariant, so subalgebras naturally correspond to quotients of the Stone space.
	Stone duality yields an isomorphism between $\PHomeo_c(X)$ and $\PAut(\mc{CO}(X))$ the group of isomorphisms between principal ideals of the Boolean algebra $\mc{CO}(X)$. 
	This restricts to the group isomorphism $\Homeo(X)\cong \Aut(\mc{CO}(X))$. 
	We may at points, and without warning, use Stone duality to switch between partial automorphisms of Boolean algebras and clopen partial homeomorphisms of compact zero-dimensional spaces. 
\end{rmk}

\subsection{Crash course on inverse monoids}\label{ssec:crash_course_inv_monoids}
We review the necessary material on inverse monoids; details can be found in\cite{Lawson_book}. 

If groups are the algebraic objects that encode bijections, inverse monoids are those that encode partial bijections. 
The analogue of the symmetric group $\Sym(\Omega)$ on a set $\Omega$ is the symmetric inverse monoid $I(\Omega)$ consisting of all partial bijections (bijections between subsets of $\Omega$). 
$\Sym(\Omega)$ is the group of units of $I(\Omega)$.

\begin{defn}
	In a monoid $M$, an \defbold{inverse} of $x \in M$ is an element $y \in M$ such that $xyx = x$ and $yxy = y$. 
	The monoid is 
	called an \defbold{inverse monoid} if every element has a unique inverse. 
	In that  case we write $x^*$ for the inverse of $x$. 
	(We use $x^*$ rather than $x^{-1}$ to avoid confusion with inverse functions or group inverses.)
	Note that uniqueness of inverses ensures that $x^{**}=x$ for every $x\in M$;  in particular, $*$ is bijective and all formulae valid in $M$ have a $*$-dual form.
	We write $1$ (or $1_M$ if the ambient monoid is not clear) for the identity element of $M$. 
	The \defbold{group of units} $M^\times$ consists of all elements $x$ such that $xx^* = x^*x=1$; such elements form a group under multiplication.

	An \defbold{inverse submonoid} of $M$ is a subset of $M$ that contains the identity and is closed under multiplication and $*$. 
	
	A \defbold{morphism} of inverse monoids is a map $\theta: M \rightarrow N$ satisfying the formulae $\theta(xy) = \theta(x)\theta(y)$, $\theta(1_M) = 1_N$ and $\theta(x)^* = \theta(x^*)$.
	
\end{defn}

\begin{defn}    
	An \defbold{idempotent} of $M$ is an element $m\in M$ such that $mm=m$. 
	Given a monoid $M$, write $\mc{E}(M)$ for the set of idempotents of $M$. 
	Note that this is an inverse submonoid of $M$, since $m^*=m$ for every $m\in \mc{E}(M)$. 
	Moreover, in an inverse monoid, all idempotents commute.

	For every $x\in M$ the elements $x^*x$ and $xx^*$ are idempotents, called, respectively, the \defbold{domain} and \defbold{range} of $x$.
	They may also be notated $\dom(x)$ and $\ran(x)$, respectively.  
	Idempotents can be thought of as the identity map on their domain.
\end{defn}

Notice that, for every $m\in M$ we have $Mm=Mm^*m$. 
This is important for the analogue of Cayley's theorem, the Wagner--Preston theorem:
\begin{thm}[Wagner--Preston theorem; see {\cite[Section 1.5]{Lawson_book}}]
	Let $M$ be an inverse monoid and $m\in M$. 
	Then the map $r_m:Mm^*\rightarrow Mm, xm^*\mapsto xm^*m$ is a bijection between subsets of $M$. 
	Moreover, $m\mapsto r_m$ is a faithful inverse monoid morphism $M\rightarrow I(M)$.
\end{thm}

\begin{lem}\label{lem:adjoin_idempotents}
	Let $M$ be an inverse monoid  with set of idempotents $E$ and let $L$ be an inverse submonoid of $M$. 
	Then the product set $LE$ is an inverse submonoid of $M$ with $E$ as its set of idempotents. 
	Moreover, $LE=EL=ELE$.
\end{lem}
\begin{proof}
	The set $LE$ is closed under composition because for any $k,l\in L, e,f\in E$ we have $kelf=kell^*lf=kll^*elf$ and $l^*el\in E$. 
	In particular, $EL \subseteq LE$ and dually $LE \subseteq EL$, so $LE = EL$ and $ELE = (LE)E = LE$.
	Clearly also $1 \in LE$, and given $e \in E$ and $l \in L$, we have $(le)^* = el^* \in EL = LE$.
	Thus $LE$ is an inverse monoid. 
	As $E = E1 \subseteq LE$, the set of idempotents of $LE$ contains $E$; conversely, every idempotent of $LE$ belongs to $E$.
\end{proof}

Thinking of an inverse monoid as partial bijections of some set, and idempotents as identity maps on a subset, it makes sense to think of an idempotent $e$ as being smaller than another one $f$ if the domain of $e$ is contained in that of $f$. 
Because of the way partial functions are composed, this implies that, when the idempotents are composed, we obtain the smaller one: $ef=e$.  
Extrapolating this to non-idempotents, we can think of a partial bijection $x$ as being smaller than another $y$ if $x$ is the restriction of $y$ to the domain of $x$; that is, $x=yx^*x$. 

\begin{defn}
	Given an inverse monoid $M$, we define a \defbold{partial order} $\le$ by setting $x \le y$ if $x = yx^*x$ (or, equivalently, $x=xx^*y$). 
	There are other equivalent conditions (see \cite[1.4 Lemma 6]{Lawson_book}).
\end{defn}

In particular, one sees (\cite[1.4 Proposition 8]{Lawson_book}) that within $M$, the idempotents $\mc{E}(M)$ form an order ideal (if $e\le f$ and $f\in \mc{E}(M)$ then $e\in \mc{E}(M)$) in which every pair of elements $\{e,f\}$ has a meet (greatest lower bound) $ef$.

In the case of $I(\Omega)$ and $\PHomeo_c(X)$, it is possible for two idempotents to be disjoint.
In that case their meet is the identity function on the empty set, the empty function,  which we will denote by 0. 
Arbitrary pairs of elements of $\PHomeo_c(X)$ need not have a meet (consider $\{1,f\}$ where the set of fixed points of $f \in \Homeo(X)$ is regular closed but not open), but if two elements have disjoint domains and ranges, their meet is 0.
%This also means that if two elements have disjoint domains and ranges their meet is 0. 

More generally, if $M$ is an inverse monoid with 0, we will say that two elements $x,y\in M$ are \defbold{disjoint}, written $x\perp y$, if $x^*y = xy^* = 0$, or equivalently, $xx^*yy^*=x^*xy^*y=0$.  Note that disjointness is preserved by translation: if $x \perp y$ and $z \in M$ then
	\begin{align*}
		(xz)^*(yz) &= z^*x^*yz = z^*(0)z = 0;  \\
		(xz)(yz)^* &= xzz^*y^* = x(x^*x)(zz^*)y^* =  xzz^*x^*(xy^*) = 0
	\end{align*}
	so $xz \perp yz$, and similarly $zx \perp zy$.

Given any two elements $f,g$ of $I(\Omega)$ (or of $\PHomeo_c(X)$), we may wish to glue or join them together to form an element $f\vee g$ whose domain and range is the union of those of $f$ and $g$. 
This is of course only possible if $f$ and $g$ agree on the intersection of their domains and ranges. 
This is equivalent to  $f^*g$ and $fg^*$ being idempotents -- the identity function on, respectively, the intersection of domains and the intersection of ranges of $f$ and $g$. 

\begin{defn}
	Let $M$ be an inverse monoid. 
	We say that $x,y\in M$ are \defbold{compatible} if $x^*y$ and $xy^*$ are idempotents of $M$.
	A subset $S$ of $M$ is \defbold{compatible} if every pair of elements of $S$ is compatible.
	Say that $x, y\in M$ are \defbold{joinable} if they are compatible and have a unique least upper bound, called their \defbold{join} and denoted $x\vee y$. 
	If every pair of compatible elements of $M$ is joinable, and left and right multiplication distributes over joins, then $M$ is called a \defbold{finitely joinable inverse monoid}. 
\end{defn}

When an inverse monoid has a distributive lattice of idempotents, multiplication distributes over joins.  In fact, the following can be deduced from the proof of \cite[1.4 Proposition 20]{Lawson_book} and the fact that idempotents form an order ideal.
	
	\begin{lem}\label{lem:multiplication_distributes_over_joins}
		Let $M$ be an inverse monoid such that $\mc{E}(M)$ is a distributive lattice.  Let $\{s_1,\dots,s_n\}$ and $\{t_1,\dots,t_m\}$ be compatible sets whose joins $s := \bigvee^n_{i=1} s_i$ and $t := \bigvee^n_{j=1} t_j$ exist. 
		Then $st = \bigvee_{i,j}s_it_j$; in particular, the set $\{s_it_j \mid 1 \le i \le n, 1 \le j \le m\}$ is compatible. 
		Moreover, $s \in \mc{E}(M)$  if and only if $s_i \in \mc{E}(M)$ for all $i$.
	\end{lem}
	
	Using the way products and inverses distribute over joins, we immediately obtain the following.
	
	\begin{cor}\label{cor:joinable_submonoid}
		Let $M$ be an inverse monoid such that $\mc{E}(M)$ is a distributive lattice, and such that $M$ is finitely joinable. 
		Let $S$ be an inverse submonoid of $M$ and let $K$ be the set of joins of finite compatible sets of elements of $S$.  Then $K$ is a finitely joinable inverse submonoid of $M$, with $\mc{E}(K)$ being the set of finite joins of elements of $\mc{E}(S)$.
	\end{cor}

Every inverse monoid acts on its idempotents by conjugation. The following lemma is easy to verify. 

\begin{lem}\label{lem:conjugation}
	Let $M$ be an inverse monoid and let $E = \mc{E}(M)$.  
	The map $\hat{\alpha}: M\rightarrow \PAut(E)$, $m\mapsto \alpha_m$ is an inverse monoid morphism, where $\alpha_m: m^*mE\rightarrow mm^*E, e\mapsto mem^*$ is the partial automorphism induced by $m$ between all idempotents $ \leq\dom(m)$ and all idempotents $\leq \ran(m)$.
	This morphism is injective if and only if for every $m\neq n\in M$ there is $e\in E$ such that $mem^*\neq nen^*$. 
\end{lem}

If the idempotents of $M$ form a Boolean algebra then, by Stone duality,  the conjugation action just described induces an inverse monoid morphism $M\rightarrow \PHomeo_c(X)$ where $X$ is the Stone space of the idempotents of $M$.

\subsection{Boolean inverse monoids}\label{ssec:BIM_prelims}
\begin{rmk}
	From now on, we will only consider inverse monoids $M$ that act faithfully on their idempotents $E$, so that we can instead speak of submonoids of $\PAut(E)$ or of $\PHomeo_c(X)$ where $X$ is the Stone space of $E$, if $E$ is a Boolean algebra.
\end{rmk}

In order to construct piecewise full groups of a space $X$, we must consider inverse monoids that contain all restrictions of, say, a group, to clopens of $X$.
The algebraic way of obtaining restrictions of elements is by composing with idempotents. 
Therefore, if we wish to obtain restrictions to all clopens of $X$, the idempotents of our inverse monoid had better be exactly that set of clopens.
This already holds for $\PHomeo_c(X)$.
Ensuring that finite joins exist in our inverse monoids is the remaining condition that we need.
 As far as we know, the following definition first appears in \cite{Lawson_bim_def}\footnote{In \cite{Lawson_bim_def}, $M$ is also required to be a meet-semilattice, however this condition is inconvenient for our purposes, since for example $\PHomeo_c(X)$ is not a meet-semilattice.}.

\begin{defn}\label{defn:Boolean_inverse}
	Let $X$ be a compact zero-dimensional space. 
	An inverse submonoid $M$ of $\PHomeo_c(X)$ is a \defbold{Boolean inverse monoid} if
	\begin{enumerate}[(i)]
		\item $M$ is finitely joinable; and
		\item $\mc{E}(M)=\mc{CO}(X)$.
	\end{enumerate}
\end{defn}

In Section \ref{sec:groups} we started with a group acting on $X$ and formed its piecewise full group. Another way to obtain the same group is to first take all restrictions to clopens of $X$, then take the Boolean inverse monoid that they generate (that is, take products, inverses and finite joins of all these restrictions) and finally take the group of units of this Boolean inverse monoid.
More generally, we may start with a monoid of clopen partial homeomorphisms of $X$ instead of a group and proceed in the same way, first taking all restrictions and then forming a Boolean inverse monoid, whose group of units will be a piecewise full group. 
The next lemma follows immediately from Lemma~\ref{lem:adjoin_idempotents} and Corollary~\ref{cor:joinable_submonoid}.

\begin{lem}\label{lem:Boolean_completion}
	If $L\leq \PHomeo_c(X)$ is an inverse monoid with set of idempotents $E:=\mc{E}(L)=\mc{CO}(X)$,  and $M$ is an inverse submonoid of $L$, then the set $K\subset \PHomeo_c(X)$ of finite compatible joins of elements of $ME\leq L$ is a Boolean inverse monoid. 
\end{lem}

\begin{defn}
	Let $X$ be a zero-dimensional compact space and $M\leq\PHomeo_c(X)$ be an inverse monoid. 
	The \defbold{Boolean completion} of $M$, denoted $\mc{BI}(M)$, is the inverse monoid $K$ constructed from $M$ as in Lemma~\ref{lem:Boolean_completion}.
	The \defbold{piecewise full group} $\Full(M)$ is the group of units of $\mc{BI}(M)$, or in other words,
	\[
	\Full(M) := \mc{BI}(M) \cap \Homeo(X).
	\]
	We say that $M$ is \defbold{piecewise full} if $M$ is a group and $M = \Full(M)$.
\end{defn}

\subsection{Topological inverse monoids}\label{ssec:top_invmonoids}
Since we wish to produce topological groups using inverse monoids, we need to introduce a suitable notion of topological inverse monoids.

\begin{defn}
	A \defbold{topological inverse monoid} is an inverse monoid $M$ equipped with a topology $\tau$, such that multiplication and $*$ are continuous (the former with respect to the product topology on $M \times M$).  Given $x,y \in M$, write $x(-)y$ for the map
	\[
	x(-)y: M \rightarrow xMy, \; m \mapsto xmy.
	\]
\end{defn}

Notice that, since multiplication is continuous, so are the maps $x(-):M\rightarrow xM$ and $(-)y:M\rightarrow My$ and therefore so is their composition $x(-)y$. 
In particular, given any idempotents $e,f\in\mc{E}(M)$, the map $e(-)f:M\rightarrow eMf$ is  continuous, onto, and fixes every point of $eMf\subseteq M$, so the subspace topology on $eMf$ coincides with the quotient topology of $e(-)f$.

\begin{lem}\label{lem:subspace_quotient}
	Let $M$ be a topological inverse monoid and let $e,f \in \mc{E}(M)$.  Then the subspace topology of $eMf$ is the same as the quotient topology of $e(-)f$. \qed
\end{lem}

In order to decide what would be a reasonable topology to put on a Boolean inverse monoid, we first consider what topology to put on $\PHomeo_c(X)\cong \PAut(\mc{CO}(X))$. 
One reasonable requirement, and one that motivates the compact-open topology for $\Homeo(X)$, is that the natural action on $X$ be continuous. 
In other words, the evaluation map $\mathrm{ev}: \Homeo(X)\times X\rightarrow X$, $(h,x)\mapsto h(x)$ is continuous with respect to the product topology on $\Homeo(X)\times X$. 

The exponential law for topological spaces states that,  given any topological spaces $A,B,C$ with $B$ locally compact, there is a homeomorphism between spaces of continuous functions $\eta:\mc{C}(A\times B, C)\rightarrow\mc{C}(A,\mc{C}(B,C))$ 
where $\mc{C}(-,-)$ is given the compact-open topology (this is in fact motivation for the definition of the compact-open topology on continuous maps). 
In particular,  $\mathrm{ev}\in \mc{C}(\Homeo(X)\times X, X)$ if and only if $\eta(\mathrm{ev})\in \mc{C}(\Homeo(X), \mc{C}(X, X))$. 
In other words, the evaluation map is continuous if and only if the inclusion map $\eta(\mathrm{ev}): \Homeo(X)\rightarrow \mc{C}(X, X)$ is continuous with respect to the compact-open topology on $\mc{C}(X,X)$. 
This is what yields the usual compact-open topology of $\Homeo(X)$.

Extending the exponential law to continuous maps with clopen domain, and allowing for continuous inversion, yields a topology on $\PHomeo_c(X)$ with the following sub-basic opens, ranging through $d,r\in \mc{CO}(X)$:
\begin{align*}
	V_{d,r} &:= \{h\in \PHomeo_c(X): d\subseteq\dom(h),  h(d) \subseteq r\} ,\\
	V_{d,r}^{*} & = \{h\in \PHomeo_c(X): d\subseteq \ran(h), h^*(d)\subseteq r\},\\
	V_{d,\#}  &:=  \{h\in\PHomeo_c(X):  \dom(h)\cap d=\emptyset\},\\
	V_{\#, d} &:= \{h\in\PHomeo_c(X):  \ran(h)\cap d=\emptyset \} = V_{d,\#}^{*}.
\end{align*}

This is the coarsest inverse monoid topology on $\PHomeo_c(X)$ that makes evaluation  $\mathrm{ev}:\PHomeo_c(X)\times X\rightarrow X$ a partial continuous map with clopen domain.  
It is also a Hausdorff topology. 
The group of units of $\PHomeo_c(X)$ is of course $\Homeo(X)$ and the subspace topology is the usual compact-open topology.

\begin{defn}
	Given an inverse monoid $M$, we define the \defbold{compact-open topology} on $M$ to be the topology generated by the sets
	\begin{align*}
		V_{e,f} & := \{m \in M \mid e\leq m^*m, mem^* \leq f\},  \\
		V_{e,f}^{*} &:= \{m\in M \mid e\leq mm^*, m^*em\leq f\},  \\
		V_{e,\#} & := \{m\in M \mid e m^*m=0\}=\{m\in M\mid m^*m\leq e^{\perp}\},  \\
		V_{\#,e} &  := \{m\in M\mid e mm^*=0\}=\{m\in M\mid mm^*\leq e^{\perp}\}. 
	\end{align*}
	where $e,f \in \mc{E}(M)$ and $e^{\perp}$ denotes the complement of $e$, if it exists.
\end{defn}

\begin{lem}\label{lem:compactopen_top_0dim_units_clopen}
	Let $M$ be an inverse monoid whose set of idempotents $\mc{E}(M)$ forms a Boolean algebra, with top and bottom elements denoted by 1 and 0. 
	Give $M$ the compact-open topology. Then:
	\begin{enumerate}[(i)]
		\item $M$ is a zero-dimensional inverse monoid; 
		\item the idempotents $\mc{E}(M)$ form a discrete submonoid of $M$;
		\item $\{0\}$ is open and the group of units of $M$ is clopen.
	\end{enumerate}
\end{lem}
\begin{proof}
	(i) The definition of compact-open topology already ensures that it is an inverse monoid topology.
	Recall that a topological space is zero-dimensional (with respect to small inductive dimension) if it has a base consisting of clopen sets.
	It suffices to show that the sub-basic opens $V_{e,f}$ and $V_{e,\#}$ are closed, for every $e,f\in\mc{E}(M)\setminus \{0\}$.
	First observe that 
	$V_{e,1}=\{m\in M \mid e\leq m^*m\}$. 
	Then 	$M\setminus V_{e,\#}=\bigcup_{d\leq e}V_{d,1}$. 
	Moreover, 
	$$M\setminus V_{e,1}=\{m\in M \mid e\nleq m^*m\}=\{m\in M \mid \exists d\leq e: dm^*m=0\}=\bigcup_{d\leq e}V_{d,\#}$$
	and so 
	$$M\setminus V_{e,f}=\{m\in M\mid e\nleq m^*m\}\cup \{m\in M \mid e\leq m^*m, mem^*\nleq f)\}=(M\setminus V_{e,1})\cup\bigcup_{0 < g\leq e}V_{g,f^{\perp}}$$
	where $e,f\in\mc{E}(M)\setminus\{0\}$ and $e^{\perp}$ denotes the complement of $e$ in $\mc{E}(M)$.
	
	(ii) The idempotents $\mc{E}(M)$ are a discrete subspace of $M$ with respect to this topology because $e$ is the only idempotent in $V_{e,e}\cap V_{e^{\perp},\#}$.
	
	(iii)Notice that  $\{0\}=V_{1,\#}=V_{\#,1}$ and that $V_{1,1}\cap V_{1,1}^*$ is the group of units of $M$.
\end{proof}

The inverse monoids we consider in this paper act faithfully on their idempotents $\mc{E}(M)$, so they can be considered submonoids of $\PAut(\mc{E}(M))$. 
The compact-open topology on $M$ in this case is just the subspace topology for the compact-open topology on $\PAut(\mc{E}(M))$.

	Part (ii) of Lemma~\ref{lem:compactopen_top_0dim_units_clopen} has a kind of converse.
	\begin{lem}\label{lem:discrete_idempotents_implies_compact-open}
		Let $M$ be a topological inverse monoid whose set of idempotents $\mc{E}(M)$ forms a Boolean algebra, with top and bottom elements denoted by 1 and 0.  Suppose that $\mc{E}(M)$ is discrete in $M$.  Then the topology of $M$ is a refinement of the compact-open topology.
	\end{lem}
	
	\begin{proof}
		Say that the function $w: M \rightarrow M$ is a word function of $M$ if $w$ can be expressed as a composition of functions $M^n \rightarrow M^n$ that permute the inputs and apply $*$ to a subset of them; functions $M^{n+1} \rightarrow M^n$ that multiply the first two entries together; and functions $M^n \rightarrow M^{n+1}$ that either duplicate the first entry, or append a fixed element of $M$ to the start of an $n$-tuple.  Say $w$ is an idempotent-valued word function if in addition, $w(M) \subseteq \mc{E}(M)$.  Since $M$ is a topological inverse monoid, we see that all word functions are continuous; in particular, if $w$ is idempotent-valued, then the set of $m \in M$ such that $w(l) \in F$ is open for every $F \subseteq \mc{E}(M)$.  In particular, we see that all the basic open sets can be expressed as conditions on idempotent-valued word functions, and therefore they must all be open sets in the given topology on $M$.
	\end{proof}

For our purposes, we shall need to find refinements of the compact-open topology on $M$ (just as for subgroups of $\Homeo(X)$ we found refinements of the compact-open topology in order to induce a group topology to the piecewise full group). 
In particular, this implies that for all pairs $e,f$ of idempotents of $M$, 
\begin{align*}
	eM &= \{m\in M \mid \ran(m)\leq e\} & = V_{\#, e^{\perp}},\\
	Mf &= \{m\in M\mid \dom(m)\leq f\} & = V_{f^{\perp},\#},\\
	eMf &= eM\cap Mf & 
\end{align*}
are open in $M$. 
Thus the quotient topology on each of $eM$, $Mf$ and $eMf$, which coincides with the subspace topology, is actually contained in the topology of $M$.

The following straightforward lemma will be useful later on.
\begin{lem}\label{lem:other_opens_in_idempotent_top}
	Let $M$ be an inverse monoid whose idempotents $E=\mc{E}(M)$ form a Boolean algebra and $d,r\in E$.  Then the following equations hold:
	\begin{enumerate}[(i)]
		\item $W_{d,r} :=\{m\in M\mid mdm^*\leq r\} = \displaystyle{\bigcup_{c\in E, c\leq d} V_{c,r}\cap V_{rc^{\perp},\#}}$
		\item $W_{d,r}\cap W_{r,d}^*=\{m\in M\mid md=rm\}$
		\item $U_{d,r}:=\{m\in M\mid mdm^*=r\} = \displaystyle{\bigcup_{c\in E, c\leq d} V_{c,r}\cap V_{r,c}^*\cap V_{dc^{\perp},\#}}$
		\item $D_d:=\{m\in M\mid m^*m=d\}=V_{d,1}\cap V_{d^{\perp},\# }.$
	\end{enumerate}
	In particular, the above and their inverses are open in the compact-open topology of $M$. 
\end{lem}

\begin{rmk}\label{rmk:cpt-open-top-centralisers-idempotents-open}
	We will make frequent use of the following observation: if the topology of $M$ refines the compact-open topology, then the centraliser $\{m \in M \mid em = me\}$ of each idempotent is open.
\end{rmk}

\subsection{Topological Boolean inverse monoids}\label{ssec:topBIMs}
We now turn to defining topological Boolean inverse monoids. 
Recall that if $M$ is a Boolean inverse monoid, we identify its Boolean algebra $E$ of idempotents with the clopens $\mc{CO}(X)$ of the Stone space $X$ of $E$. 

In a Boolean inverse monoid we have a third operation, join, which we can always take to be between elements with disjoint domain and range:
$$\bigvee:\{(s,t)\in M\times M\mid s^*st^*t=ss^*tt^*=0 \}\rightarrow M, \quad (s,t)\mapsto s\vee t.$$ 
Notice that every finite disjoint join can be expressed as a suitable combination of instances of $\bigvee$. 

\begin{defn}\label{defn:Boolean_inverse:topological}
	A \defbold{topological Boolean inverse monoid} $M$ is a Boolean inverse monoid and a topological inverse monoid, whose topology is a refinement of the compact-open topology (so that $M$ acts continuously on its idempotents) and such that the join $\bigvee$ of disjoint pairs of elements is continuous. 
\end{defn}

Any Boolean inverse monoid becomes a topological Boolean inverse monoid when equipped with the discrete topology; but there are also non-discrete topological  Boolean inverse monoids, the most basic example being $\PHomeo_c(X)$ where $X$ is a Stone space:

\begin{lem}\label{lem:phomeo_cts_joins}
	Disjoint joins are  continuous on $\PHomeo_c(X)$ with respect  to the compact-open  topology.
\end{lem}
\begin{proof}
	Define $J:=\{(f, g)\in \PHomeo_c(X)\mid ff^*\perp gg^*, f^*f\perp g^*g \}$ and $\bigvee: J\rightarrow  \PHomeo_c(X)$  by $(f,g) \mapsto f\vee g$.
	It is enough to show that the preimages of sub-basic opens under $\bigvee$ are open.
	For $d,r\in \mc{CO}(X)$, we     have
	\begin{align*}
		\bigvee{}^{ -1}(V_{d,r})&=\{(f,g)\in J\mid d\leq  f^*f\vee g^*g,   (f\vee g)d(f\vee g)^*\leq r\}\\
		& = \bigcup \left(V_{d_1, r_1}\times V_{d_2,r_2} : d_1\perp d_2, d=d_1\vee d_2, r_1\perp r_2, r_1\vee r_2\leq r \right)\\
		\bigvee{}^{ -1}(V_{d,\#}) &= \{(f,g)\in J \mid f^*f\vee g^*g\perp d\} = J \cap (V_{d,\#} \times V_{d,\#}).
%		&= \bigcup \left( (V_{d_2\vee d, \#}\cap V_{\#, r_1^{\perp}})\times (V_{d_1\vee d, \#}\cap V_{\#, r_2^{\perp}}) : d_1\perp d_2, d_1\vee d_2 \perp d, r_1\perp r_2 \right).
	\end{align*}
	As taking disjoint joins commutes with taking inverses, we see that $V_{d,r}^*$ and $V_{d,\#}^*=V_{\#,d}$ are also open by the above calculations. 
\end{proof}

Requiring disjoint joins to be continuous leads naturally to the definition of local decomposability, analogous to that seen in Section \ref{sec:groups}. 

Every element of a Boolean inverse monoid $M$ can be expressed as a disjoint join of other elements:
if $m\in M$ with $d=m^*m$ and $P\subset E=\mc{CO}(X)$ is any finite partition of $d$, then $m=\bigvee_{e\in P}me$. 
This leads to the map
$$\rho_P: Md\rightarrow \prod_{e\in P} Me, \; m\mapsto (me)_{e\in P}.$$
The image of $\rho_P$ consists of all tuples in $\prod_{e\in P}Me$ whose entries have disjoint ranges (some of them may be 0). 
This is exactly the set $\displaystyle{\bigcup_{q\in \mc{Q}_P}\prod_{e\in P}q(e)Me}$
where $\mc{Q}_P$ denotes the set of maps $q:P\rightarrow P_1$ with $P_1\subset E$ a finite set of pairwise disjoint idempotents and $q$ is injective on $P\setminus q^{-1}(0)$. 
In fact, this is exactly the set of all tuples in $\prod_{e\in P}Me$ that can be disjointly joined to form an element of $M$. 
In other words, $\rho_P$ is the inverse of the disjoint join operation $\bigvee_P$ on $\bigcup_{q\in \mc{Q}_P}\prod_{e\in P}q(e)Me$; so that this set is in bijection with $Md$.

The map $\rho_P$ is always continuous if $M$ is assumed to have continuous multiplication, as it is the product of the multiplication maps by $e\in P$. 
Therefore, requiring that $\rho_P$ be an open map turns it into a homeomorphism, as well as making the join map $\bigvee_P$ continuous. 
If all disjoint joins are continuous then $\rho_P$ must be open for every finite partition $P$ of $d$ and every $d\in \mc{E}\setminus\{0\}$.

Repeating the above with left instead of right multiplication by idempotents yields that the map
$$\lambda_P: dM\rightarrow \bigcup_{q\in \mc{Q}_P} \prod_{e\in P}eMq(e) \subset \prod_{e\in P}eM, \;\;  m\mapsto (em)_{e\in P}$$
is a homeomorphism, where $P\subset E$ is a partition of the idempotent $d$.

Since $\{0\}, fMe, Me, fM$ are open in the compact-open topology, the maps $\rho_P$ and $\lambda_P$ have open images in $\prod_{e\in P}Me$, and $\prod_{e\in M}eM$.
This leads to the following definition of locally decomposable inverse monoid.

\begin{defn}\label{defn:locally_decomposable}
	Let $X$ be a zero-dimensional compact space and $M\leq \PHomeo_c(X)$ a topological inverse monoid whose topology refines the compact-open topology of $\PHomeo_c(X)$. 
	
	For every $e \in E := \mc{CO}(X)$, equip the inverse monoids $Me$ and $eM$ with the quotient topology induced by the maps $m\mapsto me$ and $m\mapsto em$, respectively. 
	
	Let $P$ be a finite set of pairwise disjoint elements of $E$, define the maps
	$$\rho_P: M\rightarrow \prod_{e\in P}Me, \; m\mapsto (me)_{e\in P} \quad \text{and} \quad \lambda_P:M\rightarrow \prod_{e\in P}eM, \; m\mapsto(em)_{e\in P}$$
	and note that they are continuous. 
	We say that $M$ is \defbold{locally decomposable} if $\rho_P$ and $\lambda_P$ are open maps for all $P$.

\end{defn}

\begin{lem}
	Let $M$ be a topological Boolean inverse monoid.  Then $M$ is locally decomposable. \hfill \qedsymbol
\end{lem}

Note that for local decomposability, it is enough to check that one of the maps $\rho_P$ and $\lambda_P$ is open; the other one will then be open because $*$ is a homeomorphism.
In fact, it suffices to check that $\rho_e:=\rho_{\{e, e^{\perp}\}}$ is open for all $e\in \mc{CO}(X)$:

\begin{lem}\label{lem:locally_decomposable:partition}
	Let $M\leq \PHomeo_c(X)$ be a topological inverse monoid acting	 continuously on $X$. 
	Suppose that $\rho_e: M \rightarrow Me \times Me^{\perp}, \; m \mapsto (me,me^{\perp})$ is an open map for every $e\in \mc{CO}(X)$, with respect to the quotient topologies on $Me$ and $Me^{\perp}$. 
	Then $M$ is locally decomposable.
\end{lem}

\begin{proof}
	It suffices to check that $\rho_P$ is an open map for every finite set $P$ of pairwise disjoint elements of $\mc{CO}(X)$. 
	We proceed by induction on $|P|$. 
	In the base case, $P$ is a singleton $\{e\}$. 
	We then see that $\rho_{\{e\}} = \pi_e \rho_e$, where $\pi_e$ is the projection map onto the first factor of $Me \times Me^{\perp}$.  Clearly $\pi_e$ is open, and $\rho_e$ is open by hypothesis, so $\rho_{\{e\}}$ is open.
	
	Suppose now that $|P| \ge 2$; write $P = \{e_1,\dots,e_n\}$.  Then we can write $\rho_P$ as $\rho' \rho_{e_1}$, where
	\[
	\rho': \pi(M)e_1 \times \pi(M)e_1^{\perp} \rightarrow \prod^n_{i=1}\pi(M)e_i, \;\; (a,b) \mapsto (a,be_2,be_3,\dots,be_n).
	\]
	By the inductive hypothesis, 
	\[
	M \rightarrow \prod^n_{i=2}Me_i; \; b \mapsto (be_2,\dots,be_n)
	\]
	is an open map. 
	Since $Me_1^{\perp}$ carries the quotient topology of $m \mapsto me_1^{\perp}$, it follows that
	\[
	Me_1^{\perp} \rightarrow \prod^n_{i=2}Me_i; \; b \mapsto (be_2,\dots,be_n)
	\]
	is also open. 
	Thus $\rho'$ is open, and hence the composition $\rho_P = \rho' \rho_{e_1}$ is open, finishing the proof.
\end{proof}

\begin{rmk}\label{rmk:loc_dec_implies_restriction_is_open_map}
	If $M\leq\PHomeo_c(X)$ is locally decomposable, then, as established in the base of the induction above, the map $M\rightarrow Me$ (resp., $eM$) is open, for all $e\in \mc{CO}(X)$.
	
	Now, $M\cap Me=\{m\in M\mid m^*m\leq e\}=M\cap V_{e^{\perp},\#}$ is open in $M$, and similarly for $M\cap eM=\{m\in M\mid mm^*\leq e\}$.
	So, if $e\in \mc{CO}(X)\cap M$, then right (resp., left) multiplication by $e$ is a continuous and open map $M\rightarrow M$. 
	
	In particular, 	any open inverse submonoid $N\leq M$ satisfies that $Ue$ is open in $M$ for any open $U\subseteq N$ and $e\in \mc{CO}(X)\cap M$.
\end{rmk}

The following easy but important consequence of the Wagner--Preston theorem shows that multiplication by any fixed element is an open and continuous map in a locally decomposable inverse monoid. 
\begin{lem}\label{lem:loc_dec_monoid_implies_multiplication_by_m_is_open}
	Let $M$ be a topological inverse monoid such that for every idempotent $e\in\mc{E}(M)$, right-multiplication by $e$ is an open map. 
	Then right-multiplication by $m$ is an open map, for every $m\in M$.
\end{lem}
\begin{proof}
	For every $m\in M$, right-multiplication   $x\mapsto xm$ is the same as $x\mapsto (xmm^*)m$ the composition of right-multiplication by $e=mm^*$ and by $m$. 
	Now, the image of right-multiplication by $e$ is $Me=Mmm^*=Mm^*$, which, by the Wagner--Preston theorem and continuity of multiplication, is homeomorphic to $Mem=Mm^*m=Mm$ via right-multiplication by $m$ restricted to $Me$. 
	So $Uem=Um$ is open in $Mm$ for every open subset  $U$ of $M$. 
\end{proof}

Of course, this new definition of locally decomposable agrees with the previous one for subgroups of $\Homeo(X)$:

\begin{prop}\label{prop:loc_dec_monoid_matches_group}
	Let $X$ be a zero-dimensional compact space and $G\leq \PHomeo_c(X)$ a  subgroup. 
	Then $G\leq \Homeo(X)$ and the following are equivalent:
	\begin{enumerate}[(i)]
		\item $G$ is locally decomposable in the inverse monoid definition;
		\item for every clopen partition $P$ of $X$ the subgroup $\grp{\rist_G(e) \mid e \in P}$ is open and carries the product topology.
\end{enumerate}	
\end{prop}
\begin{proof}
Since $G$ is a submonoid, its identity must be the identity of $\PHomeo_c(X)$ and so for every $g\in G$ we have $gg^{-1}=g^{-1}g=\id_X$, 
so $G\leq \Homeo(X)$. 

(i) $\implies$ (ii). Suppose that $G$ is locally decomposable.

	Let $P$ be a clopen partition of $X$.  Since the topology of $G$ refines the compact-open topology, the group $G_P=\bigcap_{e \in P}G_e$ is open and we see that $\rist_G(e) \le G_P$ for all $e \in P$.  We then have a continuous group monomorphism $\rho_P:G_P\rightarrow \prod_{e \in P} G_Pe$ that is moreover open by assumption.  In particular, there are identity neighbourhoods $O_e$ in $G_Pe$ such that $\rho_P(G_P)$ contains $\prod_{e \in P} O_e$.  Since $G_P$ is a group, we see that the largest possible choice for $O_e$ here is exactly $\rist_G(e)e$, so $\prod_{e \in P} \rist_G(e)e$ is open in $\prod_{e \in P} G_Pe$.  The preimage of $\prod_{e \in P} \rist_G(e)e$ is exactly $H : = \grp{\rist_G(e) \mid e \in P}$, so $H$ is open in $G$, and the map $\rho_P$ now restricts to a homeomorphism from $H$ to $\prod_{e \in P} \rist_G(e)e$, so $H$ carries the product topology.

	(ii) $\implies$ (i). 
	For every $e\in \mc{CO}(X)$, the subgroup $G_e$ is open because it contains the open subgroup $H := \rist_G(e)\times \rist_G(e^{\perp})$,
	so the topology of $G$ refines the compact-open topology of $\PHomeo_c(X)$.
	
	It now suffices to show that for every $e\in\mc{CO}(X)$ the map $\rho_e:G\rightarrow Ge\times Ge^{\perp}$ is open.  Given an identity neighbourhood $U$ in $G$, since $H$ carries the product topology we see that $U$ contains a set of the form $W := \rist_V(e) \times \rist_V(e^{\perp})$, where $V$ is an identity neighbourhood.
	Since $H$ is open, $W$ is itself an identity neighbourhood.  Given $g \in G$, we deduce that $\rho_e(gU)$ contains the neighbourhood $gWe \times gWe^\perp$ of $\rho_e(g)$, showing that $\rho_e$ is open as required.
\end{proof}

\subsection{Constructing topological Boolean inverse monoids}\label{ssec:induce_topBIM_up}

In this subsection, $X$ denotes a zero-dimensional compact space and we have a locally decomposable inverse submonoid $M \le \PHomeo_c(X)$.  Our aim is to use $M$ to build a topological Boolean inverse monoid in which $M$ is embedded as an open submonoid.
				
We first note a general condition ensuring uniqueness of topology, which also indicates how we will extend the topology from $M$ to a larger monoid.
				
\begin{lem}\label{lem:extended_monoid_topology_is_unique}
Let $M\leq \PHomeo_c(X)$ be a locally decomposable inverse submonoid with topology $\sigma$ and suppose $M \le L$, where $L$ is also an inverse monoid.  
Then there is at most one topology $\tau$ on $L$ for which $L$ is a locally decomposable inverse monoid, $M \in \tau$ and $\sigma$ is the subspace topology of $M$ in $L$. 
Specifically, if it exists, $\tau$ must be such that each $f \in L$ has a base of open neighbourhoods of the form $Nf \cap fN$, where $N$ ranges over all open neighbourhoods of $1$ in $M$; the sets $Nf$ also form a base of neighbourhoods of $f$, as do the sets $fN$.
\end{lem}

\begin{proof}
To specify a topology, it is enough to declare a base of neighbourhoods at each point. 
Suppose that $\tau$ is a locally decomposable inverse monoid topology on $L$ extending $\sigma$.
Given $f \in M$, we see that by local decomposability and Lemma~\ref{lem:loc_dec_monoid_implies_multiplication_by_m_is_open}, there is a base of open neighbourhoods of $f$ in $\sigma$ consisting of sets of the form $Nf$, where $N$ is an open neighbourhood of $1$ in $M$; one can also take sets of the form $fN$. 
Similarly, given $f \in L$, then for all open neighbourhoods $N$ of $1$ in $M$, the sets $Nf$ and $fN$ must be $\tau$-neighbourhoods of $f$.  Let $\tau_0$ be the topology generated by this system of neighbourhoods; clearly $\tau_0 \subseteq \tau$.

Consider now $U \in \tau$ and $f \in U$. 
As $(L,\tau)$ is locally decomposable, Lemma \ref{lem:loc_dec_monoid_implies_multiplication_by_m_is_open} implies that right-multiplication by $f$ is a continuous and open map $L\rightarrow L$. 
Since $M\in\tau$, this means that $f\in U\cap Mf\in \tau$.
As right-multiplication is continuous and $\sigma $ is the subspace topology on $M\in \tau$, we have $1\in V_f:=\{m\in M\mid mf\in U\cap Mf\} \in \sigma$. 
Thus $f\in V_ff=U\cap Mf \in \tau_0$. 
As $f\in U$ was arbitrary, $U\in \tau_0$ and so $\tau_0=\tau$. 

Note that the sets $Nf$ for $N$ an open neighbourhood of $1$ in $M$ form a base of $\tau_0$-neighbourhoods of $f\in L$. 
\end{proof}

Next, we incorporate the idempotents of $\PHomeo_c(X)$ into $M$.

\begin{lem}\label{lem:locally_decomposable_extension}
Let $M\leq \PHomeo_c(X)$ be a locally decomposable inverse monoid. 
Put $E=\mc{CO}(X)$ and $K=ME$. 
Then $K$ is an inverse monoid and there is a unique locally decomposable inverse monoid topology on it for which $M$ is an open submonoid.
\end{lem}

\begin{proof}
$K=ME$  is an inverse monoid by Lemma~\ref{lem:adjoin_idempotents}.

Denote by $\sigma$ the topology on $M$.
We equip $K$ with the topology $\tau$ from Lemma~\ref{lem:extended_monoid_topology_is_unique}, which is the only topology that can be compatible with the requirements of the present lemma.  Let us check that $\tau$ satisfies all the requirements. 

\emph{We have $M \in \tau$ and the original topology of $M$ coincides with the subspace topology in $K$.}
Let $f \in K$ and $1 \in N \in \sigma$.  For each $n \in N$ there is $1 \in N_n \in \sigma$ such that $N_nn \subseteq N$.  We can write $Nf\cap M$ as the union $S$ of sets in $\{N_n(nf) \mid n \in N, nf \in M\}$; note that $S \subseteq \sigma$ by Lemma \ref{lem:loc_dec_monoid_implies_multiplication_by_m_is_open} applied to $M$.

\emph{Inversion is a homeomorphism of $K$.}
Clear from the symmetry in the definition.

\emph{Multiplication in $K$ is continuous.}
Let $O\in \sigma, e\in E$ and suppose that $(af,bg)$ is in the preimage of $Oe$ under multiplication, where $a,b\in M$ and $f,g\in E$.
We will find a neighbourhood of $(af,bg)$ with respect to $\tau\times \tau$ whose image under multiplication is contained in $Oe$. 

We have $afbg=abb^*fbg=abgb^*fb=ue=ueb^*fb$ for some $u\in O$, so 
$$u\in O_1:=\{c\in O\mid ce=ceb^*fb \}=O\cap V_{e(b^*fbg)^{\perp},\#}\in \sigma.$$
As $M$ is locally decomposable, $O_1b^*fbg$ is open in $Mb^*fbg$ where the latter carries the quotient topology; so the preimage $V$ of $O_1$ under the quotient map $M\rightarrow Mb^*fbg$ is in $\sigma$ and contains $ab$.
Now, $(a,b)\in W:=\{(s,t)\in M\times M\mid st\in V, t^*ft=b^*fb\}$. 
This is the intersection of the preimage of $V$ under multiplication and $M\times (M\cap U_{f,b^*fb}^*) \in \sigma\times \sigma$ (by Lemma \ref{lem:other_opens_in_idempotent_top}).
So $\{(sf,tg)\mid (s,t)\in W\}\in \tau\times \tau$ is the required neighbourhood of $(af,bg)$. 

\emph{$\tau$ refines the compact-open topology on $K$.}
If $e,f\in E$, then 
\begin{align*}
V_{e,\#}\cap K &=\{k\in K\mid k^*k\leq e^{\perp} \}=Ke^{\perp}=\bigcup_{d\leq e^{\perp}}Md \in \tau \\
V_{e,f}\cap K &= \{k\in K\mid e\leq k^*k, kek^*\leq f\}=\bigcup_{d\geq e}(M\cap V_{e,f})d \in \tau
\end{align*}
and $(V_{\#,e}=V_{e,\#}^* )\cap K, \, V_{e,f}^*\cap K\in \tau$ by the above.	

\emph{$K$ is locally decomposable. }
By Lemma \ref{lem:locally_decomposable:partition}, it suffices to show that $\rho_e:K\rightarrow Ke\times Ke^{\perp}$ is an open map for each $e\in E$.
Let $U\in\sigma$ and $e,f\in E$.
It is enough to prove that $\rho_e(Uf)\subset Ke\times Ke^{\perp}$ is in $\tau\times\tau$. 
Note that right multiplication by $f$ is an open map $K\rightarrow K$, as $Vdf\in \tau$ for every $ V\in\sigma $ and $ d\in E$. 
So $\cdot f\times \cdot f: (k_1,k_2)\mapsto (k_1f,k_2f)$  is also an open map.
Since $M$ is locally decomposable, $\rho_e(U)$ is open in $Me\times Me^{\perp}\in \tau\times \tau$
and, as $\cdot f\times \cdot f$ is an open map, $\rho_e(Uf)=(\cdot f\times \cdot f)\circ\rho_e(U) \in \tau\times \tau$. 
\end{proof}

Analogous to the situation for groups, we can now define the open preserver of a locally decomposable inverse monoid, and show that the topology extends.
This will allow us to effectively generalise Proposition~\ref{prop:full_contains_intersection_rists} from the group setting to the setting of Boolean topological inverse monoids.
				
\begin{defn}
Let $M \leq \PHomeo_c(X)$ be a locally decomposable inverse submonoid, let $f \in \PHomeo_c(X)$ and let $E = \mc{CO}(X)$.  We say $f$ \defbold{preserves the opens on the left} of $M$ if for every open neighbourhood $O$ of $1$ in $M$, there is a neighbourhood $U$ of $1$ in $M$ such that
\[
fUf^* \subseteq EOE.
\]
We say $f$ \defbold{preserves the opens} of $M$ if $f$ and $f^*$ preserve the opens on the left of $M$.  The \defbold{open preserver} $\mathrm{OP}(M)$ is then the set of elements of $\PHomeo_c(X)$ that preserve the opens of $M$.
\end{defn}

\begin{thm}\label{thm:open_preserver_monoid}
Let $M\leq \PHomeo_c(X)$ be a locally decomposable inverse submonoid.  Then the open preserver $\mathrm{OP}(M)$ of $M$ is a Boolean inverse monoid containing $M$ and $\mc{CO}(X)$.  Moreover, there is a unique choice of topology $\tau$ on $\mathrm{OP}(M)$ making $\mathrm{OP}(M)$ a topological Boolean inverse monoid such that $M \in \tau$ and the restriction of $\tau$ to $M$ is the topology of $M$.
\end{thm}
			
\begin{proof}
Write $E = \mc{CO}(X)$ and $L = \mathrm{OP}(M)$.  Equip $ME$ with the topology extending that of $M$, as in Lemma~\ref{lem:locally_decomposable_extension}.  Since $M$ is an open neighbourhood of $1$ in $ME$ and has the subspace topology, it is clear that $\mathrm{OP}(M) = \mathrm{OP}(ME)$.  Thus from now on, without loss of generality we may assume $E \subseteq M$.  By continuity of multiplication it is also clear that $M \subseteq L$.

We now proceed by series of claims.

\emph{Claim 1: Given $f \in L$ and a neighbourhood $O$ of $1$ in $M$, there is a neighbourhood $U$ of $1$ in $M$ such that
\[
fUf^* \subseteq ff^*O \cap Off^*.
\]}

Using Remark \ref{rmk:cpt-open-top-centralisers-idempotents-open}, by replacing $O$ with a smaller neighbourhood we may assume that  $ff^*o = off^*$ for all $o\in O$.
Let $U$ be a neighbourhood of $1$ in $M$ such that
\[
fUf^* \subseteq EOE.
\]
After passing to a smaller neighbourhood we may assume that $U$ commutes with $ff^*$ and $f^*f$.  Since the topology of $M$ refines the compact-open topology, by Lemma~\ref{lem:compactopen_top_0dim_units_clopen} we may assume that $U$ consists of units.  Given $u \in U$, then $(fuf^*)fu^*f^* = ff^*fuu^*f^* = ff^*$, so $ff^* \in fuf^*M$, and similarly $ff^* \in Mfuf^*$.  Suppose $fuf^* = eoe'$ for $e,e' \in E$.  Then
\[
ff^* \in eoe'M \cap Meoe' \subseteq eM \cap Me',
\]
so $eff^* = ff^* = ff^*e'$, in other words, $e,e' \ge ff^*$.  We can then rewrite $fuf^*$ as
\[
fuf^* = ff^*fuf^*ff^* = ff^*eoe'ff^* = ff^*off^* = ff^*o,
\]
showing that in fact $fUf^* \subseteq ff^*O = Off^*$, satisfying the claim.

\emph{Claim 2: $L$ is an inverse submonoid of $\PHomeo_c(X)$.}

It is clear from the definition that $E \subseteq L$, and that $f^* \in L$ whenever $f \in L$.

Suppose $f_1,f_2 \in L$ and let $O$ be an open neighbourhood of $1$ in $M$.  Then by Claim 1 there is a neighbourhood $U_1$ of $1$ in $M$ such that
\[
f_1U_1f^*_1 \subseteq f_1f^*_1O,
\]
and in turn there is a neighbourhood $U_2$ of $1$ in $M$ such that
\[
f_2U_2f^*_2 \subseteq f_2f^*_2U_1.
\]
Write $e = f_1f_2f^*_2f^*_1$.  Combining the two inclusions, we see that 
\[
f_1f_2U_2f^*_2f^*_1 \subseteq f_1f_2f^*_2U_1f^*_1 = ef_1U_1f^*_1 \subseteq ef_1f^*_1O \subseteq EOE.
\]
Thus $f_1f_2$ preserves the opens on the left of $M$.  Similarly, $f^*_2f^*_1$ preserves the opens on the left of $M$, and hence $f_1f_2 \in L$.  Thus $L$ is an inverse submonoid of $\PHomeo_c(X)$ that contains the idempotents, as claimed.

\emph{Claim 3: $L$ is closed under taking compatible joins.}

Since $L$ is a submonoid that contains the idempotents, it is enough to show that $L$ is closed under taking joins of disjoint pairs.  So we consider $f,g \in L$ such that $f \perp g$, and write $h = f \vee g \in \PHomeo_c(X)$.  Let $O$ be an open neighbourhood of $1$ in $M$; we can then take an open neighbourhood $U$ of $1$ in $M$ such that
\[
fUf^*, gUg^* \subseteq EOE.
\]
Write $d = f^*f$ and $r = ff^*$. 
Since the topology of $M$ refines the subspace topology in $\PHomeo_c(X)$, by  Remark \ref{rmk:cpt-open-top-centralisers-idempotents-open}, there is an open neighbourhood of 1 that centralises the idempotents $d,r,d^\perp,r^\perp$.
Thus we may assume that $O$ and $U$ centralises them too.
Note also that for all $u \in U$ we have $fuf^* \in Mr$ and $gug^* \in Mgg^* \subseteq Mr^\perp$.

Given $u \in U$, we see that
\[
hudh^* = huddh^* = hdudh^* = hf^*fuf^*fh^* = fuf^*,
\]
since $f \le h$.  Similarly,
\[
hud^\perp h^* = h(h^*h)d^\perp ud^\perp (h^*h) h^* = hg^*g u g^*gh = gug^*.
\]
We can thus write
\[
huh^* = h(ud \vee ud^\perp)h^* = (fuf^*) \vee (gug^*).
\]
We claim there is a neighbourhood $V$ of $1$ in $U$ such that $(fuf^*) \vee (gug^*)$ is an element of $EOE$ for all $u \in V$.  Since $EOE$ and $\rho_r$ are open, the set $\rho_r(EOE)$ is open in the product $Mr \times Mr^\perp$, and hence there is an open neighbourhood of $\rho_r(1) = (r,r^\perp)$ of the form $O'r \times O'r^\perp$, for $O'$ an open neighbourhood of $1$ in $O$, such that $O'r \times O'r^\perp \subseteq \rho_r(EOE)$; in other words, $EOE$ contains the join of any element of $O'r$ with any element of $O'r^\perp$.  Now take an open neighbourhood $V$ of $1$ in $U$ such that
\[
fVf^*, gVg^* \subseteq EO'E.
\]
Given $u \in V$, then we can write $fuf^* = e_fo_fe'_f = e_fo_fe'_fr$ and $gug^* = e_go_ge'_g = e_go_ge'_gr^\perp$, where $e_f,e_g,e'_f,e'_g \in E$ and $o_f,o_g \in O'$.  The join then belongs to $E(o_fr \vee o_gr^\perp)E$, specifically
\[
e_fo_fe'_fr \vee e_go_ge'_gr^\perp = e_fo_fre'_f \vee e_go_gr^\perp e'_g = (e_fr \vee e_gr^\perp)(o_fr \vee o_gr^\perp)(re'_f \vee r^\perp e'_g).
\]
Thus $(fuf^*) \vee (gug^*) \in EOE$ as desired, in other words, $hVh^* \subseteq EOE$.  This completes the proof that $h$ preserves the opens on the left; similarly, $h^*$ preserves the opens on the left, and hence $h \in L$.  This completes the proof that $L$ is closed under compatible joins.

\

By Claims 2 and 3, we see that $L$ is a Boolean inverse monoid; it remains to show that the topology $\sigma$ extends to $L$, to produce a topological Boolean inverse monoid.  We equip $L$ with the topology $\tau$ from Lemma~\ref{lem:extended_monoid_topology_is_unique}, which is the only possibility for a suitable topology for $L$.  The proof of the lemma also shows that $M = M1$ belongs to $\tau$, and that $\sigma$ is exactly the $\tau$-subspace topology on $M$: every $\tau$-neighbourhood of a point contains a $\sigma$-neighbourhood, and vice versa.

Claims 4, 5 and 6 will show that $\tau$ is indeed a topology compatible with the Boolean inverse monoid structure.

\emph{Claim 4: The join operation is continuous on disjoint pairs in $L$.}

As in Claim 3 we consider $f,g \in L$ such that $f \perp g$, and write $h = f \vee g$.  It is enough to consider neighbourhoods of $h$ of the form $Oh$ for $O$ a neighbourhood of $1$ in $M$; we can take $O = O^*$ and take $O$ small enough that it centralises $\{ff^*,f^*f,gg^*,g^*g\}$.  Given $o_f, o_g \in O$ we claim that $o_ff \perp o_gg$:
\begin{align*}
(o_ff)^* (o_gg) = f^*(ff^*)(o^*_fo_g)(gg^*)g = f^*(o^*_fo_g)(ff^*)(gg^*)g = 0; \\
(o_ff)(o_gg)^* = o_ffg^*o_g = 0.
\end{align*}
Thus the join $o_ff \vee o_gg$ is a well-defined element of $L$.
By local decomposability of $M$, the image of $O$ under the map $o \mapsto (off^*,ogg^*)$ is open in $Mff^* \times Mgg^*$; in particular, there exists a neighbourhood $O_1$ of $1$ in $O$ such that for every pair $(o_fff^*,o_ggg^*)$ with $o_f,o_g \in O_1$, there exists $o \in O$ such that $off^* = o_fff^*$ and $ogg^* = o_ggg^*$.  We now claim, given $o_f,o_g \in O_1$, that the join $o_ff \vee o_gg$ is an element of $Oh$.  Specifically, we can write $o_ff = (o_fff^*)f$ and $o_gg = (o_ggg^*)g$, and then find $o \in O$ such that $off^* = o_fff^*$ and $ogg^* = o_ggg^*$, so that
\[
o_ff \vee o_gg = off^*f \vee ogg^*g = of \vee og = o(f \vee g) = oh,
\]
recalling that multiplication distributes over joins.  Thus the preimage of $Oh$ under the join operation contains the neighbourhood $O_1f \times O_1g$ of $(f,g)$, showing the required continuity property.

\emph{Claim 5: $L$ equipped with $\tau$ is a topological inverse monoid.}

We have already shown that $L$ is an inverse monoid; the fact that $*$ is a homeomorphism is evident from the symmetry in the definition of $\tau$.  All that remains is to check continuity of multiplication.  Let $f,g \in L$, and consider a neighbourhood of $fg$ of the form $Ofg$, where $O$ is a neighbourhood of $1$ in $M$.  By continuity of multiplication in $M$, there is $1 \in O_1 \in \sigma$ such that $O_1O_1 \subseteq O$.  Since $f$ preserves the opens of $M$, by Claim 1 there is $1 \in O_2 \in \sigma$ such that $fO_2f^* \subseteq O_1ff^*$.  Now given $o_f,o_g \in O_2$, we see that
\[
o_ffo_gg = o_fff^*fo_gg = o_f(fo_gf^*)fg \subseteq O_1(O_1ff^*)fg \subseteq Off^*fg = Ofg,
\]
showing that the preimage of $Ofg$ under multiplication contains $O_2f \times O_2g$.  Thus multiplication is continuous, proving the claim.

\emph{Claim 6: $\tau$ refines the subspace topology of $L$ inside $\PHomeo_c(X)$.}

The subspace topology in this case is the compact-open topology; since $E \subseteq M$ and $\sigma$ refines the subspace topology inside $\PHomeo_c(X)$, we see that $E$ is a discrete subspace of $L$.  The claim then follows by Claim 5 and Lemma~\ref{lem:discrete_idempotents_implies_compact-open}.
\end{proof}
			
In particular, in the situation of Theorem~\ref{thm:open_preserver_monoid} we see that $\mathrm{OP}(M)$ contains the Boolean completion $\mc{BI}(M)$, which is just the submonoid generated by $M \cup E$ together with the disjoint join operation.
				
\begin{cor}\label{cor:piecewise_full_monoid_extension}
	Let $M\leq\PHomeo_c(X)$ be a  locally decomposable topological inverse monoid and let $L=\mc{BI}(M)$ be its Boolean completion. 
	Then there is a unique topology $\tau$ on $L$ that turns it into a topological Boolean inverse monoid with $M$ as an open submonoid.
\end{cor}
			
\begin{proof}
				The existence of the topology $\tau$ follows from Theorem~\ref{thm:open_preserver_monoid}, since $\mc{BI}(M) \subseteq \mathrm{OP}(M)$.  Uniqueness follows from Lemma~\ref{lem:extended_monoid_topology_is_unique}.
			\end{proof}
				
\begin{rmk}
	If $M$ is a group, then the group of units of $\mathrm{OP}(M)$ is just the open preserver of $M$ in $\Homeo(X)$ in the group-theoretic sense, as defined in Section~\ref{sec:group_extend_topology}.
\end{rmk}

The topology on $\mc{BI}(M)$ inherits some useful properties from $M$.

\begin{prop}\label{prop:piecewise_full_extension:lcsc}
	Let $M\leq \PHomeo_c(X)$ be a locally decomposable topological inverse monoid and equip $\mc{BI}(M)$ and $\mathrm{OP}(M)$ with the respective unique Boolean inverse monoid topologies that make $M$ open.
		\begin{enumerate}[(i)]
			\item If $M$ is locally compact or first-countable, then so are $\mc{BI}(M)$ and $\mathrm{OP}(M)$.
			\item If $M$ and $X$ are second-countable then so is $\mc{BI}(M)$.
		\end{enumerate}
\end{prop}

\begin{proof}
		We recall from Lemma~\ref{lem:extended_monoid_topology_is_unique} how the topology of $\mathrm{OP}(M)$ is constructed, and see that every point has a neighbourhood that is the image under a continuous open map of a neighbourhood of $1$ in $M$.  Thus if $M$ is locally compact or first-countable, so is $\mathrm{OP}(M)$.  The open submonoid $\mc{BI}(M)$ then inherits these properties from $\mathrm{OP}(M)$, proving (i).
		
		For (ii), we assume $M$ and $X$ are second-countable.  Let $E$ be the set of idempotents of $\PHomeo_c(X)$ and let $\mc{D}$ be the set of finite subsets $D$ of $E$ such that the elements of $D$ are pairwise disjoint.  Since $X$ is second-countable and compact, we see that $E$ is countable, hence also $\mc{D}$ is countable.  We can now write 
		\[
		\mc{BI}(M) = \bigcup B_D, \text{ where } B_D = \{ \bigvee_{e \in D}m_e \mid m_e \in Me\}.
		\]
		Given $D \in \mc{D}$, the set $B_D$ is the continuous image of the finite product $\prod_{e \in D}Me$, and in turn, each set $Me$ is a continuous image of $M$.  Since $M$ is second-countable, so is $B_D$; each of the sets $B_D$ is open in $\mc{BI}(M)$ and $\mc{BI}(M)$ is a countable union of such sets, so $\mc{BI}(M)$ is second-countable.
\end{proof}

\section{Generation properties and expansivity}\label{sec:compact_gen}
In this section we deal with the alternating full group of a Boolean inverse monoid.
Following Nekrashevych (\cite{Nekra}), who does everything in the setting of \'{e}tale groupoids, we prove in Theorem \ref{thm:expansive_generation} that the alternating full group is compactly generated if the Boolean inverse monoid whence it comes is compactly generated. 
In turn, the compact generation of a Boolean inverse monoid is related to \emph{expansive} actions and subshifts.

Throughout this section, $X$ denotes a compact zero-dimensional space, unless stated otherwise.

\subsection{Expansive actions, subshifts, and compact generation}

In the theory of dynamical systems, expansive homeomorphisms of compact metric spaces are usually of interest. 
These are the homeomorphisms which eventually separate points some fixed minimum amount. 
This can be generalised to uniform spaces (sets equipped with a uniformity, see e.g., \cite[Ch. II]{Bourbaki_top14_07}), of which metric spaces and topological groups are prime examples.

\begin{defn}\label{def:expansive_uniformity}
	Let $X$ be a uniform space. 
	A group action of $G$ on $X$ is \defbold{expansive} if there is an entourage $E\subset X\times X$ such that if $x\neq y\in X$ there is some $g\in G$ such that $(gx,gy)\notin E$. 
\end{defn}

Every uniformity gives rise to a topology. 
If the resulting space is compact, there is exactly one uniformity compatible with this topology, which consists of all neighbourhoods of the diagonal (\cite[II, p.27, Th\'{e}or\`{e}me 1]{Bourbaki_top14_07}).

In the setting of actions on compact zero-dimensional spaces, there is a close connection between expansivity and subshifts.

\begin{defn}
	A \defbold{subshift} consists of a group action $(G,X)$ on a compact zero-dimensional space $X$, such that there is a $G$-equivariant closed embedding $\phi: X \rightarrow Y^G$, where $Y$ is a finite discrete set,  $Y^G$ is given the product topology and $G$ acts on $Y^G$ by shifts:
	\[
	h\cdot f(g) = f(gh) \text{ for } g,h\in G, \, f\in Y^G.
	\]
\end{defn}

\begin{lem}\label{lem:subshift}
	Let $G$ be a group acting by homeomorphisms on a  space $X$.
	\begin{enumerate}[(i)]
		\item Suppose that $(G,X)$ is a subshift. 
		Then $X$ is compact and zero-dimensional and the action is expansive.
		\item Suppose that $X$ is  compact and zero-dimensional.
		Given a clopen partition $\mc{P}$ of $X$, consider $\mc{P}$ as a finite discrete set  and define $\phi_{\mc{P}}: X \rightarrow (\mc{P})^G$ by 
		\[
		\phi_{\mc{P}}(x)(g) = [g(x)]_{\mc{P}} \text{ the part of } \mc{P} \text{ containing } g(x).
		\]
		Then $\phi_{\mc{P}}$ is a $G$-equivariant continuous map, and as $\mc{P}$ ranges over the clopen partitions of $X$, the spaces $(G,\phi_{\mc{P}}(X))$ are subshifts that form an inverse system (via the maps $\phi_{\mc{P}}$) with limit $X$. 
	 	\end{enumerate}
\end{lem}

\begin{proof}
	(i)
	Let $\phi: X \rightarrow Y^G$ be as in the definition of a subshift.
	As $Y$ is finite, $Y^G$ is compact by Tychonoff's theorem and zero-dimensional. 
	As $\phi$ is a closed embedding, $X$ is also compact and zero-dimensional.
	For each $y\in Y$ define $U_y:=\{x\in X\mid \phi(x)(1)=y\}$. 
	These sets form a clopen partition of $X$, so $E:=\bigcup_{y\in Y}U_y\times U_y$ is an entourage for the unique uniformity on $X$. 
	If $x\neq z\in X$ then, as $\phi$ is injective, there is some $g\in G$ such that $\phi(x)(g)\neq \phi(z)(g)$.
	But $\phi(x)(g)=g\cdot \phi(x)(1)=\phi(g(x))(1)$ and $\phi(z)(g)=g\cdot\phi(z)(1)=\phi(g(z))(1)$, 
		
	(ii)
	Given $g,h \in G$ and $x \in X$, we have
	\[
	h.\phi_{\mc{P}}(x)(g) = \phi_{\mc{P}}(x)(gh) = [gh(x)]_{\mc{P}} = \phi_{\mc{P}}(h(x))(g),
	\]
	so $\phi_{\mc{P}}$ is $G$-equivariant.
	To see that $\phi_{\mc{P}}$ is continuous, it is enough to observe that the cylinder sets $\{f \in \mc{P}^G \mid f(g) = \alpha\}$ generate the topology of $\mc{P}^G$ as $g$ ranges over $G$ and $\alpha$ ranges over $\mc{P}$, and the preimage of such a set is the open set $\{x \in X \mid g(x) \in \alpha\}$.
	
	Let $\mathbf{P}$ be the set of clopen partitions of $X$. 
	For any finite set $\mc{P}_1,\dots,\mc{P}_n$ of elements of $\mathbf{P}$, we can factor all of the maps $\phi_{\mc{P}_i}$ through $\phi_{\mc{P}}$, where $\mc{P}$ is the partition generated by $\bigcup^n_{i=1}\mc{P}_i$.  Thus the maps $(\phi_{\mc{P}})_{\mc{P} \in \mathbf{P}}$ form an inverse system. 
	 Moreover, given $x,y \in X$ such that $x \neq y$,  there is a clopen partition $\mc{P}$ such that $[x]_{\mc{P}} \neq [y]_{\mc{P}}$, so that $\phi_{\mc{P}}(x)(1) \neq \phi_{\mc{P}}(y)(1)$. 
	  Hence $(\phi_{\mc{P}})_{\mc{P} \in \mathbf{P}}$ separates points in $X$; since $X$ is compact, this ensures that the inverse limit is isomorphic to $X$ as a $G$-space.
	
	Since $\mc{P}^G$ is zero-dimensional, we see that $\phi_{\mc{P}}(X)$ is zero-dimensional; since $X$ is compact, so is $\phi_{\mc{P}}(X)$. 
	Thus $\phi_{\mc{P}}(X)$ is a subshift. 
	\end{proof}

\begin{cor}\label{cor:subshift}
	Let $G$ be a group.
	\begin{enumerate}[(i)]
		\item Every compact zero-dimensional $G$-space is an inverse limit of expansive $G$-spaces.
		\item A $G$-space $X$ is a subshift if and only if it is compact, totally disconnected and expansive. 	
		This is equivalent to there being a clopen partition of $X$ whose $G$-translates separate the points of $X$. 
	\end{enumerate}
\end{cor}
\begin{proof}
	Let $X$ be a compact zero-dimensional $G$-space.  By Lemma~\ref{lem:subshift}(ii), $X$ is an inverse limit of $G$-spaces $\phi_{\mc{P}}(X)$; each $G$-space $\phi_{\mc{P}}(X)$ is expansive by Lemma~\ref{lem:subshift}(i).
	Now suppose that $(G,X)$ is expansive where  $E\subset X\times X$ is the entourage witnessing this.
	Then we may assume without loss of generality that $E=\bigcup_{\alpha \in \mc{P}}\alpha\times\alpha$ where $\mc{P}$ is a clopen partition of $X$. 
	The assumption on $E$ guarantees that $G$-translates of $\mc{P}$ separate points of $X$; equivalently, the corresponding $\phi_{\mc{P}}$ from Lemma \ref{lem:subshift} is injective. 
	Because $\phi_{\mc{P}}$ is continuous, $X$ is compact and $\mc{P}^G$ is Hausdorff, we see that $\phi_{\mc{P}}$ is a closed embedding, so $(G,X)$ is a subshift.  
\end{proof}

\begin{defn}
	Let $M\leq \PHomeo_c(X)$ be a Boolean inverse monoid. 
	For a subset $S \subseteq M$, denote by $S^{\infty}$ the smallest inverse submonoid of $M$ containing $S$ and write $J(S)$ for the set of elements obtainable as joins of finitely many compatible elements of $S$.
	We say $S$ \defbold{generates $M$} if $S$ is not contained in any proper finitely joinable inverse submonoid of $M$.
	A topological Boolean inverse monoid $M$ is \defbold{compactly generated} if there is a compact generating set $S$ for $M$.
\end{defn}

\begin{lem}\label{lem:bim_generating_set}
	Let $M$ be a Boolean inverse monoid generated by $S\subseteq M$. 
	 Then $M=J(S^{\infty})$ and $\mc{E}(M) = J(\mc{E}(S^\infty))$.
	If $\mc{E}(S^\infty)$ contains a base of topology of $X$ that is a union of partitions, then every element of $M$ is a finite join of disjoint elements of $S^\infty$.
\end{lem}
\begin{proof}
		By Corollary~\ref{cor:joinable_submonoid}, $J(S^\infty)$ is a finitely joinable inverse submonoid of $M$ and $\mc{E}(J(S^\infty))$ is the set of finite joins over $\mc{E}(S^\infty)$.  Since $S$ generates $M$ we deduce that $J(S^\infty) = M$, and in particular every element of $\mc{E}(M)$ is a finite join of elements of $\mc{E}(S^\infty)$. 
		
		Suppose now that $B \subseteq \mc{E}(S^\infty)$ is a base of topology of $X$ that is a union of partitions.  Given $e \in \mc{E}(M)$ we can write $e^\perp = \bigvee^n_{i=1}e_i$ where $e_i \in B$.  Then $e_i$ belongs to some partition $P_i$ of $\mc{E}(M)$ contained in $B$, so we can write $e^\perp_i$ as the join of the disjoint set $P'_i := P_i \setminus \{e_i\}$.  Now
		\[
		e = \prod^n_{i=1} e^\perp_i = \prod^n_{i=1} \bigvee P'_i = \bigvee \mc{S},
		\]
		where $\mc{S}$ consists of products formed by taking one element from each of the sets $P'_i$, and the last equality is by distributivity.
		Any two distinct elements of $\mc{S}$ are disjoint, so $e$ is expressed as a disjoint join of elements of $\mc{E}(S^\infty)$.
		Given an arbitrary $m \in M$, we can write $m = \bigvee^n_{i=1}s_i$ where $s_i \in S^\infty$; we can then rewrite $m$ as the disjoint join $m = \bigvee^n_{i=1}(s_if_i)$, where $f_i = s^*_is_i\prod_{j < i}(s^*_js_j)^\perp$, and we can express $m$ as a disjoint join of elements of $S^\infty$ as follows:
		\[
		f_i = \bigvee^{m_i}_{j=1}e_j, \; \{e_1,\dots,e_{m_i}\} \subseteq \mc{E}(S^\infty) \text{ disjoint}; \; m = \bigvee^n_{i=1}\bigvee^{m_i}_{j=1}s_ie_j. \qedhere
		\]
	
\end{proof}

Expansive group actions are  intimately related to generation properties of the corresponding Boolean inverse monoid. 
Recall that a group $G$ acts expansively on $X$ if and only if there is a clopen partition of $X$ whose $G$-translates separate points of $X$.

	\begin{prop}\label{prop:subshift_expansive}
		Let $M$ be a topological Boolean inverse monoid and let $G \le M^\times$ be an open subgroup such that $M = \mc{BI}(G)$.
		\begin{enumerate}[(i)]
			\item If $M$ is compactly generated, then there is a compact identity neighbourhood $S$ in $G$ and a partition $P \subseteq \mc{E}(M)$ such that $S \cup P$ generates $M$.  For any such choice of $(S,P)$, the action of the group $H = \grp{S}$ on $X$ is such that the $H$-translates of $P$ separate points in $X$, so $H$ acts expansively on $X$; moreover, $\mc{E}((SP)^\infty)$ contains a base of topology that is a union of partitions.
			\item If $G$ is compactly generated and acts expansively on $X$, then $M$ is compactly generated.
		\end{enumerate}
	\end{prop}

	\begin{proof}
		Denote $E:=\mc{E}(M)=\mc{CO}(X)$.
		
		For (i) we suppose that $M$ is compactly generated, say by a compact subset $T$. 
		Given Lemma~\ref{lem:bim_generating_set}, we can write $M$ as a directed union $M = \bigcup_{i,j} J((G_iE_j)^\infty)$, where $G_i$ ranges over compact identity neighbourhoods in $G$ and $E_j$ ranges over finite subsets of $E$.  
		Moreover, for each choice of $G_i$ and $E_j$, the submonoid $(G_iE_j)^\infty$ is open as it is a union of cosets of the open subgroup $\grp{G_i}$ (use Remark \ref{rmk:loc_dec_implies_restriction_is_open_map} and Lemma \ref{lem:loc_dec_monoid_implies_multiplication_by_m_is_open});
		hence $J((G_iE_j)^\infty)$ is open, because joins are open and continuous in Boolean inverse monoids.
		Thus by compactness of $T$, there exist $G_i$ and $E_j$ such that $T \subseteq J((G_iE_j)^\infty)$  and hence $M = J((G_iE_j)^\infty)$.  We see that in fact $J((G_iE_j)^\infty) \subseteq J((G_iP_1)^\infty)$ where $P$ is the set of atoms of the finite Boolean algebra generated by $E_j$ and $P_1 = P \cup \{1\}$.  Now take $S = G_i$; it is clear that $S \cup P$ generates $M$ and that $(S \cup P)^\infty = (SP_1)^\infty$.
		
		Write $H = \grp{S}$.  Since $M = J((SP_1)^\infty)$, we see that $\mc{E}((SP_1)^\infty)$ forms a base of topology.  On the other hand, given $e \in \mc{E}((SP_1)^\infty)$, then $e$ can be written as $\prod^n_{i=1} h_ie_ih^*_i$ for some $h_i \in H$ and $e_i \in P$.  We deduce that $P$ is a partition whose $H$-translates separate points in $X$, and so the action of $H$ on $X$ is expansive.
		
		It remains to see that $\mc{E}((SP_1)^\infty)$ contains a base of topology that is a union of partitions.
		Given a finite sequence $\mathbf{s} = (s_1,\dots,s_n)$ of elements of $S$, let $A_{\mathbf{s}}$ be the set of all elements of the form $s_ne_ns_{n-1}e_{n-1} \dots s_1e_1$, with $e_i\in P$ and let $P_{\mathbf{s}}$ be the set of domains of $A_{\mathbf{s}}$.  
		Then $P_{\mathbf{s}}$ is a partition contained in $\mc{E}((SP_1)^\infty)$.
		 Given finitely many such sequences $\mathbf{s}_1,\dots,\mathbf{s}_k$, we can take the concatenation $\mathbf{s}$, and see that $P_{\mathbf{s}}$ refines $P_{\mathbf{s}_i}$ for all $i$; thus we obtain a collection of partitions in $\mc{E}((SP_1)^\infty)$ that is directed by inclusion.  We see that every element of $\mc{E}((SP_1)^\infty)$ belongs to $J(B)$, where $B \subseteq \mc{E}((SP_1)^\infty)$ is the union of the partitions $P_{\mathbf{s}}$ over all finite sequences in $S$; since $\mc{E}((SP_1)^\infty)$ forms a base of topology, so does $B$.
		
		For (ii), we fix a compact symmetric generating set $S$ of $G$ such that $1 \in S$.  Suppose that the action of $G$ on $X$ is expansive.  Then there is a partition $P \subseteq E$ whose $G$-translates separate points in $X$.  Let $T = S \cup P$; note that $T$ is compact and that $T^\infty = (SP_1)^\infty$ where $P_1 = P \cup \{1\}$.  By Corollary~\ref{cor:joinable_submonoid}, $J(T^\infty)$ is a finitely joinable submonoid of $M$, so in particular $\mc{E}(J(T^\infty))$ is a sublattice of $E$, consisting of all joins of finite subsets of $\mc{E}(T^\infty)$.  Construct the subset $B$ of $\mc{E}(T^\infty)$ in the same way as in the proof of (i).  Since the $G$-translates of $P$ separate points in $X$, we see that $B$ generates the topology, and is consequently a base of topology.  In particular, it follows that $\mc{E}(J(T^\infty)) = E$ and we deduce that $J(T^\infty) = J(GE) = M$, showing that $M$ is compactly generated.
	\end{proof}
	
	If each $M$-orbit has at least three points, then we will see (Lemma~\ref{lem:alternating_to_full}) that 
	$$M = \mc{BI}(A) = \mc{BI}(M^\times)$$
	 where $A = \Al(M)$ is a specific normal subgroup of $M^\times$, based on a construction of Nekrashevych (\cite{Nekra}); thus the requirement that there exists $G \le M^\times$ open such that $M = \mc{BI}(G)$ is not a significant restriction for our purposes.  Moreover, we will see that the compactly generated open subgroup $H$ appearing in Proposition~\ref{prop:subshift_expansive}(i) can always be chosen to be $A$. 
	 The main work of the rest of this section will be to show (Theorem~\ref{thm:expansive_generation}) that the normal subgroup $A$ constructed is indeed compactly generated; we will deduce later (Corollary~\ref{cor:A_is_open}) that $A$ is open.

\subsection{Alternating full groups}\label{ssec:alternating_defn}

A useful class of subgroups of a piecewise full group was introduced by Nekrashevych \cite{Nekra}, which we recall here.  
Where appropriate we restate definitions in terms of inverse monoids instead of groupoids.

\begin{defn}
	Let $M$ be an inverse monoid with a 0 element and let $N$ be an inverse submonoid of $M$. 
	A \defbold{multisection of $N$ in $M$} is
	a finite inverse subsemigroup $S$ of $N\mc{E}(M)$ satisfying:
	\begin{enumerate}[(i)]
		\item all idempotents of $S$ are pairwise disjoint; and
		\item given any two idempotents $e_1,e_2 \in S\setminus\{0\}$, there is exactly one element $f_{12} \in S$ such that $f^*_{12}f_{12} = e_1$ and $f_{12}f^*_{12} = e_2$.
	\end{enumerate}
	Write $\mc{E}^*(S) : = \mc{E}(S) \setminus \{0\}$.
	The \defbold{degree} of a multisection $S$ is $\deg(S) := |\mc{E}^*(S)|$.
	A multisection of degree $d$ is also called a \defbold{$d$-section}.
\end{defn}

In essence, a multisection $S$ of degree $d$ consists of partial transpositions between $d$ disjoint idempotents of $\mc{E}(M)$. 
In other words, it is a copy of the subsemigroup of all maps of rank 1 in the symmetric inverse monoid of degree $d$ (bijections $f_{ij}:i\mapsto j$ for $i,j\in\{1,\dots, d\}$ where $f_{ii}=e_i$).
By allowing joins as well as multiplication and inversion, a multisection of degree $d$ generates a copy of the symmetric inverse monoid  of degree $d$. 
Its group of units is of course the symmetric group of degree $d$.
Notice that this  copy of the symmetric group is not in general in the group of units of $M$. 
For instance, $f_{12}\vee f_{21}\vee e_3\vee \dots \vee e_d$ is a full permutation of the $d$ idempotents of $S$, but is not invertible in $M$ unless $1$ is the join of all idempotents of $S$. 
To obtain an invertible element of $M$, we simply join the missing idempotent $1\setminus \bigvee_{i=1}^de_i$ to this: $h_{12}:=f_{12}\vee f_{21} \vee e_{3}\dots \vee e_d \vee (1\setminus \bigvee_{i=1}^de_i)$ is an invertible element of $M$.

\begin{defn}
	Let $S\subset M$ be a multisection of $N$ in $M$ of degree $d$. 
	For a permutation $\pi\in \Sym(d)$ denote by $h_{\pi}\in M$
	the element defined analogously to $h_{12}$.
	The \defbold{symmetric}, respectively, \defbold{alternating group associated to $S$} is the group 
	$$\Sym[S]:=\{h_{\pi} : \pi \in \Sym(d) \},  \text{ respectively, } \Alt[S]:=\{h_{\pi} : \pi \in \Alt(d) \}.$$

	In practice, we will consider $M=\PHomeo_c(X)$ and $N\leq M$.
	Given $x \in X$ we define the \defbold{$M$-orbit} of $x$ to be the set $\{mx \mid x \in m^*m\}$.
	The $M$-orbits form a partition of $X$. 
	Say $M$ is \defbold{minimal} if every $M$-orbit is dense.
	For $d \ge 2$, define the \defbold{$\Sy_d$-group} and \defbold{$\Al_d$-group} of $N$ as follows:
	\[
	\Sy_d(N;X) := \grp{ \Sym[S] \mid \text{$S$ is an $N$-multisection of degree $d$} };
	\]
	\[
	\Al_d(N;X) := \grp{ \Alt[S] \mid \text{$S$ is an $N$-multisection of degree $d$}}.
	\]
	
	We may write $\Sy_d(N)$ and $\Al_d(N)$ if the $X$ is clear from context. 
\end{defn}

Note that $\Sy_1(N;X) = \Al_2(N;X) = \triv$, but thereafter, given the well-known generation properties of the symmetric and alternating groups, there are descending chains:
\[
\Sy_2(N) \ge \Sy_3(N) \ge \dots; \quad \Al_3(N) \ge \Al_4(N) \ge \dots.
\]
\begin{defn}
	For an inverse submonoid $N\leq \PHomeo_c(X)$, define its \defbold{symmetric full group} by $\Sy(N)=\Sy_2(N)$ and its \defbold{alternating full group }by $\Al(N)=\Al_3(N)$.
	
	We say that  $N$ is \defbold{$d$-alternatable} if $\Al_d(N) \le N$, abbreviating  $3$-alternatable to   \defbold{alternatable}.
\end{defn}

If $N\leq\Homeo(X)$ is a piecewise full group, it is of course alternatable. 
The converse is not true in general: if $g$ is a minimal homeomorphism of the Cantor space $X$ and $G=\Full(\grp{g})$ then, by the theorem we are about to cite, the derived subgroup $\Der(G)$ of $G$ is alternatable, but it is known that $\Der(G)\neq G=\Full(\Der(G))$ (see e.g., \cite[Proposition 5.5]{GPS99}), so $\Der(G)$ is not piecewise full. 

The definition of alternatable action is motivated by the following result of Nekrashevych.

\begin{thm}[{\cite[Theorem~4.1]{Nekra}}]\label{thm:Nekrashevych_simple}
	Let $X$ be a nonempty perfect compact zero-dimensional space and let $G$ be a group with a faithful minimal action on $X$.  Let $\triv\neq H \le \Full(G;X)$ be such that $\Al(G;X)$ normalises $H$.  Then $\Al(G;X) \le H$.  In particular, $\Al(G;X)$ is simple and $H$ is alternatable.
\end{thm}

The above theorem implies that $\Al(G;X)$ is the intersection of all normal subgroups of $\Full(G;X)$. 
In other words, it is the \defbold{monolith} of $\Full(G;X)$.

The alternating full group recovers, under mild assumptions, all the information of the Boolean inverse monoid. 

\begin{lem}\label{lem:alternating_to_full}
Let $M\leq\PHomeo_c(X)$ be an inverse monoid.
	If every $M$-orbit has at least 2 points then $\mc{BI}(\Sy(M))=\mc{BI}(M)$. 
	If every orbit has at least 3 points then $\mc{BI}(\Al(M))=\mc{BI}(\Sy(M))=\mc{BI}(M)$. 
	In particular, $\Full(\Al(M))=\Full(\Sy(M))=\Full(M)$. 
\end{lem}

\begin{proof}
	One inclusion follows from  $\mc{BI}(\Al(M))\leq \mc{BI}(\Sy(M))\leq \mc{BI}(\Full(M)) \le \mc{BI}(M)$.
	
		We only prove the other inclusion for the statement involving $\Al(M)$, the one with $\Sy(M)$ being analogous.
	Let $b=\bigvee_{i=1}^nm_ie_i$ with $m_i\in M, e_i\in E$. 
	As there are at least 3 points in every $M$-orbit, for each $x$ in the domain of $b$,  there is some $l_x\in M$  and $x\in d_x \le e_i$, with $d_x$ sufficiently small, such that $d_x$, $bd_xb^*$ and $l_xd_xl_x^*$ are pairwise disjoint. 
	Set $f_x = d^\perp_x (bd_xb^*)^\perp (l_xd_xl_x^*)^\perp$; we see that $h_x:= bd_x\vee l_x(bd_x)^* \vee (l_xd_x)^* \vee f_x^{\perp} \in \Al(M)$ and $h_xd_x=bd_x $. 
	The $d_x$ then form a clopen cover of the domain of $b$, which is compact, and therefore finitely many $d_x$ suffice; say $d_1, \dots, d_k$. 
	So $b=\bigvee_{i=1}^kh_id_i \in \mc{BI}(\Al(M))$, as required. 
\end{proof}

\begin{lem}\label{lem:A(G)_quotientspace}
	Let $G\leq \Homeo(X)$ act minimally. 
	Then $X$ does not admit any proper non-trivial Hausdorff quotient as an $\Al(G;X)$-space, and $\Al(G;X)$ is minimal and expansive on $X$.
\end{lem}

\begin{proof}
	Using multisections it is easy to see that $\Al(G;X)$ is highly transitive on every $G$-orbit in $X$.  In particular, $\Al(G;X)$ acts minimally on $X$ and primitively on $Gx$ for all $x \in X$.  Suppose now that $\phi:X \rightarrow Y$ is an $\Al(G;X)$-equivariant quotient map to some compact Hausdorff space $Y$, and write $[x]$ for the set $\{y \in X \mid \phi(x) = \phi(y)\}$.  In particular, each fibre $[x]$ is closed and forms a block of imprimitivity for the action of $\Al(G;X)$ on $X$.  If $[x]$ contains $Gx$ for some $x \in X$, then $[x]$ is dense and hence equal to $X$, so $Y$ is a singleton.
	
	From now on, let us assume instead that $[x] \cap Gx$ is a proper subset of $Gx$ for every $x \in X$.  Then since $\Al(G;X)$ acts primitively on $Gx$, we must have $[x] \cap Gx = \{x\}$ for every $x \in X$.  Given $x \in X$ and $y \in [x] \setminus \{x\}$, we see (using multisections) that $(x,y)$ and $(gx,y)$ are in the same $\Al(G;X)$-orbit for some $g \in G$ such that $x \neq gx$; but then we would have $gx \in [y] = [x]$, contradicting the fact that $[x] \cap Gx = \{x\}$.  Thus in fact $[x] = \{x\}$ for all $x \in X$, that is, $\phi$ is injective.  Thus there are no proper non-trivial Hausdorff quotients of $X$ as an $\Al(G;X)$-space.  Applying Lemma~\ref{lem:subshift} to the action of $\Al(G;X)$, it follows that $\Al(G;X)$ is expansive on $X$.
\end{proof}

\subsection{Compactly generated alternating full groups}

In \cite{Nekra}, Nekrashevych gives a sufficient condition for $\Al(G)$ to be finitely generated, where $G$ is an \'{e}tale groupoid.  
In this section we prove an analogous result to establish compact generation of $\Al(G;X)$, where $G$ is a compactly generated locally compact group with locally decomposable action on a compact zero-dimensional space $X$. 
Because it is more convenient for the purposes of compatibility with the group topology, we develop the argument in the setting of Boolean inverse monoids, rather than \'{e}tale groupoids.

The next theorem is the main goal of this section.

\begin{thm}\label{thm:expansive_generation}
	Let $M\leq \PHomeo_c(X)$ be a compactly generated topological Boolean inverse monoid such that every orbit of $M$ on $X$ has at least $5$ points. 
	Then the closure $\ol{\Al(M)}$ of $\Al(M)$ in $M^\times$ is compactly generated.
\end{thm}

%\red{We obtain a bounded generating set $S$ for $\Al(M)$, but we don't know if $S$ is closed (probably it is not), so we cannot say at this stage that $\Al(M)$ is compactly generated.  Of course this becomes moot as soon as we prove $\Al(M)$ is open.}

For the proof we will use covers of multisections, in the following sense.

\begin{defn}
	Let $S$ be a multisection of degree $d$.  A \defbold{cover} of $S$ is a finite set $\{S_{k} \mid 1 \le k \le m\}$ of multisections of degree $d$, each of which consists of restrictions of elements of $S$,  such that for each $k$ and $e \in \mc{E}^*(S)$ there is exactly one $e_k \in \mc{E}^*(S_{k})$ with $e_k \le e$; and every element of $S$ is a join of elements of $\bigcup^m_{k=1}S_{k}$.
\end{defn}

The basic scheme follows Matui's ideas (\cite{MatuiZ}) for $M=\mc{BI}(\mathbb{Z})$, but there are some variations that need to be made for technical reasons. 
We would like to find some compact generating set $K$ for $\Al(M)$ such that every 3-section $F$ can be covered by finitely many 3-sections $F_1, \dots, F_n$ satisfying $\Alt[F_i]\leq \langle K \rangle$ and then show that $\Alt[F]\leq \langle \Alt[F_i]: 1\leq i\leq n\rangle$.
This last part is true if all $F_i$ are pairwise disjoint, and this happens for expansive $\mathbb{Z}$-actions. 
Unfortunately, for other Boolean inverse monoids, we cannot always guarantee that the $F_i$ will be pairwise disjoint, and if $F_1, F_2$ intersect then $\langle \Alt[F_1], \Alt[F_2]\rangle$ does not contain $\Alt[F_1\vee F_2]$:
We can view $\langle \Alt[F_1], \Alt[F_2]\rangle$ as $\langle (\pi, \pi, 1), (1, \pi, \pi): \pi \in \Alt(3)\rangle$ where the first and third coordinates correspond to $F_1\setminus F_2$ and $F_2\setminus F_1$, respectively, and the middle coordinate to $F_1\cap F_2$. 
Clearly, this subgroup of $\Alt(3)\times \Alt(3)\times \Alt(3)$ does not contain the non-trivial diagonal elements, which correspond to the elements of $\Alt[F_1\vee F_2]$.

However, this does work for 5-sections in place of 3-sections, as $\Alt(5)$ is perfect, and $\langle (\pi, \pi, 1), (1, \pi, \pi): \pi \in \Alt(5)\rangle$ contains $[(\pi, \pi, 1), (1, \sigma, \sigma)]=(1, [\pi,\sigma], 1)$. 
The next lemma is essentially \cite[Proposition~3.3]{Nekra}.

\begin{lem}\label{lem:multisection_cover}
	Let $S$ be a multisection of degree at least $5$ and let $\{S_k \mid 1 \le k \le m\}$ be a cover of $S$. 
	Then $\Alt[S] \le \grp{\Alt[S_k] \mid 1 \le k \le m}$.
\end{lem}

The next result ensures that it will be no problem to restrict attention to $5$-sections.

\begin{prop}[{\cite[Proposition~3.8]{Nekra}}]\label{prop:n-orbits_implies_An=A}
	Let $M\leq \PHomeo_c(X)$ and suppose that every  $M$-orbit has at least $n$ points. 
	Then $\Al(M)=\Al_d(M)$ for $3\leq d\leq n$. 
\end{prop}
\begin{proof}
	It suffices to show that  $\Al_d(M)\leq \Al_{d+1}(M)$ for $3\leq d\leq n-1$. 
	Let $S$ be a $d$-section and $x\in e\in \mc{E}^*(S)$. 
	Suppose that $f_1, \dots, f_{d}\in S$ are the elements with domain $e$. 
	By hypothesis, there is some $f_{d+1}\in M$ such that $f_1(x),\dots, f_d(x), f_{d+1}(x)$ are pairwise distinct.   
	Take $x \in e_x \in \mc{CO}(X)$ such that $e_x \le f^*_{d+1}f_{d+1}e$ and let $S_x$ be the subsemigroup generated by $\{f_ie_x: i=1,\dots, d+1\}$. 
	Then $S_x$ is a ($d+1$)-section containing the restriction of $S$ to  $e_x$. 
	Repeating this for every $x\in e$ we obtain a cover of $e$ by clopens $\{e_x\}_{x\in e}$.
	As $e$ is compact, finitely many of these $e_x$ suffice, say $e_1, \dots, e_k$, and we can assume that they are disjoint.
	This yields finitely many $(d+1)$-sections $S_1, \dots, S_k$,   all of whose non-0 idempotents are pairwise disjoint, each containing the restriction of $S$ to $\mc{E}^*(S_i)$ and such that for every $f\in S$ we have $f=\bigvee_{i=1}^k f_i$ where $f_i\in S_i$. 
	This implies that $\Alt[S]\leq \langle \Alt[S_i]: 1\leq i \leq k\rangle \leq \Alt_{d+1}(M)$ and therefore $\Al_d(M)\leq \Al_{d+1}(M)$. 
\end{proof}

So the new goal is to show that there is some set $K\subset \Al_5(M)$ with compact closure such that every 5-section $F$ can be covered by finitely many 5-sections $F_1, \dots, F_n$ with $\Alt[F_i]\leq \langle K\rangle$ for $1\leq i\leq n$.
Then $\Alt[F]\leq \langle K\rangle$, by Lemma \ref{lem:multisection_cover}.

The proof is essentially that of \cite[Theorem 5.10]{Nekra}, translated into the language of inverse monoids and making adjustments for the topology. 

\begin{lem}[{\cite[Proposition~3.5]{Nekra}}]\label{lem:multisection_combine}
	Let $S_1$, $S_2$ be  multisections of $M$ of degrees at least 3.
	Suppose that the intersection $e$ of the supports of $S_1$ and $S_2$ is equal to $f_1f_2$ for $f_i \in \mc{E}^*(S_i)$. 
	Let $S_3$ be the inverse subsemigroup of $M$ generated by $S_1f_1f_2 \cup S_2f_1f_2$. 
	Then $S_3$ is a multisection and $\Alt[S_3] \le \grp{\Alt[S_1] \cup \Alt[S_2]}$.
\end{lem}

\begin{lem}\label{lem:separating}
	Let $M$ be a compactly generated topological Boolean inverse monoid with $\mc{CO}(X) = \mc{E}(M)$.  Suppose every $M$-orbit on $X$ has at least $n$ points for $n \ge 2$.  Then there is a compact symmetric generating set $A$ of $M$ and a clopen partition  $P \subseteq \mc{E}(M)$  of $X$ satisfying:
	\begin{enumerate}
		\item every $M$-orbit  meets at  least  $n$ parts of $P$,
		\item every $a\in A$ has domain and range in different parts of $P$,
		\item for every part $e\in P$ there are $a_1, \dots, a_{n-1}\in A$ with domain $e$ and ranges in pairwise different parts of $P\setminus \{e\}$,
		\item the set $\mc{E}(A^\infty)$ contains a base of topology of $X$ that is a union of partitions.
	\end{enumerate}
\end{lem}

\begin{proof}
	Fix $x=x_0\in X$ and choose $x_1,\dots,x_{n-1}$ other points in the $M$-orbit of $x_0$. 
	Pick any partition $P_x$ that separates $x_0,\dots,x_{n-1}$.
	Suppose that $m_{x,1}, \dots, m_{x, n-1}\in M$ take $x_0$ to $x_1, \dots, x_{n-1}$, respectively. 
	Restricting domains if necessary, we can assume that all $m_{x,i}$ have a common domain, in the part $e_x$ of $P_x$ containing $x_0=x$, and that the range of each $m_{x,i}$ is in the part of $P_x$ containing $x_i$. 
	Doing this for each $x\in X$ yields a clopen cover $X=\bigcup_{x\in X}e_x$ of $X$. 
	As $X$ is compact, it can be covered by finitely many of the $e_{x}$, say $e_1,\dots, e_k$ with corresponding partitions $P_1,\dots, P_k$ and elements $m_{1,1},\dots, m_{1,n-1}, \dots, m_{k,1},\dots, m_{k,n-1}$.
	Let $Q$ be the partition generated by $P_1, \dots,P_k$ and denote by $W$ the set of (non-0) restrictions of the $m_{i,j}$ to parts of $Q$.

	By Proposition~\ref{prop:subshift_expansive}, there is a generating set of $M$ of the form $S \cup Q'$ where $S$ is a compact identity neighbourhood in $M^\times$ and $Q' \subseteq \mc{E}(M)$ is a partition, such that the $\grp{S}$-translates of $Q'$ separate points in $X$.  We may assume that $Q'$ refines $Q$ and that $S = S^*$.
	Then the set $\{ses^*\mid s\in S, e\in Q'\}$ is finite because: $S$ is compact, for each $e\in \mc{E}(M)$  the map $m\mapsto mem^*$ is continuous with discrete image, and $Q'$ is finite.
	Let $P$ be the partition generated by $Q'$ and $\{ses^*\mid s\in S, e\in Q'\}$. 
	We  will replace each element of $S$ by at most  $2|P|$ elements of $M$ to produce the required generating set of $M$.
	
	Given $s\in S, e\in P$ there are unique $f_0, f_1\in Q$ such that  $e\leq f_0$ and $ses^*\leq f_1$. 
	If $f_0\neq f_1$ put $s_{e,0}:=s_{e,1}:=se$. 
	If $f_0=f_1$ then, by the construction of $Q$, there is some $m_{f_0, 1}\in W$ with domain $f_0$ and taking $f_0$ to a different part of $Q$. 
	Put $s_{e,0}:=m_{f_0,1}e$ and $s_{e,1}:=sem_{f_0, 1}^*$ so that $s_{e,1}s_{e,0}=sem_{f_0, 1}^*m_{f_0, 1}e=se$. 
	By the choice of $m_{f_0, 1}$, the domain and range of $s_{e,0}$ (respectively, $s_{e,1}$) lie in different parts of $Q$.
	Each $s_{e,i}$ depends continuously on $s$ for each $(e,i)\in P\times \{0,1\}$, so $S_1:=\{s_{e,i}\mid s\in S, i=0,1\}$ is compact.
	Also, $S_2:=\{we \mid e\in P, w\in W\} \setminus \{0\}$ is finite. 
	Thus $A:=S_1\cup S_2 \cup S^*_1 \cup S^*_2$ is a compact symmetric generating set of $M$ that satisfies the first three conditions.
	
	Finally, note that we have ensured $se \in A^\infty$ for all $s \in S$ and $e \in P$; it follows from Proposition~\ref{prop:subshift_expansive} that $\mc{E}(A^\infty)$ contains a base of topology that is a union of partitions.
\end{proof}

We say that a subset of a topological inverse monoid $M$ is \defbold{bounded} if it has compact closure.

\begin{lem}\label{lem:bounded_multisections}
	Let $\mc{S}$ be a set of multisections of a topological Boolean inverse monoid $M$ such that $\bigcup_{S \in \mc{S}}S$ is bounded in $M$. 
	Then $\bigcup_{S \in \mc{S}}\Sym[S]$ and $\bigcup_{S \in \mc{S}}\Alt[S]$ are bounded.
\end{lem}

\begin{proof}
	We only prove the statement for $\Sym[S]$, the other one being similar. 
	Let $T = \bigcup_{S \in \mc{S}}S$. 
	The fact that $T$ is bounded ensures that only finitely many domains and ranges occur, so the degree of multisections in $\mc{S}$ is bounded.
	
	Given $S \in \mc{S}$, every transposition in $\Sym[S]$ is of the form
	\[
	\tilde{m} := m \vee m^* \vee (m^*m)^\perp(mm^*)^\perp
	\]
	for some $m \in S$.  We see that $\tilde{m}$ depends continuously on $m$; thus the set $T' = \{\tilde{m} \mid m \in T\}$ is a bounded set that accounts for all the transpositions.  In turn, since $\sup\{\deg(S) \mid S \in \mc{S}\} < \infty$, there is a product of finitely many copies of $T'$ that contains every element of $\bigcup_{S \in \mc{S}}\Sym[S]$.
	We conclude that $\bigcup_{S \in \mc{S}}\Sym[S]$ is bounded.
\end{proof}

We now take a finite partition $P$ of $X$ and a compact symmetric generating set $A$ of $M$ satisfying the four conditions in Lemma~\ref{lem:separating} for $n=5$.

Define $T:=\{m\in \bigcup^3_{k=1}A^k \mid P \text{ separates } m^*m, mm^*\}$. 
The reason for this choice of generating set will be made apparent in the proof of Lemma \ref{lem:alternating_fg_lengthreduction}.
Note that $A\subseteq T$ and that $T$ is bounded. 
We now produce a set of 5-sections from $T$:
For every $m\in T$ and $e\in P$ such that $me\neq 0$, the domain $m^*me$ and range $mem^*$ of $me$ are in different parts of $P$. 
By part (3) of Lemma \ref{lem:separating}, there are $a_2, a_3, a_4\in A$ with domain $e$ and ranges in parts of $P$ different from $e$ and the one containing $me$. 
Denote by $F_{m,e}$ the 5-section generated by $\{me, a_im^*m: i=2,3,4\}$.
Let $\mc{F}:= \{F_{m,e}\mid m\in T, e\in P, me\neq 0\}$. 
Then 
$$\bigcup_{F\in \mc{F}}F\subset T\cup AT^*T\cup AT^*\cup TA^*$$
which is a bounded set, because $A$ and $T$ are bounded and multiplication and inversion are continuous maps.
By Lemma \ref{lem:bounded_multisections}, the set 
$$K:=\overline{\bigcup_{F\in \mc{F}}\Alt[F]}$$
is compact. 
It is now enough to show that $\Al(M) \le \langle K\rangle$. 

For this, we first prove

\begin{lem}[{\cite[Lemmas~5.12 and 5.13]{Nekra}}]\label{lem:alternating_fg_lengthreduction}
	Let $m \in M$ and suppose that $mx\neq x$ for some $x\in X$. 
	Then there is a 5-section $F\leq M$  with $\Alt[F]\leq \grp{K}$ and some idempotent $e\leq m^*m$ containing $x$ 
	such that $me\in F$ and such that
	\begin{enumerate}[(i)]
		\item if $e$ and $mem^*$  are in different parts of $P$, then all idempotents of $F$ are separated by $P$;
		\item if $e$ and $mem^*$ are in the same part of $P$, then all other idempotents of $F$ are separated by $P$. 
	\end{enumerate} 
\end{lem}

\begin{proof}
	Since, by Lemma \ref{lem:bim_generating_set}, every element of $M$ can be written as a disjoint join of elements of $A^{\infty} \setminus \{1\}$, it suffices to prove the lemma for $m \in A^{\infty} \setminus \{1\}$ and then apply Lemma \ref{lem:multisection_cover}.
	We prove each item in turn. 
	
	The statement (i) is proved by induction on the length of $m$ as a word over $A$.
	Note that item (2) of Lemma \ref{lem:separating} guarantees that the domain and range of $m$ are each contained in a single part of $P$, say $m^*m \le d_m \in P$ and $mm^* \le r_m \in P$, so either $d_m=r_m$ or $P$ separates $m^*m$ and $mm^*$.
	If the latter holds and $m$ has length at most 3, then the statement holds by the choice of $K$. 
	Suppose then that (i) is true for all elements of $A^{\infty}$ of length at most $n$, for some $n\geq3$, and suppose that $m=a_0a_1\cdots a_n$ for $a_i \in A$ and that $d_m \neq r_m$.
	By Lemma \ref{lem:separating}, there is some $a\in A$ with domain $d_m$ and range disjoint from $d_m \vee r_m$. 
	Then $a_0a_1a^*\in T$, so, by the choice of $K$, there is some 5-section $G$ containing $a_0a_1a^*$, with $P$-separated idempotents and $\Alt[G]\subset K$.
	Pick some sub-section $G_1$ of $G$ of degree 3, containing $a_0a_1a^*$.   
	As $m':= aa_2\cdots a_n\in A^{n}$ has $P$-separated domain and range, the induction hypothesis gives a 5-section $H$ and an idempotent $e$ with $x \in e$, such that (i) is satisfied with respect to $m'$. 
	Let $H_1$ be a sub-section of $H$ of degree 3, containing $m'e$ and having as third non-0 idempotent something disjoint from the idempotents of $G_1$. 
	Thus the idempotents of $G_1$ and $H_1$ only intersect in the intersection of the domain of $a_0a_1a^*$ with the range of $aa_2\cdots a_n$, so Lemma \ref{lem:multisection_combine} gives a 5-section $F$ whose idempotents are all $P$-separated, such that $m'' \in F$ where $m'' = a_0a_1a^*aa_2\cdots a_ne$ and $\Alt[F]\leq \grp{\Alt[G_1], \Alt[H_1]}\leq \grp{\Alt[G],\Alt[H]}\leq \grp{K}$.
	Note that by construction, we have ensured that $m'' = me'$ for an idempotent $e'$ such that $x \in e'$.

	For (ii), suppose that $d_m = r_m$.
	By Lemma \ref{lem:separating} there is some $a\in A$ with domain $d_m$ and range in a different part $r'$ of $P$.
	Then $ma^* \in A^{\infty}$ has domain and range in different parts of $P$, so the previous paragraph yields a 5-section $G$ containing a restriction $ma^*e$ of $ma^*$ with the point $ax$ in the domain of $ma^*e$, such that $\Alt[G]\leq \grp{K}$ and all idempotents of $G$ are $P$-separated. 
	By the choice of $K$, there is also a 5-section $H\ni a$ with $\Alt[H]\leq \grp{K}$ and $P$-separated idempotents. 
	Now, as in the proof of (i), we can choose sub-sections $G_1\leq G$ and $H_1\leq H$ of degree 3 such that $G_1, H_1$ only intersect in the intersection of the domains of $a^*$ and $ma^*e$. 
	Lemma \ref{lem:multisection_combine} then yields a 5-section $F$ that contains $ma^*ea$ and satisfies all the conditions of (ii).
\end{proof}

We now complete the proof of Theorem~\ref{thm:expansive_generation}. 
Let $N\subset M$ be a 5-section, generated by $m_1, m_2, m_3, m_4$ where all $m_i$ have common domain $e$. 
Focus first on the sub-3-section generated by $m_1, m_2$. 
Pick a point $x\in e$. 
By Lemma \ref{lem:alternating_fg_lengthreduction},  there are 5-sections $S_1, S_2\leq M$ and idempotents $e_1, e_2\leq e$ containing $x$ such that $m_ie_i\in S_i$  and $\Alt[S_i]\leq \grp{K}$ for $i=1,2$.
Moreover, by the separation conditions on the idempotents of $S_1$ and $S_2$ there exist $s_i\in S_i$ with domain $e_i\leq e$ ($i=1,2$) and ranges in different parts of $P$. 
In particular, the ranges of $m_1e_1, m_2e_2, s_1, s_2$ are all pairwise disjoint, so that their restrictions to $e_1e_2$ generate a 5-section $F_x$. 
By Lemma \ref{lem:multisection_combine}, $\Alt(F_x)\leq \langle \Alt(S_1), \Alt(S_2)\rangle \leq \grp{K}$. 
Thus, for every point $x\in e$ there is some $d_x\leq e$ containing $x$ and a 5-section $F_x$ such that $m_1d_x, m_2d_x\in F_x$ and $\Alt(F_x)\leq \grp{K}$. 
Since $e$ is compact, it can be covered by finitely many of the $d_x$; say, $d_1, \dots, d_n$ for some $n\in \Nb$, with corresponding 5-sections $F_i$,
so $m_i$ is the (not necessarily disjoint) join $\bigvee_{j=1}^nm_id_j$ for $i=1,2$ and we have $\Alt(F_1), \dots, \Alt(F_n) \leq \grp{K}$. 

Doing the same for the sub-3-section generated by $m_3, m_4$ we obtain 5-sections $G_1,\dots, G_l\leq M$ for $l\in \Nb$ and a cover $e=\bigvee_{j=1}^lf_j$ where $m_3f_j, m_4f_j\in G_j$ and $\Alt(G_j)\leq \grp{K}$. 

Now, for each pair $(i,j)\in \{1, \dots, n\}\times \{1,\dots, l\}$, the elements $\{m_1d_i, m_2d_i, m_3f_j, m_4f_j\}$ all have domain containing $d_if_j$ and, by Lemma \ref{lem:multisection_combine}, their restrictions to $d_if_j$ generate a 5-section $H_{i,j}$ satisfying $\Alt(H_{i,j})\leq \langle \Alt(F_i), \Alt(G_j)\rangle\leq \grp{K}$. 
Since the $H_{i,j}$ form a cover of $N$, Lemma \ref{lem:multisection_cover} then yields that $\Alt(F)\leq \grp{\Alt(H_{i,j}): (i,j)}\leq \grp{K}$, as required. \hfill \qedsymbol

\

Returning to the purely group-theoretic context, we now obtain the desired relationship between a locally decomposable, minimal and expansive action of a compactly generated \tdlc group $G$ and compact generation of the associated group $\ol{\Al(G)}$.

\begin{cor}\label{cor:group_expansive_iff_alternating_compactly_gen}
	Let $X$ be a compact zero-dimensional space and let $G\leq \Homeo(X)$ be a \tdlc group.
	Suppose that the action is locally decomposable and minimal.  Say $H \le G$ is piecewise dense if $\Full(H) = \Full(G)$.
	Then the following are equivalent:
	\begin{enumerate}[(i)]
		\item there is a compactly generated piecewise dense open subgroup $H$ of $G$ that acts expansively on $X$;
		\item there is a compactly generated piecewise dense subgroup $H$ of $\Full(G)$ that acts expansively on $X$;
%		\item the group $\Al(G)$ is compactly generated;
		\item the group $\ol{\Al(G)}$ is compactly generated.
	\end{enumerate}
\end{cor}	
\begin{proof}
	We extend the topology of $G$ to $\Full(G;X)$ using Proposition~\ref{prop:full_contains_intersection_rists}.  The implication (i) $\Rightarrow$ (ii) is immediate.
	
	Suppose that (ii) holds. 
	Then there is a compactly generated piecewise dense subgroup $H$ of $\Full(G)$ that is expansive on $X$; by replacing $H$ with $\grp{H,U}$ for $U$ a compact identity neighbourhood in $G$, we may assume $H$ is open in $G$.
	We then see that $H$ is  locally decomposable and minimal on $X$, and $\Full(G) = \Full(H)$.
	It follows that $\Al(H) = \Al(G)$, while $\ol{\Al(H)}$ is compactly generated by Proposition~\ref{prop:subshift_expansive}(ii) and Theorem~\ref{thm:expansive_generation}, proving (iii).
	
%	Since  $\ol{\Al(G)}$ is locally compact, it is clear that (iii) implies (iv).
	
	Finally, suppose that $\ol{\Al(G)}$ is compactly generated. 
	By Lemma~\ref{lem:A(G)_quotientspace}, the minimality of $G$ ensures that $\Al(G)$ is expansive, so a fortiori, $\ol{\Al(G)}$ is expansive. 
	Lemma \ref{lem:alternating_to_full} and Proposition~\ref{prop:subshift_expansive}(ii) ensure that $\mc{BI}(\ol{\Al(G)})=\mc{BI}(G)$ is compactly generated.  Proposition~\ref{prop:subshift_expansive}(i) then provides a compactly generated piecewise dense open subgroup $H$ of $G$ that acts expansively on $X$, so (iii) implies (i), completing the cycle of implications.
\end{proof}

We do not know in general what the relation is between compact generation of $\ol{\Al(G)}$ and of $\Full(G)$.

\section{Many piecewise full groups are simple-by-(discrete abelian)}\label{sec:A_is_open}

We are interested in the class $\ms{S}$ of \tdlc groups that are compactly generated, topologically simple and non-discrete.  Theorems \ref{thm:Nekrashevych_simple} and \ref{thm:expansive_generation} and Proposition \ref{prop:piecewise_full_extension:lcsc} together give conditions for $\overline{\Al(G)}$ to be in $\ms{S}$, starting from a locally decomposable group (or more generally, inverse monoid) $G$ acting on the Cantor space.  However, this naturally raises the questions of how far the group in $\ms{S}$ we have constructed is from being piecewise full (or in other words, how large the quotient $\Full(G)/\overline{\Al(G)}$ can be), and whether or not it contains the group $G$ we started with (or at least a large subgroup of it).  In this section we will in fact see that under the conditions of the previous theorems and assuming $G$ is not discrete, then actually $\Full(G)/\Al(G)$ is discrete and abelian; in other words, $\Al(G) = \Der(\Full(G))$ and $\Al(G)$ contains an open subgroup of $G$.  This also means that $\Al(G) = \ol{\Al(G)}$, which means we are obtaining an abstractly simple group in $\ms{S}$, and for practical examples the definition of $\Al(G)$ is more intuitive than that of $\ol{\Al(G)}$.
Before this, we recall some necessary preliminaries on \tdlc groups that are also required for later sections.

\subsection{Quasicentre and local decomposition lattice}\label{ssec:QZ_LD}

\begin{defn}
	Let $G$ be a topological group.  The \defbold{quasicentre} $\QZ(G)$ is the set of elements $g \in G$ such that $\CC_G(g)$ is open.  An action of $G$ on a locally compact zero-dimensional space $X$ is \defbold{non-discretely micro-supported} if for every $\alpha \in \mc{CO}(X) \setminus \{\emptyset\}$, the kernel of the action is not open in $\rist_G(\alpha)$.
\end{defn}

The next lemma makes some basic observations on the quasicentre, which will be used without further comment.
Recall that a \defbold{locally normal} subgroup of a topological group is one whose normaliser is open. 

\begin{lem}
	Let $G$ be a topological group and let $K$ be a discrete locally normal subgroup.  Then $K \le \QZ(G)$ and $\QZ(\N_G(K)/K) = \QZ(\N_G(K))/K$.  In particular, if $\QZ(G)$ is discrete, then it is the largest discrete normal subgroup of $G$, and $\QZ(G/\QZ(G)) = \triv$.
\end{lem}

\begin{proof}
	Since $\N_G(K)$ is open in $G$, we see that $\QZ(\N_G(K)) \le \QZ(G)$.  From now on we may assume without loss of generality that $G = \N_G(K)$.
	
	The image in $G/K$ of any open subgroup of $G$ is open, from which it is clear that $\QZ(G)K/K \le \QZ(G/K)$.  On the other hand, if $g \in G$ is such that $gK \in \QZ(G/K)$, then there is an open subgroup $H$ of $G$ such that $H \ge K$ and $hgh\inv \in gK$ for every $h \in H$.  In particular, the $H$-conjugacy class of $g$ is discrete; since $H$ acts continuously by conjugation, it follows that $\CC_H(g)$ is open and hence $g \in \QZ(H) \le \QZ(G)$.  In particular, we see that $K \le \QZ(G)$ and $\QZ(G)/K \ge \QZ(G/K)$.
	
	The final sentence of the lemma is now clear as a special case.
\end{proof}

\begin{lem}\label{lem:non-discrete_condition}
	Let $G$ be a non-trivial topological group acting continuously by homeomorphisms on a compact zero-dimensional space $X$.  Suppose that the action is faithful, locally decomposable and minimal.  Then the following are equivalent:
	\begin{enumerate}[(i)]
		\item The action is non-discretely micro-supported;
		\item $G$ has trivial quasicentre;
		\item $G$ is not discrete.
	\end{enumerate}
\end{lem}

\begin{proof}
	Suppose $G$ is non-discretely micro-supported, and let $g \in G \setminus \triv$.  Since $G$ acts faithfully there is some $\alpha \in \mc{CO}(X) \setminus \{\emptyset\}$ such that $\alpha \cap g\alpha = \emptyset$.  In particular, we see that $\CC_{\rist_G(\alpha)}(g)$ is trivial, since $g$ does not preserve the support of any non-trivial element of $\rist_G(\alpha)$.  Since $\rist_G(\alpha)$ is not discrete, it follows that $\CC_G(g)$ is not open, that is, $g \not\in \QZ(G)$.  This proves that $\QZ(G)=\triv$, showing that (i) implies (ii).
	
	It is clear that (ii) implies (iii).  Now suppose that $G$ is not discrete and let $\alpha \in \mc{CO}(X) \setminus \{\emptyset\}$.  Since $G$ acts minimally there exist $g_1,\dots,g_n \in G$ such that $X = \bigcup^n_{i=1}g_i\alpha$.  We can thus find a partition $\mc{P} = \{\beta_1,\dots,\beta_n\}$ of $X$ into clopen sets such that $\beta_i \subseteq g_i\alpha$.  Since the action is locally decomposable, the group
	\[
	H = \prod^n_{i=1}\rist_G(\beta_i)
	\]
	is open in $G$.  Since $G$ is not discrete, $H$ is not discrete, so for some $1 \le i \le n$ the factor $\rist_G(\beta_i)$ is non-discrete.  We then have $g\inv_i(\rist_G(\beta_i)g_i \le \rist_G(\alpha)$, so $\rist_G(\alpha)$ is non-discrete.  Since $\alpha \in \mc{CO}(X) \setminus \{\emptyset\}$ was arbitrary we deduce that the action is non-discretely micro-supported, showing that (iii) implies (i).
\end{proof}

Some useful tools for studying \tdlc groups become available under the assumption that the quasicentre is discrete, one of which is the decomposition lattice.

\begin{defn}\label{defn:ldlat}
	Let $G$ be a \tdlc group such that $\QZ(G)$ is discrete and let $U$ be a compact open subgroup of $G$.  A \defbold{locally direct factor} of $G$ is a closed subgroup $K$ such that $G$ has an open subgroup of the form $K \times L$.  Two locally direct factors $K_1$ and $K_2$ are equivalent if their intersection is open in both.
	The \defbold{(local) decomposition lattice} $\ldlat(G)$ of $G$ is then the set of equivalence classes of locally direct factors, with partial order induced by inclusion of subgroups.
	
	The group $G$ acts on $\ldlat(G)$ by conjugation. 	
	Say that $G$ is \defbold{faithful locally decomposable} if this action is faithful.
\end{defn}

As long as $\QZ(G)$ is discrete, then $\ldlat(G)$ is a Boolean algebra; see \cite[Theorem~4.5]{CRW-Part1}.  Note also that $\ldlat(G)$ is determined by the isomorphism type of any identity neighbourhood, so we have $\ldlat(G) \cong \ldlat(G/\QZ(G))$.  The Stone space of $\ldlat(G)$ is then a compact zero-dimensional space admitting a canonical action of $G$ by homeomorphisms.
 The kernel of this action always contains $\QZ(G)$, but can be larger in general.

The faithful locally decomposable terminology is justified by the following.

\begin{lem}[{See \cite[Theorem~5.18]{CRW-Part1}}]\label{lem:loc_decomp:ldlat}
	Let $G$ be a non-trivial \tdlc group with trivial quasicentre.
	\begin{enumerate}[(i)]
		\item Let $\mc{A}$ be a $G$-invariant subalgebra of $\ldlat(G)$ on which $G$ acts faithfully.  Then the action of $G$ on $X = \mf{S}(\mc{A})$ is locally decomposable.
		\item Let $X$ be a compact zero-dimensional space equipped with a faithful locally decomposable action of $G$.  Then $\mc{CO}(X)$ is $G$-equivariantly isomorphic to a unique subalgebra $\mc{A}$ of $\ldlat(G)$.
	\end{enumerate}
\end{lem}

\subsection{Robustly monolithic groups and compressible actions}\label{sec:robust}

The \defbold{monolith} $\Mon(G)$ of a group $G$ is the intersection of all non-trivial normal subgroups of $G$. 
If $G$ is a topological group, then $\TMon(G)$ is the intersection of all non-trivial closed normal subgroups of $G$. 
Note that if $\Mon(G)$ is non-trivial, then $\TMon(G)$ is the closure of $\Mon(G)$.

\begin{defn}\label{defn:compressible}
	Let $G$ be a group acting on a locally compact space $X$.  Say a subset $Y$ of $X$ is \defbold{$G$-compressible} if every nonempty open subset of $X$ contains a $G$-translate of $Y$ as a non-dense subset.  We say the action of $G$ is \defbold{compressible} if a nonempty compressible open set exists for the action.
	The action of $G$ on $X$ is \defbold{fully compressible} if $X$ is compact and every non-dense open subset of $X$ is $G$-compressible.
\end{defn}
Of course, both properties are inherited by overgroups. 
\begin{rmk}
	Suppose that $X$ is locally compact and zero-dimensional. 
	If $G$ acts compressibly on $X$ then we may assume without loss of generality that the compressible subset $Y$ is clopen. 
	If $G$ is a topological group acting continuously on $X$ (in particular, the topology of $G$ refines the compact-open topology) and $H\leq G$ is dense then the action of $G$ is compressible if and only if the action of $H$ is compressible. 
\end{rmk}

\begin{defn}\label{defn:expansive_group}
	A \tdlc group $G$ is \defbold{expansive} (acting on itself) if and only if it has an identity neighbourhood $U$ such that $\bigcap_{g \in G}gUg\inv=\triv$.
	It is  \defbold{regionally expansive} if it has a compactly generated subgroup $H$ that is expansive. 
	
	Denote by $\ms{R}$ the class of \defbold{robustly monolithic groups}: those \tdlc groups $G$ such that $\TMon(G)$ is non-discrete, topologically simple and regionally expansive.  
	Note that $\ms{R}$ contains the class $\ms{S}$ of \tdlc groups $G$ such that $G$ is non-discrete, topologically simple and compactly generated.  
		
		Given a class $\mc{C}$ of \tdlc groups, we write $\mc{C}_{\ldlat}$ for the class of groups in $\mc{C}$ that are faithful locally decomposable.
\end{defn}

\begin{lem}[{See \cite[Section~5.1]{CRW-Part1}}]\label{lem:centraliser_lattice}
	Let $G$ be a \tdlc group.
	\begin{enumerate}[(i)]
		\item Suppose that $G$ has no non-trivial abelian locally normal subgroup.  Then the following is a Boolean algebra:
		\[
		\mathrm{LC}(G):= \{ \CC_G(K) \mid K \le G \text{ locally normal}\};
		\]
		write $\Omega_G$ for the Stone space of $\mathrm{LC}(G)$, equipped with the action induced by conjugation on $\mathrm{LC}(G)$.  The set of elements $K$ of $\mathrm{LC}(G)$ such that $K \times \CC_G(K)$ is open in $G$ form a subalgebra $\mathrm{LD}(G)$, which is $G$-equivariantly isomorphic to $\ldlat(G)$.  If $G$ acts faithfully on $\Omega_G$, then the action is micro-supported.
		\item If $G$ has no non-trivial abelian locally normal subgroup and $G$ is an open subgroup of a \tdlc group $H$ such that $\QZ(H) =\triv$, then $H$ has no non-trivial abelian locally normal subgroup, and $\Omega_G \cong \Omega_H$ as topological spaces equipped with a $G$-action.
		\item If $X$ is a compact zero-dimensional $G$-space such that the action is faithful and micro-supported, then $G$ has no non-trivial abelian locally normal subgroup and $X$ is isomorphic as a $G$-space to a quotient of $\Omega_G$.
	\end{enumerate}
\end{lem}
	
The argument of the next lemma is taken from a preprint version of \cite{CapraceLeBoudec}, although it is not stated exactly in this form in the final article.

\begin{lem}\label{lem:centraliser_lattice_minimal}
	Let $G$ be a \tdlc group with no non-trivial abelian locally normal subgroup.  Suppose $G$ admits a faithful minimal micro-supported action on a compact zero-dimensional space $X$.  Then the action of $G$ on $\Omega_G$ is faithful, minimal and micro-supported.  
 If moreover the action of $G$ on $X$ is fully compressible, then so is the action on $\Omega_G$.
\end{lem}

\begin{proof}
	By Lemma~\ref{lem:centraliser_lattice}(iii) we have a $G$-equivariant quotient map $\pi: \Omega_G \rightarrow X$.  Clearly $G$ acts faithfully on $\Omega_G$, and by Lemma~\ref{lem:centraliser_lattice}(i), this action is micro-supported.
	
	Let $O$ be a nonempty open subset of $\Omega_G$.  The subgroup $\rist_G(O)$ is non-trivial; take $g \in \rist_G(O) \setminus \{1\}$.  Since $g$ acts faithfully on $X$, there is some $x \in X$ such that $gx \neq x$, and hence $g\pi\inv(x)$ is disjoint from $\pi\inv(x)$.  Since $g$ fixes $\Omega_G \setminus O$ pointwise, we must have $\pi\inv(x) \subseteq O$.
	
	Now suppose $O$ is $G$-invariant.  Then $\pi(\Omega_G \setminus O)$ is a proper closed $G$-invariant subset of $X$; since $G$ acts minimally on $X$, we deduce that $O = \Omega_G$, showing that $\Omega_G$ is a minimal $G$-space.
	
	For the rest of the proof we suppose that the action of $G$ on $X$ is fully compressible.  Given proper non-empty clopen subsets $O_1,O_2 \subseteq \Omega_G$, then $\pi(O_1)$ and $\pi(\Omega_G \setminus O_2)$ are proper non-empty closed subsets of $X$, so $\pi(O_1)$ is contained in some non-dense open subset $Y$ of $X$.  There is then $g \in G$ such that $gY$ is contained in the complement of $\pi(\Omega_G \setminus O_2)$, and hence $gO_1$ is disjoint from $\Omega_G \setminus O_2$, in other words, $gO_1 \subseteq O_2$, showing that the action of $G$ on $\Omega_G$ is fully compressible.
\end{proof}

\begin{lem}\label{lem:simple_dynamics}
	Let $G \in \ms{R}$.
	\begin{enumerate}[(i)]
		\item \textup{(See \cite[Proposition~5.1.2]{CRW-DenseLC})} $G$ is regionally expansive, $\QZ(G)=\triv$ and $G$ has no non-trivial abelian locally normal subgroup.
		\item \textup{(See \cite[Theorem~7.3.3]{CRW-DenseLC})} The action of $\TMon(G)$ on $\Omega_G$, and hence on every quotient $G$-space of $\Omega_G$, is minimal and compressible.  Consequently the action of $G$ on any non-trivial quotient $G$-space of $\Omega_G$ is faithful.
	\end{enumerate}
\end{lem}

\begin{thm}[{\cite[Theorem 1.13]{GR_compressible}}]\label{thm:loc_decomp_compressible}
	Let $G$ be a \tdlc group acting faithfully, minimally and locally decomposably on a compact zero-dimensional space $X$, and let $A$ be a normal subgroup of $G$ (not assumed to be closed).  Suppose that the action of $A$ on $X$ is compressible.  Then $A$ is open in $G$ and the action of $A$ on $X$ is minimal and locally decomposable.
\end{thm}

\begin{prop}[{\cite[Corollary 1.10]{GR_compressible}}]\label{prop:Matui}
	Let $X$ be a infinite zero-dimensional Hausdorff space and let $G$ be the piecewise full group of some group of homeomorphisms of $X$.  Suppose that the action of $G$ is compressible and let $H$ be the group generated by rigid stabilisers of $G$-compressible clopen sets.   Then $\Mon(G)=\Der(H)$ is non-abelian and simple.
	If $X$ is compact and $G$ acts minimally, then $G = H$ and the action of $\Der(G)$ is fully compressible.
\end{prop}

	\begin{cor}\label{cor:A_is_open}
		Let $X$ be an infinite compact zero-dimensional Hausdorff space and let $G$ be a \tdlc group acting faithfully, minimally and locally decomposably on $X$, such that $G = \Full(G;X)$.
		\begin{enumerate}[(i)]
			\item If the $G$-action is compressible, then $\Mon(G) = \Der(G) = \Al(G;X)$; moreover, $\Al(G;X)$ is open in $G$ and has fully compressible action on $X$.
			\item If $G$ is non-discrete and $\ol{\Al(G;X)}$ is compactly generated, then $\ol{\Al(G;X)}$ acts compressibly on $X$ (and hence so does $G$); consequently, $\Al(G;X) = \TMon(G)$ is open and belongs to $\ms{S}$. 
			Moreover, every faithful micro-supported action of $\Al(G;X)$ (hence also of $G$) is fully compressible.
		\end{enumerate}
	\end{cor}
	
	\begin{proof}
		We see from Theorem~\ref{thm:Nekrashevych_simple} that $\Al(G;X) = \Mon(G)$, so also $\ol{\Al(G;X)} = \TMon(G)$, and $\Mon(G)$ is simple, so $\TMon(G)$ is topologically simple.
		
		(i)
		We are in the situation of Proposition~\ref{prop:Matui}, with $X$ compact and $G$ acting minimally, so $\Mon(G) = \Der(G)$ and the action of $\Der(G)$ is fully compressible; moreover $\Der(G)$ is non-abelian and simple.  By Theorem~\ref{thm:loc_decomp_compressible}, $\Der(G)$ is open in $G$.
		
		(ii)
		In this case we know by Lemma~\ref{lem:non-discrete_condition} that $G$ has trivial quasicentre; in particular, $\TMon(G)$, being a non-trivial normal subgroup of $G$, cannot be discrete.  We deduce that in fact $G \in \ms{R}$, so $\TMon(G)$ acts compressibly on $X$ by Lemma~\ref{lem:simple_dynamics}(ii).  In particular, by part (i), $A:= \Al(G;X)$ is open, so we can identify $\Omega_G$ with $\Omega_A$ as $A$-spaces, and we have $A \in \ms{S}$, with fully compressible action on $X$.
	  Lemma~\ref{lem:centraliser_lattice_minimal} ensures that in fact $A$ has fully compressible action on $\Omega_A$, and hence by Lemma~\ref{lem:centraliser_lattice}(iii), every faithful micro-supported action of $A$ is fully compressible; similar conclusions also apply to $G$.
	\end{proof}

\section{Examples}\label{sec:examples}

Corollaries~\ref{cor:group_expansive_iff_alternating_compactly_gen} and~\ref{cor:A_is_open} imply that when a \tdlc group $G$ acting by homeomorphisms on a compact zero-dimensional space $X$ satisfies certain additional conditions, then $\Al(G) \in \ms{S}$.
We pause the development of the general theory to discuss a special case.  The constructions we will use in this section are familiar from the prior literature, but serve as a motivating special case for the more general perspective of this article.

	Most of the examples arise as commensurators or germs of automorphisms of certain profinite groups, so we start by recalling some general theory on commensurators.

\subsection{Commensurators or germs of automorphisms}\label{ssec:germs_def}
An important aspect of the theory of \tdlc groups is to relate the local properties of a \tdlc group $G$, that is, properties that can be detected in any neighbourhood of the identity, with global properties of $G$.  The most convenient setting to make this discussion precise is when $\QZ(G)$ is assumed to be discrete; in that case, we can appeal to the notion of the group of germs of automorphisms of $G$, as studied by Barnea--Ershov--Weigel \cite{BEW} and Caprace--De Medts \cite{CapDeM}.

\begin{defn}\label{def:group_of_germs}
	A \defbold{local homomorphism} between two \tdlc groups $G$ and $H$ is a continuous homomorphism $\phi: U \rightarrow H$, where $U$ is an open subgroup of $G$. 
	It is a \defbold{local isomorphism} if $\phi$ restricts to an isomorphism from $U$ to an open subgroup of $H$.  We say $G$ and $H$ are \defbold{locally isomorphic} if a local isomorphism exists.  
	Two local homomorphisms $\phi_1,\phi_2$ are \defbold{equivalent} if both are defined and agree on some group $W$ that is in the domain of both $\phi_1$ and $\phi_2$. 
	The equivalence class $[\phi]$ of a local homomorphism is then called the \defbold{germ} of $\phi$. 
\end{defn}

\begin{thm}[{see \cite{BEW} and \cite{CapDeM}}]\label{thm:AbstractCommensurator}
	Let $G$ be a \tdlc group with discrete quasicentre.  Define the group $\ms{L}(G)$ of \defbold{germs of automorphisms} of $G$ to be the set of germs of local isomorphisms from $G$ to itself.
	\begin{enumerate}[(i)]
		\item The quasicentre $\QZ(G)$ is the unique largest discrete normal subgroup of $G$ and $G/\QZ(G)$ is locally isomorphic to $G$, with trivial quasicentre.
		\item $\ms{L}(G)$ has a group structure induced by composition of local isomorphisms.
		\item Let $\ad: G \rightarrow \ms{L}(G)$ be defined by $\ad(g) = [y \mapsto gyg\inv]$.  Then $\mathrm{ad}$ is a group homomorphism with kernel $\QZ(G)$, and there is a unique group topology on $\ms{L}(G)$ such that $\ad$ is continuous and open.  With respect to this topology, $\QZ(\ms{L}(G)) = \triv$.
		\item Define the topology of $\ms{L}(G)$ as in (iii) and let $\phi_1: U \rightarrow G$ and $\phi_2: U \rightarrow H$ be continuous and open homomorphisms with discrete kernel, where $U$ and $H$ are \tdlc groups.  Then there is a unique continuous and open homomorphism $\theta: H \rightarrow \ms{L}(G)$, with kernel $\QZ(H)$, such that the following diagram commutes:
		\[
		\xymatrixcolsep{3pc}\xymatrix{
			U \ar^{\phi_1}[r] \ar^{\phi_2}[d] & G \ar^{\ad}[d] \\
			H \ar^{\theta}[r] & \ms{L}(G). }
		\]
	\end{enumerate}
\end{thm}

From now on, if $G$ is a \tdlc group with discrete quasicentre, we regard $\ms{L}(G)$ as a topological group equipped with the topology given in Theorem~\ref{thm:AbstractCommensurator}(iii), and via the map $\ad$, we regard $G/\QZ(G)$ (or $G$ itself, if $\QZ(G)=\triv$) as an open subgroup of $\ms{L}(G)$.  Note that given the universal property of the group of germs, we have $\ms{L}(\ms{L}(G)) = \ms{L}(G)$ whenever $G$ is a \tdlc group with discrete quasicentre.

The property of being faithful locally decomposable passes to the group of germs.

\begin{lem}\label{lem:faithful_germ}
	$G$ be a \tdlc group such that $\QZ(G)$ is discrete and is the kernel of the action on $\ldlat(G)$.  Then $\ms{L}(G)$ is faithful locally decomposable.  Hence for every \tdlc group $H$ that is locally isomorphic to $G$, then $H/\QZ(H)$ is faithful locally decomposable.
\end{lem}

\begin{proof}
	Via the map $\ad$, we have an action of $G$ on $\ldlat(\ms{L}(G))$ with kernel $\QZ(G)$.  In particular, $\ad(G)$ acts faithfully on $\ldlat(\ms{L}(G))$, and hence the kernel of the action of $\ms{L}(G)$ on $\ldlat(\ms{L}(G))$ is discrete.  Since $\QZ(\ms{L}(G)) = \triv$, in fact $\ms{L}(G)$ has no non-trivial discrete normal subgroups, so $\ms{L}(G)$ acts faithfully on $\ldlat(\ms{L}(G))$.
	
	Now given a \tdlc group $H$ that is locally isomorphic to $G$, by Theorem~\ref{thm:AbstractCommensurator} we have a natural open embedding of $H/\QZ(H)$ into $\ms{L}(G)$.  We thus obtain a faithful locally decomposable action of $H$ on $\mf{S}(\ldlat(\ms{L}(G)))$, so $H$ is faithful locally decomposable by Lemma~\ref{lem:loc_decomp:ldlat}.
\end{proof}

\begin{prop}\label{prop:germs_is_relative_comm}
	Let $U$ be a faithful locally decomposable profinite group, so that $U$ acts faithfully on the Boolean algebra $\ldlat(U)$ and on its Stone space $X$.
	Then $\Comm_{\Homeo(X)}(U)\cong \ms{L}(U)$. 	
	In particular, $\ms{L}(U)$ has a faithful piecewise full action on $X$.
\end{prop}
\begin{proof}
	By Theorem \ref{thm:AbstractCommensurator} (iv), there is a continuous homomorphism $\Comm_{\Homeo(X)}(U)\rightarrow \ms{L}(U)$ and this is an embedding because $\Comm_{\Homeo(X)}(U)$ is non-discretely micro-supported (see the proof of (i)$\Rightarrow$(ii)  in Lemma \ref{lem:non-discrete_condition}).
	
	Let $\alpha:A\rightarrow B$ be a local automorphism of $U$. 
	Then $\ldlat(A)=\ldlat(B)=\ldlat(U)$ as $A, B$ are open in $U$, and $\alpha$ induces a map $\tilde{\alpha}:\ldlat(U)\rightarrow \ldlat(U), [K\cap A]\mapsto [\alpha(K\cap A)]$ which is easily checked to be a well-defined automorphism of Boolean algebras. 
	Now, $U$ acts on $\ldlat(U)$ faithfully by conjugation, and for every $a\in A$, conjugation by $\alpha(a)$ is the same automorphism of $\ldlat(U)$ as $\tilde{\alpha}\circ a\circ\tilde{\alpha}^{-1}$.
	So $\alpha$ is induced by conjugation by $\tilde{\alpha}$. 
	Identifying automorphisms of $\ldlat(U)$ with homeomorphisms of $X$ gives the isomorphism in the statement. 
	
	By Corollary \ref{cor:piecewise_full_extension}, the group of germs can be seen as a piecewise full subgroup of $\Homeo(X)$. 
\end{proof}

\subsection{Profinite branch groups and their groups of germs}

\begin{defn}\label{def:profinite_branch_group}
	Let $U$ be a profinite group.
	We say $U$ is a \defbold{branch group} if there is an infinite locally finite rooted tree $T$ and a faithul continuous action of $U$ on $T$, such that $U$ acts transitively on each level of $T$ and locally decomposably on $\partial T$. 
	A profinite group is \defbold{just infinite} if every non-trivial closed normal subgroup is open.
\end{defn}

Branch groups are important, among other things, because they are one of three types into which profinite just infinite groups can be separated, as shown by Wilson in \cite{WilsonNewHorizons}.
This is proved by associating a structure lattice to the profinite group and examining the possible cases.
In the case of just infinite  profinite branch groups, Wilson's structure lattice corresponds to the decomposition lattice.

Profinite branch groups turn out to be precisely the second-countable profinite groups that are locally decomposable and act faithfully and transitively on the Stone space of their local decomposition lattice:

\begin{lem}\label{lem:branch_ldlat}
	Let $U$ be  a  profinite group.  Then the following are equivalent:
	\begin{enumerate}[(i)]
		\item $U$ is a  branch group;
		\item $U$ is second-countable with $\QZ(U) =\triv$, and $U$ acts faithfully and transitively on the Stone space $\mf{S}(\ldlat(U))$ of $\ldlat(U)$.
	\end{enumerate}
	Moreover, if $U$ is a branch group, with branch action on the tree $T$, then $\partial T$ is $U$-equivariantly isomorphic to $\mf{S}(\ldlat(U))$, and there is no proper quotient of $\mf{S}(\ldlat(U))$ on which $U$ acts faithfully.
\end{lem}

\begin{proof}
	Suppose that $U$ is a branch group, with branch action on the infinite locally finite rooted tree $T$. 
	Since $U$ is compact, it is homeomorphic to its image in $\Aut(T)$, which is second-countable, so $U$ is second-countable. 
	Given $v \in VT$, write $\rist_U(v)$ for the subgroup of $U$ fixing  all ends $\xi \in \partial T$ not represented by rays from the root passing through $v$; note that $\rist_U(v)$ is the rigid stabiliser of a clopen subset of $\partial T$.
	Since $U$ acts transitively on each level of the tree, we see that $U$ acts minimally on $\partial T$, ensuring that $\partial T$ is a perfect space; since $U$ compact, the action is in fact transitive.  Since the action is locally decomposable, the rigid stabilisers of nonempty clopen sets are all infinite. 
	Given $u \in U \setminus \triv$, then $u$ moves a vertex $v \in VT$, and then the corresponding rigid stabiliser is an infinite subgroup that has trivial intersection with $\CC_U(u)$. 
	Thus $\QZ(U) = \triv$. 
	
	By Lemma~\ref{lem:loc_decomp:ldlat} there is then a $U$-equivariant quotient map $\pi: \mf{S}(\ldlat(U)) \rightarrow \partial T$.
	We now claim that $\pi$ is a homeomorphism; in other words, $\{[\rist_U(v)] : v\in VT\}$ generates $\ldlat(U)$ as a Boolean algebra. For this, it suffices to show that every element of $\ldlat(U)$ can be represented by a product of rigid stabilisers $\rist_U(v)$ for $v \in VT$. 
	Let $\alpha \in \ldlat(U)$; then $\alpha$ has a representative $H \le U$, which is a closed subgroup such that $|U:\N_U(H)|<\infty$.
	Writing $L_n$ for the set of vertices of $T$ at level $n$, and $R_L:=\prod_{v \in L}\rist_U(v)$ for $L \subseteq L_n$, then by \cite[Theorem 1.2]{GarridoWilson}, there exists $n$ and an $\N_U(H)$-invariant subset $Y$ of $L_n$ such that
	\[
	H \cap \Der(R_{L_n}) = \Der(R_{Y}) \text{ and } H \cap R_{L_n \setminus Y}=\triv.
	\]
	In $\ldlat(U)$, the element $[R_{L_n \setminus Y}]$ is the complement of $[R_{Y}]$, so the condition that $H \cap R_{L_n \setminus Y}=\triv$ implies that $[H] \le [R_{Y}]$. 
	Also, $D(R_{Y})\leq H$.
	Now, as $[H]\in\ldlat(U)$, there is a locally normal subgroup $M\leq U$ such that $H\times M$ is open in $U$. 
	Then $[M\cap R_{Y}]\in \ldlat(U)$ and $M\cap D(R_{Y})=\{1\}$ implies that $M\cap R_{Y}$ is abelian. 
	But then the centraliser of $M\cap R_{Y}$ contains $\grp{M\cap R_{Y}, H, R_{L_n\setminus Y}}$, which is open in $U$. 
	So $M\cap R_{Y}=\{1\}$ and therefore  $[H]=[R_Y]$.
	
	Conversely, suppose (ii) holds and let $\mc{B}$ be a subalgebra of $\ldlat(U)$ on which $U$ acts faithfully; by Stone duality, such subalgebras correspond to quotients of $\mf{S}(\ldlat(U))$ on which $U$ acts faithfully.  Since $U$ is second-countable there is a descending chain $(U_i)$ of open subgroups of $U$ with trivial intersection.  We then choose an ascending sequence $\mc{A}_i$ of finite $U$-invariant subalgebras of $\mc{B}$ such that the kernel of the action of $U$ on $\mc{A}_i$ is contained in $U_i$; we can do this since $U$ is compact, ensuring that $\mc{B}$ is a directed union of finite $U$-invariant subalgebras, and $U$ acts faithfully on $\mc{B}$.  Now let $\mc{A} = \bigcup^\infty_{i=1}\mc{A}_i$ and let $\mc{P}_i$ be the set of atoms of $\mc{A}_i$; also set $\mc{P}_0 = \{\infty\}$.  Then regarding $\mc{P} = \bigcup^\infty_{i=1}\mc{P}_i$ as a subset of $\mc{B}$ with the induced partial order, we see that the Hasse diagram of $\mc{P}$ forms a locally finite rooted tree $T$ (with root vertex $\infty$) on which the action of $U$ is a branch action.  The same argument as in the previous paragraph then shows that $\partial T$ is isomorphic to $\ldlat(U)$ as a $U$-space; in particular, we deduce that every element of $\ldlat(U)$ is a join of elements of $\mc{P}$, so that $\mc{B} = \ldlat(U)$.  Thus (ii) implies (i) and we have no proper quotient of $\mf{S}(\ldlat(U))$ on which $U$ acts faithfully.
\end{proof}

\begin{rmk}
	The above lemma together with Proposition \ref{prop:germs_is_relative_comm} recover a special case of \cite[Theorem 1.2]{Roever}:  $\Comm_{\Homeo(\partial T)}(U)\cong \ms{L}(U)$ for any profinite branch group $U$.
	This group of germs is moreover piecewise full and hence expansive (by Lemma~\ref{lem:A(G)_quotientspace}), but it is not necessarily compactly generated. 	
\end{rmk}

\begin{ex}\label{eg:iterated_wreath_same_group}
	Let $U$ be a branch group acting on the $n$-regular infinite rooted tree $T$, for some $n>1$. 
	For a vertex $v\in T$, and $g\in \Aut(T)$, denote by $\delta_v(g)$ the automorphism of $T$ which acts like $g$ on the subtree rooted at $v$, and fixes every other vertex. 
	Assume that $U$ contains a finite index normal subgroup $K$ such that $ \prod_{|v|=1}\delta_v(K)\leq \prod_{|v|=1}\rist_U(v)\leq K$ with finite index, so $U$ is commensurable with its $n$th direct power. 
	This is what is known as a \defbold{regular branch group} in the literature. 
	For example, one could take $U$ to be an iterated wreath product of copies of the same transitive permutation group of degree $n$, in which case $U=K$; but there are many other more general examples.
	In fact, most of the well-known examples of branch groups are regular branch, up to replacing the tree $T$ by that obtained when ignoring some of the levels; for instance, the $n^d$ regular tree from the $n$-ary tree. 
	This last case occurs for the (closure of the) Grigorchuk groups $G_{\omega}$ of \cite{Grig_omegagroups} for periodic $\omega$.
	
	Denote by $V_n$ the Higman--Thompson group of almost automorphisms of the $n$-regular rooted tree $T$ (see \cite{Roever} for definitions and notation used here).
	Given  $\alpha\in V_{n}$, suppose it has a representative $a$ which takes the complete antichain $D$ of vertices of $T$ bijectively to the complete antichain $R$. 
	Then $K_D:=\prod_{v \in D}\delta_v(K)$ and $K_R:=\prod_{v\in R}\delta_v(K)$ are both open subgroups of $U$
	and $a K_D a^{-1} = \prod_{v\in D}\delta_{a(v)}(K)=K_R$, so $a$ commensurates $U$ and conjugation by $a$ gives a local automorphism of $U$. 
	As $\prod_{|v|=1}\delta_v(K)$ has finite index in $K$, any expansion of $a$ produces an equivalent local automorphism of $U$. 
	This means that $V_{n}$ commensurates  $U$ and so $G:=\grp{V_{n},U}$ can be given a non-discrete \tdlc topology if $U$ is profinite and $K$ is closed. 
	Now, $V_{n}$ acts expansively on $\partial T$ and therefore so does $G$; moreover, $V_{n}$ is finitely generated, so $G$ is compactly generated. 
	Corollaries~\ref{cor:group_expansive_iff_alternating_compactly_gen} and~\ref{cor:A_is_open} imply that $\Al(G)$ is a group in $\ms{S}$ that is open in $\ms{L}(U)$; in particular, $\Al(G)$ contains a finite index subgroup of $U$.
	It is much harder to establish when $\ms{L}(U)$ or $\Al(\ms{L}(U))$ are compactly generated. 
	This is unknown in the case that $U$ is the closure in $\Aut(T)$ of the first Grigorchuk group -- Theorem 4.16 of \cite{BEW} confirms that the compactly generated subgroup $\grp{U, V_{2}}$ is open in $\ms{L}(U)$ and topologically simple, but does not indicate whether $\ms{L}(U)$ or its monolith are compactly generated.

	If $U = W \wr F$, where $F \le \Sym(k)$ is transitive and $W$ is an infinitely iterated wreath product of copies of some transitive $D \le \Sym(d)$, with $1 < d,k < \infty$, we see that $U$ is a profinite branch group and $\ms{L}(U)$ contains a copy of the Higman--Thompson group $V_{d,k}$ for the tree where the root has $k$ children and all other vertices have $d$ children.  Such a group $U$ occurs as the stabiliser of a vertex in a Burger-Mozes universal group $U(F)$ (see \cite{BM00} or \cite{lc:book_trees}) in the following special case: $F$ is $2$-transitive, $k = d+1$, and $D$ is a point stabiliser of $F$.
		
	Very explicit results are obtained for the group of germs of such a group $U$ in \cite[Theorems 6.14 and 6.15]{CapDeM}: the group $\Al(\ms{L}(U))$ is compactly generated if and only if $\ms{L}(U)$ is compactly generated if and only if $\N_{\Sym(d)}(D) = D$.
	In that case, $\ms{L}(U)= \grp{V_{d,k},U}$ and $\Al(\ms{L}(U))$ has index at most 2.

%\green{\sout{Both cases are also illustrations of Theorem \ref{thm:jibranch:germ}.}}
\end{ex}

Only countably many examples of the above kind are known; namely, only countably many local isomorphism classes are known of groups in $\ms{S}$ that are locally isomorphic to a profinite branch group. 

\subsection{Examples from other groups acting on trees}

If $U\leq \Aut (T)$ is not a branch group, the decomposition lattice may be larger than that generated by $T$, even if the tree $T$ is locally finite.
Nevertheless, the action on $\partial T$ can still be enough to obtain examples of groups in $\ms{S}$ by using the results in the previous sections.

\begin{ex}\label{eg:waltraud_full_burgermozes}
	Consider a Burger-Mozes universal group $U(F)$ acting on the $d$-regular unrooted tree $T$, where $d\geq 3$ and $F\leq \Sym(d)$, and let $U$ be the stabiliser of some vertex $v$ (they are all conjugate as $U(F)$ acts transitively on vertices, for every $F$).
	Suppose that $F$ is intransitive, so that $U$ is not a branch group, and that $F$ has non-trivial point stabilisers, so that $U$ is not discrete.
	Instead of dealing with the general case, where notation could obscure the main points, we will focus on the illustrative special case where the orbits of $F$ on $\{1,\dots, d\}$ are $D_1=\{1,2\}$ and $D_2=\{3,\dots, d\}$.
	Each vertex $w$ of the tree rooted at $v$ can be identified with the sequence of labels in $\{1,\dots,d\}$ on the edges of the path from $v$ to $w$; the finite sequences that occur are exactly those where each digit is different from the next digit.
	Let $F_1$ be the pointwise stabiliser of $D_1$ in $F$ and suppose $F_1$ is non-trivial.
	Given a sequence $w$ that alternates between the elements of $D_1$, let 
		\[
		H(w):= \left( \prod_{l\in D_2}\rist_U(wl) \right) \rtimes F_1
		\] 
		where $g \in F_1$ permutes the factors, sending $\rist_U(wl)$ to $\rist_U(wg(l))$.
		Let $w_{1,n}$ be the alternating sequence $1212 \dots$ of length $n$ and let $w_{2,n}$ be the alternating sequence $2121 \dots$ of length $n$.
	Now set
	\begin{align*}
		K_1 & :=\prod_{d\in D_1}\prod_{n\geq 0} H(w_{d,2n}), &
		K_2 & :=\prod_{d\in D_1}\prod_{n\geq 0} H(w_{d,2n+1}), & 
		K_3 & :=\prod_{d\in D_2}\rist_U(d).
	\end{align*}
	Then $K_1 \times K_2 \times K_3$ is the pointwise stabiliser in $U$ of the ball of radius 1 around $v$, so the commensurability classes $[K_1], [K_2], [K_3]$ are in $\ldlat(U)$.
However, neither $K_1$ nor $K_2$  are commensurable to $\prod_{y\in Y}\rist_U(y)$ where $Y$ is a finite set of vertices. 
In particular, the Boolean algebra generated by $\{[\rist_U(y)] : y\in T\}$ is a $U$-invariant proper subalgebra of $\ldlat(U)$.
Denoting by $X$ the Stone space of $\ldlat(U)$, we have in particular that $\Comm_{\Homeo(\partial T)}(U)$ embeds properly into $\ms{L}(U)\cong \Comm_{\Homeo(X)}(U)$.

The universal group $U(F)$ commensurates $U$ and is compactly generated since it acts transitively on $T$ (see \cite[Lemma 2.4]{CapDeM} for a proof). 
Vertex transitivity also implies that the action of $U(F)$ on $\partial T$ is minimal (because there are no proper invariant subtrees or fixed ends) and expansive.
Thus $\Al(U(F);\partial T) \in \ms{S}$ by Theorem \ref{thm:Nekrashevych_simple} and Corollaries \ref{cor:group_expansive_iff_alternating_compactly_gen} and \ref{cor:A_is_open}.

The groups $\Full(U(F);\partial T)$ for $F$ not necessarily transitive have been previously considered in \cite{Lederle} where it is shown that they are always compactly generated and that the derived subgroup is the monolith of $\Full(U(F);\partial T)$ (and therefore coincides with $\Al(U(F);\partial T)$), which is open and of finite index.

It is not clear when $\Full(U(F);\partial T)=\Comm_{\Homeo(\partial T)}(U(F))$ for intransitive $F$.
\end{ex}

\begin{ex}\label{eg:leboudec_G(FF')}
Let $F\leq F'\leq \Sym(d)$ with $d\geq 3$ and assume that $F$ has some non-trivial point stabiliser. 
In \cite{LB16}, the group $G(F,F')$  is defined as those automorphisms of the $d$-regular unrooted tree $T$ whose local action is in $F$ for all but finitely many vertices, for which the local action is in $F'$. 
Therefore $U(F)\leq G(F,F')\leq U(F')$ and, according to \cite[Lemma 3.2]{LB16}, $G(F,F')$ commensurates $U(F)$, so the topology of $G(F,F')$ is defined, using Proposition \ref{prop:bourbaki_gentop} so as to make $U(F)$ an open subgroup.
The action of $G(F,F')$ on $T$ is continuous but not proper as soon as $F\neq F'$ (see \cite[Section 1.6]{LB16}), so these groups are locally isomorphic to $U(F)$ but of a different nature.
In particular, $G(F,F')\leq \Comm_{\Homeo(\partial T)}(U)$ where $U$ is the stabiliser of a vertex in $U(F)$. 

Since $U(F)$ already acts vertex-transitively on $T$, the action of $G(F,F')$ on  $\Homeo(\partial T)$ is expansive. Moreover, according to  \cite[Corollary 3.8]{LB16}, $G(F,F')$ is always compactly generated. 
Therefore $\Al(G(F,F')) \in \ms{S}$. 
Note that, in the case that $F$ is 2-transitive and its point stabilisers are self-normalising as subgroups of $\Sym(d-1)$ then $\ms{L}(U)\cong \Comm_{\Homeo(\partial T)}(U) =\grp{U, V_{d-1,d}}= \Full(U(F))$ (as noted in Example \ref{eg:iterated_wreath_same_group}) and so $\Al(G(F,F'))= \Al(U(F))$. 

In general, it is not clear when  $\Al(G(F,F'))= \Al(U(F))$.
\end{ex}

\section{Alternatable and splittable actions}\label{sec:alternatable}

In this section we consider some variations on having a minimal piecewise full action, and relate them to normal subgroups and automorphisms of piecewise full groups.

\begin{defn}\label{defn:alternatable}
	Recall that an action of a group $G$ on a compact zero-dimensional space $X$ is \defbold{alternatable} if $\Al(G;X) \le G$.  
	We say the action is \defbold{splittable} if every element $g \in G$ can be written as a product of elements $g = g_1g_2$ of $G$ such that for $i=1,2$, there is some nonempty clopen subset $Y_i$ of $X$ that is fixed pointwise by $g_i$.
	The action is \defbold{locally minimal} if there is a base of topology $\mc{B}$ for $X$ such that $\rist_{G}(Y)$ acts minimally on $Y$ for all $Y \in \mc{B}$.
\end{defn}

There are several known spatial reconstruction theorems for piecewise full groups and their relatives; see \cite{MatteBon} for strong results of this type.  The following will suffice for our purposes.

\begin{thm}[{\cite[Theorem~3.10]{Nekra}, based on \cite{Rubin}}]\label{thm:reconstruction}
	For $i=1,2$, let $X_i$ be a locally compact Hausdorff space and let $G_i \le \Homeo(X_i)$.  Suppose that $G_i$ is locally minimal on $X_i$.  Then for every group isomorphism $\phi: G_1 \rightarrow G_2$, there is a homeomorphism $F: X_1 \rightarrow X_2$ such that $\phi(g) = F \circ g \circ F\inv$ for all $g \in G_1$.
\end{thm}

In particular, if a group $G$ admits a locally minimal action on a locally compact Hausdorff space $X$, the action is essentially unique, and we can write $\Al(G)$ and $\Full(G)$ for $\Al(G;X)$ and $\Full(G;X)$ without having to specify the space.

The next lemma relates some of the properties given in Definition~\ref{defn:alternatable} to one another.

\begin{lem}\label{lem:support_split}
	Let $X$ be a nonempty perfect compact zero-dimensional space and let $G$ be a group of homeomorphisms of $X$.  Suppose the action of $G$ is minimal and alternatable.
	\begin{enumerate}[(i)]
		\item The action of $G$ is splittable.
		\item Let $Y \in \mc{CO}(X) \setminus \{\emptyset\}$.  Then $\rist_G(Y)$ is non-trivial and for every non-trivial normal subgroup $H$ of $\rist_G(Y)$, the action of $H$ on $Y$ is minimal and alternatable.  In particular, the action of $G$ is locally minimal.
	\end{enumerate}
\end{lem}

\begin{proof}
	Let $A = \Al(G;X)$.
	
	(i)
	Let $g \in G \setminus \triv$ and let $Y$ be a nonempty compact open subset of $X$ that is not fixed pointwise by $g$; say $gx \neq x$ for $x \in Y$.  We see that there is a nonempty compact open neighbourhood $Z$ of $x$ contained in $Y$, and an element $h \in G$, such that
	\[
	\{Z,gZ,hZ,ghZ\}
	\]
	is pairwise disjoint and has union $Y'$ contained in $Y \cup gY$; by choosing $Z$ small enough we may additionally ensure $gy \neq y$ for some $y \in Y \setminus \ol{Y'}$.  Define $g_1$ as follows:
	\[
	g_1x = 
	\begin{cases}
		gx &\mbox{if} \;  x \in Z \cup hZ \\ 
		g\inv x &\mbox{if} \;  x \in gZ \cup ghZ\\
		x &\text{otherwise}
	\end{cases}.
	\]
	Thus $g_1 \in A$, so $g_1 \in G$ and hence $g_2:= g\inv_1g \in G$.  Observe that $g_2$ fixes $Z$ pointwise, unlike $g$; meanwhile, $g_1$ fixes pointwise a neighbourhood of $y$, whereas $g$ does not.  The statement 
	\[
	\forall x \in X: gx = x \Rightarrow g_ix=g_ix
	\]
	is clear for $i=1$, from, which it follows also for $i=2$.
	
	(ii)
	Fix $Y \in \mc{CO}(X) \setminus \{\emptyset\}$ and let $R = \Al(\rist_G(Y);Y)$.  Note that $R$ is a subgroup of $A$ that is supported on $Y$, so in fact
	\[
	R \le \rist_A(Y) \le \rist_G(Y).
	\]
	In particular, the action of $H$ on $Y$ is alternatable whenever $R \le H \le \rist_G(Y)$.
	
	Let $g \in G$ be such that $x$ and $gx$ are distinct points in $Y$.  Then by minimality we can find another distinct point $hx$ for $h \in G$; there is then a clopen neighbourhood $Z$ of $x$ such that $Z$, $gZ$ and $hZ$ are disjoint clopen subsets of $Y$.  In particular, there is $k \in \rist_A(Y)$ such that $kx = gx$.  Thus the $\rist_A(Y)$-orbit of $x$ is $Gx \cap Y$.  In particular, $\rist_A(Y)$ acts minimally on $Y$.  By repeating the argument we see that $\Al(\rist_A(Y);Y)$ itself acts minimally on $Y$, so $R$ acts minimally on $Y$.
	
	We now see by Theorem~\ref{thm:Nekrashevych_simple} that $R$ is simple and is the unique smallest non-trivial normal subgroup of $\rist_G(Y)$.  Thus given a non-trivial normal subgroup $H$ of $\rist_G(Y)$, then $R \le H$, and hence the action of $H$ is minimal and alternatable.
\end{proof}

We obtain the following relationships between a minimal piecewise full group $G$ and its automorphisms.

\begin{prop}\label{prop:piecewise_rigidity}
	Let $X$ be a nonempty perfect compact zero-dimensional space and let $G \le \Homeo(X)$ be minimal and piecewise full.
	\begin{enumerate}[(i)]
		\item Let $\Al(G) \le H \le \Homeo(X)$.  Then there is an isomorphism $\theta: \N_{\Homeo(X)}(H) \rightarrow \Aut(H)$ such that for $h \in H$, $\theta(h)$ is conjugation by $h$.
		\item Every automorphism of $\Al(G)$ extends to $G$ and $\Al(G)$ is characteristic in $G$.  
		In particular, $\N_{\Homeo(X)}(G) = \N_{\Homeo(X)}(\Al(G))$.
		\item We have $G = SS$, where $S$ is the set of homeomorphisms of $X$ that normalise $G$ and are supported on a proper compact subset of $X$.  In particular, $G$ is the largest splittable subgroup of $\N_{\Homeo(X)}(G)$.
		\item Given $H \le \N_{\Homeo(X)}(G)$ such that every $H$-orbit on $X$ has at least $3$ points, we have $H \le G$ if and only if $\Al(H) \le \Al(G)$.
	\end{enumerate}
\end{prop}

\begin{proof}
	For (i), we note that $\Al(G)$ is locally minimal by Lemma~\ref{lem:support_split}(ii), so $H$ is also locally minimal.  Theorem~\ref{thm:reconstruction} then yields a surjection $\theta: \N_{\Homeo(X)}(H) \rightarrow \Aut(H)$ with kernel $\CC_{\Homeo(X)}(H)$ such that for $h \in H$, $\theta(h)$ is conjugation by $h$.  Given $f \in  \N_{\Homeo(X)}(H) \setminus \triv$, there is $Y \in \mc{CO}(X) \setminus \{\emptyset\}$ such that $Y$ and $fY$ are disjoint, and then $f$ does not normalise $\rist_H(Y)$.  Thus $\CC_{\Homeo(X)}(H)$ is trivial, so $\theta$ is an isomorphism.
	
	It is now clear that $\Al(G)$ is characteristic in $G$.  On the other hand, given $f \in \N_{\Homeo(X)}(\Al(G))$, then $f$ normalises $\Full(\Al(G);X)$, and $\Full(\Al(G);X) = G$ by Lemma~\ref{lem:alternating_to_full}; thus (via part (i)) every automorphism of $\Al(G)$ extends to $G$.  This proves (ii).
	
	For (iii), let $f \in S$, let $x \in X$ and let $Y$ be a nonempty compact open set fixed pointwise by $f$.  Since $G$ is minimal there is $g \in G$ and a neighbourhood $Z$ of $x$ such that $gZ \subseteq Y$.  We then see that the commutator $h = fg\inv f\inv g$ agrees with $f$ on $Z$; since $f \in \N_{\Homeo(X)}(G)$ we have $h \in G$.  Thus $f$ acts as an element of $G$ in a neighbourhood of $x$; since $x \in X$ was arbitrary and $G$ is piecewise full, we deduce that $f \in G$.  Thus $S \subseteq G$.  On the other hand, by Lemma~\ref{lem:support_split}(i), $G$ is splittable, so $G = SS$.
	
	For (iv), given $H \le \N_{\Homeo(X)}(G)$, if $H \le G$ then certainly $\Al(H) \le \Al(G)$.  On the other hand, if $\Al(H) \le \Al(G)$ then $\Full(H) \le \Full(G) = G$, since $\Full(H) = \Full(\Al(H))$ by Lemma~\ref{lem:alternating_to_full}.
\end{proof}

\begin{ex}
	Let $G$ be a countably infinite residually finite group.  Suppose $G$ has an automorphism $h$ that does not centralise any finite index subgroup of $G$.  Let $N$ be an $h$-invariant closed normal subgroup of the profinite completion $\widehat{G}$ of $G$, such that the natural homomorphism $\pi$ from $G$ to $X := \widehat{G}/N$ is injective.  (For a specific example, one can take $G = \Zb$, $h: x \mapsto -x$ and $X = \Zb_2$.)  Then $G$ has a natural translation action on $X$ via $\pi$, so we can regard $G$ as a subgroup of $\Homeo(X)$ that acts minimally on $X$, and $h$ extends to a homeomorphism (indeed, a topological group automorphism) $\hat{h}$ of $X$.  As subgroups of $\Homeo(X)$, we see that $\hat{h} \in \N_{\Homeo(X)}(G)$, and consequently $\hat{h} \in \N_{\Homeo(X)}(F)$ where $F = \Full(G;X)$.  However, we claim that $\hat{h}$ is not an element of $F$.  Supposing for a contradiction that $\hat{h} \in F$, then there is a coset $\pi(k)W$ of an open subgroup $W$ of $X$ (with $k \in G$) such that $\hat{h}$ acts on $\pi(k)W$ as a constant element $g$ of $G$, that is, $\hat{h}(\pi(k)w) = \pi(g)\pi(k)w$ for all $w \in W$.  In other words
	\[
	\forall w \in W: \; \hat{h}(w) = \pi(h(k)^{-1}gk)w.
	\]
	Taking $w=1$ we see that $\pi(\hat{h}(k)^{-1}gk) = 1$, so in fact $\hat{h}(w) = w$ for all $w \in W$.  But then also $h(w') = w'$ for all $w' \in \pi^{-1}(W)$, which contradicts our assumption about $h$.  In particular, we conclude that $\N_{\Homeo(X)}(F) > F$.
\end{ex}

The above example has the limitation that the action of $G$ extends to a continuous action of a compact group; any such action on a nonempty perfect compact space $X$ will be equicontinuous, which is far from being expansive.  In the case where $G$ is a compactly generated \tdlc group acting minimally and expansively on $X$, the authors are not aware of an example where $\N_{\Homeo(X)}(\Al(G;X)) > \Full(G;X)$; see Remark~\ref{rmk:simple} below.

\section{Alternatable almost simple t.d.l.c. groups}\label{sec: alternatable almost simple}

In this section we introduce a new class $\ms{A}$ of \tdlc groups inspired by the results of the previous sections, for which we can prove a number of structural properties.

\begin{defn}
	Let $G$ be a non-discrete locally compact group.  An \defbold{$\ms{A}$-action} of $G$ is a faithful minimal locally decomposable action of $G$ on a compact zero-dimensional space $X$, such that $\Al(G;X)$ is a compactly generated open subgroup of $G$.
	
	Say that $G \in \ms{A}$ if $G$ admits an $\ms{A}$-action.
\end{defn}

By Theorem~\ref{thm:reconstruction}, $\ms{A}$-actions are essentially unique when they exist, so we can speak of \emph{the} $\ms{A}$-action of $G$ and write $\mc{BI}(G)$, $\Full(G)$, $\Al(G)$ and so on, where it is implied that we are taking the $\ms{A}$-action of $G$.

Note that in the definition of $\ms{A}$ we do not assume $G$ is compactly generated, only that $\Al(G;X)$ is compactly generated.  We do not know in general if $\Al(G;X)$ being compactly generated implies that $\Full(G;X)$ is compactly generated.

\subsection{Properties of groups in $\ms{A}$}

The following gives an equivalent condition for an $\ms{A}$-action.
In particular, it shows that the requirement that $\Al(G;X)$ be \emph{open} in $G$ is not critical for specifying the class $\ms{A}$, as it follows from other properties.
However, it is a natural and convenient assumption to make from the outset in the context of local isomorphism classes of \tdlc groups, which we will discuss in the section after this one.

\begin{prop}\label{prop:AS_characterisation}
	Let $G$ be a non-discrete \tdlc group acting by homeomorphisms on the nonempty perfect compact zero-dimensional space $X$.  Suppose that the action is faithful and locally decomposable.  Then the following are equivalent:
	\begin{enumerate}[(i)]
		\item The action of $G$ is an $\ms{A}$-action;
		\item There is a compactly generated closed subgroup $H$ of $G$ with minimal expansive action on $X$, such that $\Al(H;X) \le G \le \Full(H;X)$.
	\end{enumerate}
\end{prop}

\begin{proof}
	If (i) holds, then we can set $H = \Al(G;X)$; then (ii) holds by Lemmas~\ref{lem:A(G)_quotientspace} and~\ref{lem:alternating_to_full}.
	
	Suppose (ii) holds and let $L = \Full(H;X)$. 
	Corollary~\ref{cor:group_expansive_iff_alternating_compactly_gen} implies that $\ol{\Al(H;X)}$ is compactly generated. Corollary~\ref{cor:A_is_open} then implies that $\Al(H;X)$ is simple and open in $L$, hence also in $G$, and also that $\Al(H;X)$ has fully compressible (hence minimal) action on $X$, so the $G$-action is also fully compressible.  Thus $\Al(G;X) = \Mon(G)$, and in particular $\Al(G;X) \le \Al(H;X)$; since $\Al(H;X)$ is simple, in fact $\Al(H;X) = \Al(G;X)$ and we deduce (i).
\end{proof}

Let us apply the results of the previous sections to obtain some properties of groups in $\ms{A}$.
Recall from Section~\ref{sec:robust} the classes $\ms{R}_{\ldlat}$ and $\ms{S}_{\ldlat}$.

\begin{thm}\label{thm:A_properties}
	Let $G \in \ms{A}$, with $\ms{A}$-action on a compact zero-dimensional space $X$.  Then the following holds.
	\begin{enumerate}[(i)]
		\item We have $\QZ(G)=\triv$.  Consequently the action of $G$ on $X$ is non-discretely micro-supported.
		\item We have $\Der(G) = \Al(G) = \ol{\Mon}(G)$.
		\item We have $G \in \ms{R}_{\ldlat}$ and $\Der(G) \in \ms{S}_{\ldlat}$.
		\item Every faithful micro-supported action of $\Der(G)$ on a compact zero-dimensional space, in particular the action of $\Der(G)$ on $X$, is fully compressible.
		 %The actions of $\Der(G)$ on $\mf{S}(\ldlat(G))$ and $X$ are fully compressible.
		\item For every $H$ such that $\Der(G) \le H \le G$, the action of $H$ on $X$ is an $\ms{A}$-action, with $\Full(H) = \Full(G)$ and $\Al(H) = \Al(G)$.  In particular, $H \in \ms{A}$.
		\item The space $X$ is homeomorphic to the Cantor space, while $\Full(G)$, and hence also $G$, is second-countable.
		\item There is a $G$-equivariant homeomorphism from $X$ to $\mf{S}(\mc{A})$, where $\mc{A}$ is the unique non-trivial sublattice of $\ldlat(G)$ such that $G$ acts locally minimally on $\mf{S}(\mc{A})$.
	\end{enumerate}
\end{thm}

\begin{proof}
	All actions are on $X$ unless otherwise specified.
	As $G$ is not discrete, part (i) follows immediately from Lemma~\ref{lem:non-discrete_condition}.
	Part (ii) follows from Corollary~\ref{cor:A_is_open} which also shows together with Theorem \ref{thm:Nekrashevych_simple} that $G \in \ms{R}$ and $\Der(G) \in \ms{S}$, with $\Der(G)$ open in $G$. 
	From now on we can also read $\Der(G)$ as $\Al(G)$ or $\ol{\Mon}(G)$ as appropriate for the result being referenced.
		
	By Lemma~\ref{lem:loc_decomp:ldlat} and Stone duality, we see that $X$ is $G$-equivariantly homeomorphic to $\mf{S}(\mc{A})$, for some sublattice $\mc{A}$ of $\ldlat(G)$; from now on we can identify $X$ with $\mf{S}(\mc{A})$. 
	In particular, note that $G$ acts faithfully on $\ldlat(G)$; this completes the proof of (iii). 
	Part (iv) follows from Corollary~\ref{cor:A_is_open}.
	%By Corollary~\ref{cor:A_is_open} we see that $\Der(G)$ has fully compressible action on $X$, proving (iv).
	Now consider $H$ such that $\Der(G) \le H \le G$.  We have $\Full(H) = \Full(G)$ by Lemma~\ref{lem:alternating_to_full}, and hence $\Al(H;X) = \Al(G;X) = \Der(G)$; moreover the action of $H$ is faithful and locally decomposable, since $H$ is open in $G$. Thus the action of $H$ on $X$ is an $\ms{A}$-action, completing the proof of (v).
	
	Since we have established that $\Der(G)$ is compactly generated, the fact it acts expansively ensures that $X$ is second-countable, and then since the action is faithful and minimal, we see that $X$ is perfect; thus $X$ is homeomorphic to the Cantor space.  Lemma~\ref{lem:alternating_to_full} now shows that $\Full(\Der(G)) = \Full(G)$.  We then see by Proposition~\ref{prop:piecewise_full_extension:lcsc} that $\mc{BI}(G)$ is second-countable, so $\Full(G)$ and $G$ are second-countable, proving (vi).
		
	For part (vii), it suffices to prove that $\mc{A}$ satisfies the given characterisation.  The action of $G$ on $\mf{S}(\mc{A})$ is locally minimal by Lemma~\ref{lem:support_split}(ii).  Conversely, let $\mc{B}$ be a non-trivial $G$-invariant sublattice of $\ldlat(G)$ such that $G$ acts locally minimally on $\mf{S}(\mc{B})$.  Then by Lemma~\ref{lem:simple_dynamics} the action of $\ol{\Mon}(G)$ on $\mf{S}(\mc{B})$ is minimal, hence non-trivial, so the action of $G$ on $\mf{S}(\mc{B})$ is faithful.  Applying Theorem~\ref{thm:reconstruction} we see that $\mf{S}(\mc{B})$ is $G$-equivariantly homeomorphic to $X$; considering rigid stabilisers it follows that $\mc{B} = \mc{A}$.  Thus (vii) holds.
\end{proof}

\begin{rmk}
	We see here a distinction in the dynamics of the discrete case versus the non-discrete case.  Suppose that $(X,\bZ)$ is a minimal subshift, and let $G = \Full(G;X)$ with the discrete topology.  Then $G$ is finitely generated and the action of $G$ on $X$ is faithful, locally decomposable, minimal and expansive.  The only part of the definition of an $\ms{A}$-action that fails is that the action fails to be \emph{non-discretely} micro-supported.  However, it was shown by Juschenko--Monod \cite{JuschenkoMonod} that $G$ is amenable, so in particular $X$ admits a $G$-invariant probability measure.  In particular, since $(X,G)$ is minimal, this action is certainly not compressible.  This dichotomy in the dynamics between discrete and non-discrete compactly generated \tdlc groups is explored in more detail in a recent article \cite{CapraceLeBoudec} of Caprace--Le Boudec, in which it is used to show, for $G$ in various classes of finitely generated groups, that $G$ cannot have an infinite commensurated subgroup with trivial core.
\end{rmk}

\subsection{Free action on the Furstenberg boundary}

In a recent preprint \cite{CLBMB}, Caprace, Le Boudec and Matte Bon show that a locally compact group $G$ acts freely on its Furstenberg boundary $\partial_F G$ if and only if it satisfies the following strong nonamenability property: given a closed relatively amenable subgroup $H$ of $G$, then there is a net $(g_i)$ in $G$ such that $g_iHg\inv_i \rightarrow \triv$ in the Chabauty space $\mathrm{Sub}(G)$.  They show that Neretin's groups $\mc{N}_{d,k}$ have this property, giving the first examples of $G \in \ms{S}$ such that $G$ acts freely on $\partial_F G$, and remark that it is likely that their result also applies to many other simple groups.  We note here that their result applies to all of the class $\ms{A}$.

\begin{thm}\label{thm:CLBMB}
	Let $G \in \ms{A}$.  Then $G$ acts freely on its Furstenberg boundary.  Equivalently, every $G$-conjugacy class of closed relatively amenable subgroups of $G$ contains the trivial group in its $\mathrm{Sub}(G)$-closure.
\end{thm}

\begin{proof}
	By hypothesis, the $\ms{A}$-action of $G$ on the Cantor space $X$ is faithful, minimal and alternatable. 
	Part (iv) of Theorem \ref{thm:A_properties} ensures that the action of $G$ on $X$ is fully compressible. 
	Let $Y \in \mc{CO}(X) \setminus \{\emptyset\}$ and let $H = \rist_G(Y)$.  The action of $H$ on $Y$ is minimal and alternatable by Lemma~\ref{lem:support_split}(ii).
	Let $Z,W \in \mc{CO}(Y) \setminus \{\emptyset\}$.  Then there is $g \in G$ such that $gZ \subseteq W$ and $Z \cup gZ$ is properly contained in $Y$.  There are then $g_1,g_2 \in Z$ such that $g_1Z$ and $g_2Z$ are contained in $Y$ and disjoint from $Z \cup gZ$; it is then easy to obtain $h \in \Al(G;X) \le G$ supported on $Y$ such that $hZ = gZ$.  Thus $Z$ is compressible for the action of $H$; given the freedom of choice of $Z$, we conclude that the action of $H$ on $Y$ is fully compressible.
	
	Fully compressible actions on compact spaces are also strongly proximal.
	Thus the $\ms{A}$-action of $G$ is piecewise minimal-strongly-proximal in the sense of  \cite{CLBMB}, so by \cite[Theorem~1.3]{CLBMB}, every non-trivial conjugacy class of closed relatively amenable subgroups of $G$ accumulates at the trivial group, and hence $G$ acts freely on $\partial_F G$.  
\end{proof}

\subsection{Actions on hyperbolic spaces}\label{sec:hyperbolic}

Another property of piecewise full groups that has been studied in the recent literature is whether or not they admit ``large'' actions on hyperbolic spaces.  To be precise, we use the following definitions from \cite[Section~2]{BFG}.

\begin{defn}
	An isometry $g$ of a hyperbolic space $X$ is \defbold{loxodromic} if given $x \in X$, the map $n \mapsto g^nx$ is a quasi-isometric embedding from $\Zb$ to $X$.
	
	An action of a group $G$ on a hyperbolic space $X$ by isometries is of \defbold{general type} if $G$ does not fix any point in the Gromov boundary $\partial X$, and has loxodromic elements with distinct limit points.
	
	We say a group $G$ has \defbold{property (NGT)} if it has no general type actions on hyperbolic spaces.  We say $G$ has \defbold{property (NL)} if, for all actions of $G$ by isometries of hyperbolic spaces, no element of $G$ acts as a loxodromic isometry.
\end{defn}

In \cite{BFG}, Balasubramanya, Fournier-Facio and Genevois show that many examples of groups acting by homeomorphisms have property (NL).  In \cite{GR_compressible}, the methods of \cite{BFG} are used to show the following:

\begin{lem}[{See \cite[Corollary~3.5]{GR_compressible}}]\label{lem:NGT}
	Let $X$ be a perfect zero-dimensional Hausdorff space, let $G$ be a compressible piecewise full group of homeomorphisms of $X$, and let $\Der(G) \le G_0 \le G$.  Then $G_0$ has property (NGT).
\end{lem}

We now obtain the stronger property (NL) for simple groups in $\ms{A}$.

\begin{thm}\label{thm:NL}
	Let $G \in \ms{A}$.  Then $G$ has property (NGT), while $\Der(G)$ has property (NL) and is one-ended.
\end{thm}

\begin{proof}
	We suppose $G$ has $\ms{A}$-action on $X$; in particular, by Theorem~\ref{thm:A_properties} this action is fully compressible, and we have $\Der(G) = \Al(G) = \Der(F)$ and $\Al(F) \le G \le F$, where $F$ is the piecewise full group of the action.  By Lemma~\ref{lem:NGT} we deduce that $G$ has property (NGT).
	
	We can now appeal to a theorem of Gal and Gismatullin: since $F$ is minimal, fully compressible and piecewise full, \cite[Theorem~5.1]{GalGis} shows that given any non-trivial $f \in \Der(F)$ and any non-trivial conjugacy class $C$ in $\Der(F)$, then $f$ can be written as a product of at most $9$ elements of $C^{\pm 1}$.  It then follows by \cite[Corollary~2.18]{BFG} that $\Der(G) = \Der(F)$ has property (NL).  Since $G \in \ms{A}$ we also know that $\Der(G)$ is compactly generated.  The proof that $\Der(G)$ is one-ended is given in \cite[Corollary~3.4]{GR_compressible}.
\end{proof}

\section{The role of $\ms{A}$ in the local structure theory of t.d.l.c. groups}\label{sec:A and S}

In this section, we consider the role the class $\ms{A}$ plays within the class of all \tdlc groups.  More specifically, we are interested in which local isomorphism classes of \tdlc groups are represented by members of $\ms{A}$. 
Recall the definitions and results of Section \ref{ssec:germs_def} on germs of automorphisms of \tdlc groups. 

There are two natural questions to ask about $\ms{A}$ from the perspective of local isomorphism classes.

\begin{enumerate}[(1)]
	\item Given a \tdlc group $G$, under what circumstances is $G$ locally isomorphic to a group in $\ms{A}$?
	\item Given $G \in \ms{A}$, what can we say about the structure of the class of (compactly generated) \tdlc groups locally isomorphic to $G$?
\end{enumerate}

For (1), we see from Theorem~\ref{thm:A_properties}(i) that a necessary condition on $G$ is that $\QZ(G)$ is discrete.  Thus both questions can be rephrased in terms of groups of germs as follows:

\begin{enumerate}[(1)]
	\item Given a \tdlc group $G$, under what circumstances does $\ms{L}(G)$ have an open subgroup in $\ms{A}$?
	\item Given $G \in \ms{A}$, what can we say about the structure of the class of (compactly generated) open subgroups of $\ms{L}(G)$?
\end{enumerate}

Question (1) is also motivated by the analogous question for the class $\ms{S}$ of non-discrete compactly generated topologically simple \tdlc groups. 
Since every $G \in \ms{S}$ satisfies $\QZ(G)=\triv$ (recall Lemma~\ref{lem:simple_dynamics}), the group of germs framework applies without any difficulties. 
By \cite{SmithDuke}, it is known that $\ms{S}$ and $\ms{S}_{\ldlat}$ (those groups in $\ms{S}$ that are faithful locally decomposable) have $2^{\aleph_0}$ isomorphism classes of groups. 
However, only countably many \emph{local} isomorphism classes in $\ms{S}$ are known. 
Even for known examples of $G \in \ms{S}_{\ldlat}$, the structure of $\ms{L}(G)$ can be mysterious. 
The main cases where we have a good description of $\ms{L}(G)$ for $G \in \ms{S}_{\ldlat}$ are similar to those described in Example \ref{eg:iterated_wreath_same_group}, where $G$ is the monolith of $\ms{L}(U)$ for $U$ a profinite branch group.

 \subsection{A characterisation of the local isomorphism types in $\ms{A}$}

As an answer to Questions (1) and (2), we find the following, which will be proved over the course of this subsection.
\begin{thm}\label{thm:locally_in_A}
	Let $G$ be a \tdlc group with trivial quasicentre and let $L = \ms{L}(G)$.  Let $X = \mf{S}(\ldlat(L))$.  Then the following are equivalent:
	\begin{enumerate}[(i)]
		\item $L$ has an open subgroup in $\ms{A}$;
		\item $L$ has an open subgroup in $\ms{S}_{\ldlat}$;
		\item $L$ has an open subgroup in $\ms{R}_{\ldlat}$;
		\item $L \in \ms{R}_{\ldlat}$;
		\item We can write $L$ as a directed union of compactly generated open subgroups $F_i$, each with $\ms{A}$-action on a quotient space $X_i$ of $X$, such that $X = \varprojlim X_i$;
		\item $L$ acts faithfully on $X$, and there are compactly generated open subgroups $A$ and $B$ of $L$ such that $A$ is an expansive \tdlc group, and $B$ admits a faithful minimal micro-supported action on a compact zero-dimensional space.
	\end{enumerate}
\end{thm}

Recall from Definition \ref{defn:expansive_group} that a \tdlc group $G$ is \defbold{expansive} if and only if it has an identity neighbourhood $U$ such that $\bigcap_{g \in G}gUg\inv=\triv$.
Given Van Dantzig's theorem, it is easy to see that every \tdlc group is an inverse limit of expansive \tdlc groups; conversely, expansive \tdlc groups are those that cannot be approximated by proper quotients with compact kernel.

When $G$ is compactly generated, we can say something more about approximation of $G$ by its quotients and about translating expansivity of $G$ to expansivity of its actions.

\begin{lem}[{see \cite[Theorem~3.3]{RW_chief} and \cite[Corollary~4.1]{CapraceMonod}}]\label{lem:filtering_normal}
	Let $G$ be a compactly generated \tdlc group and let $\mc{F}$ be a filtering family of closed normal subgroups of $G$ with trivial intersection.  Let $U$ be a compact open subgroup of $G$, and let $R = \bigcap_{g \in G}gUg\inv$.  Then there is $N \in \mc{F}$ such that $N \cap R$ is open in $N$; hence also for any closed normal subgroup $M \le N$, then $M \cap R$ is compact, normal in $G$ and open in $M$. 
\end{lem}

\begin{lem}\label{lem:expansive_quotient}
	Let $G$ be a compactly generated expansive \tdlc group with no non-trivial discrete normal subgroup.  Suppose $G$ acts continuously and faithfully on a compact zero-dimensional space $X$.  Then $X$ is an inverse limit of $G$-spaces $Y$ such that $G$ acts faithfully and expansively on $Y$.
\end{lem}

\begin{proof}
	Since $G$ is expansive, there is a compact open subgroup $U$ of $G$ such that $\bigcap_{g \in G}gUg\inv = \triv$.  Let $I$ be the set of clopen partitions of $X$ ordered by refinement (that is, ${\mc{P}} \le {\mc{P}}'$ if ${\mc{P}}'$ is a refinement of ${\mc{P}}$).  By Lemma~\ref{lem:subshift}, we can write $X$ as an inverse limit of $G$-spaces, $X = \varprojlim (X_{\mc{P}};\phi_{\mc{P}})_{\mc{P} \in I}$, where $X_{\mc{P}}$ is a quotient space of $X$ on which $G$ acts as a subshift, hence expansively.  Let $K_{\mc{P}}$ be the kernel of the action of $G$ on $X_{\mc{P}}$; note that if $\mc{P}' \ge \mc{P}$ then $K_{\mc{P}'} \le K_{\mc{P}}$.  Since $X = \varprojlim (X_{\mc{P}};\phi_{\mc{P}})_{\mc{P} \in I}$, we see that $\bigcap_{{\mc{P}} \in I}K_{\mc{P}} = \triv$.  Thus by Lemma~\ref{lem:filtering_normal}, there is some $\mc{P} \in I$ such that $K_{\mc{P}} \cap \triv$ is open in $K_{\mc{P}}$, in other words, $K_{\mc{P}}$ is discrete.  Since $G$ has no non-trivial discrete normal subgroup, we have $K_{\mc{P}} = \triv$, in other words, $G$ acts faithfully on $X_{\mc{P}}$.  We now obtain the conclusion by writing $X = \varprojlim (X_{\mc{P}};\phi_{\mc{P}})_{\mc{P} \in I'}$, where $I' = \{\mc{P}' \in I \mid \mc{P}' \ge \mc{P}\}$.
\end{proof}

The equivalence of (iii) and (iv) in Theorem~\ref{thm:locally_in_A} is given by the following. 
\begin{lem}\label{lem:R_is_regional}
	Let $G$ be a \tdlc group with trivial quasicentre.  Then $G \in \ms{R}_{\ldlat}$ if and only if some compactly generated open subgroup of $G$ belongs to $\ms{R}_{\ldlat}$.
\end{lem}

\begin{proof}
	Given Lemma~\ref{lem:faithful_germ}, there is no loss of generality in assuming that $G$ is faithful locally decomposable; thus open subgroups of $G$ belong to $\ms{R}_{\ldlat}$ if and only if they belong to $\ms{R}$.
	
	If $G \in \ms{R}$, then by \cite[Theorem~5.2.2]{CRW-DenseLC}, there is a compactly generated open subgroup $O$ of $G$ belonging to $\ms{R}$.
	
	Conversely, let $O$ be an open subgroup of $G$ such that $O \in \ms{R}_{\ldlat}$.  From Lemma~\ref{lem:simple_dynamics}(ii) and Theorem~\ref{thm:loc_decomp_compressible}, we see that $\TMon(O)$ is open in $O$.  Let $H$ be a non-trivial closed subgroup of $G$ that is normalised by $\TMon(O)$.  Given $h \in H$ there is a compact open subgroup $U$ of $\TMon(O)$ such that $hUh\inv \le \TMon(O)$, and hence $[h,u] \in H \cap \TMon(O)$ for all $u \in U$.  Since the quasicentre of $G$ is trivial, we deduce that $H \cap \TMon(O) \neq \triv$; since $\TMon(O)$ is topologically simple, we must have $H \ge \TMon(O)$.  Thus $\TMon(G)$ is the normal closure of $\TMon(O)$ in $G$ and $\TMon(G)$ is topologically simple.  Since $\TMon(O)$ is non-discrete and regionally expansive, so is $\TMon(G)$.  Thus $G \in \ms{R}$.
\end{proof}

\begin{proof}[Proof of Theorem \ref{thm:locally_in_A}]
Since (iii) and (iv) are equivalent, it suffices to  prove
\[
\text{ (v) } \Rightarrow \text{ (i) } \Rightarrow \text{ (ii) } \Rightarrow \text{ (iii) } \Rightarrow \text{ (vi) } \Rightarrow \text{ (v) }.
\]

The first implication is clear, the second follows from Theorem~\ref{thm:A_properties}, and the third holds since $\ms{S}_{\ldlat} \subseteq \ms{R}_{\ldlat}$.

Suppose (iii) holds and let $H$ be an open subgroup belonging to $\ms{R}_{\ldlat}$; by Lemma~\ref{lem:R_is_regional} we can take $H$ to be compactly generated.  Then $L$ acts faithfully on $X$ by Lemma~\ref{lem:faithful_germ}; moreover, by Lemma~\ref{lem:simple_dynamics} we know that $H$ is an expansive \tdlc group and acts minimally on $X$.  We see via Lemma~\ref{lem:loc_decomp:ldlat} that the action of $H$ on $X$ is also non-discretely micro-supported.  Thus (iii) implies (vi).

For the remainder of the proof, it suffices to assume (vi) and prove (v).  Thus we have compactly generated open subgroups $A$ and $B$ of $L$ such that $A$ is expansive and $B$ admits a faithful minimal micro-supported action on a compact zero-dimensional space $Y$.

By Lemmas~\ref{lem:centraliser_lattice} and \ref{lem:centraliser_lattice_minimal}, we may assume $Y$ is the canonical micro-supported $B$-space $\Omega_B$ and we can identify $\Omega_B$ with $\Omega_L$ as $B$-spaces.  In particular, $B$ acts minimally on $\Omega_L$, so $H_0 := \grp{A,B}$ acts minimally on $\Omega_L$.  The action of $L$ on $X$ is faithful by hypothesis, and occurs as a quotient space of $\Omega_L$.

Let $\{H_i \mid i \in I\}$ be the set of compactly generated open subgroups of $L$ containing $H_0$, let $J$ be the set of finite subsets of $\ldlat(L)$ and consider the action of $H_i$ on $X$ for some $i \in I$.  Then $H_i$ is expansive because it contains $A$ as an open subgroup, and acts minimally on $X$ because it contains $B$. 
Using Lemma~\ref{lem:expansive_quotient} we obtain an inverse system $\mc{Y}_i$ of $H_i$-spaces $Y$ with inverse limit $X$, such that $H_i$ acts faithfully and expansively on each $Y \in \mc{Y}_i$.  Such a quotient $H_i$-space $Y$ arises as $Y = \mf{S}(\mc{A}_Y)$ where $\mc{A}_Y$ is an $H_i$-invariant subalgebra of $\ldlat(L)$.  Since $H_i$ acts expansively on $Y$, the subalgebra $\mc{A}_Y$ must be generated by finitely many $H_i$-orbits; conversely, since $\mc{Y}_i$ has inverse limit $X$, we see that $\ldlat(L)$ is a directed union of $\{\mc{A}_Y \mid Y \in \mc{Y}_i\}$, so given $j \in J$ there is $Y_{ij} \in \mc{Y}_i$ such that $j \subseteq \mc{A}_{Y_{ij}}$.  Given $i \in I$ and $j \in J$, we now set $F_{ij} = \grp{H_i,\Al(H_i,Y_{ij})}$.

It follows by Proposition~\ref{prop:AS_characterisation} that $F_{ij}$ has $\ms{A}$-action on $Y_{ij}$; note that $\ol{\Al(H_i,Y_{ij})}$ is compactly generated by Corollary~\ref{cor:group_expansive_iff_alternating_compactly_gen}, so $F_{ij}$ is a compactly generated open subgroup of $L$.  We now equip $I \times J$ with the ordering that $(i,j) \le (i',j')$ if $H_i \le H_{i'}$ and $j \subseteq j'$, and write $L = \bigcup_{(i,j) \in I \times J}F_{ij}$.  To see this is a directed union, note that given $i_1,\dots,i_n \in I$ and $j_1,\dots,j_n \in J$, we can write $H_{i'} = \grp{F_{i_aj_b} \mid 1 \le a,b\le n}$ and $j' = \bigcup^n_{b=1}j_b$, and then $F_{i'j'} \ge F_{i_aj_b}$ for all $1 \le a,b \le n$.  Similarly, one sees that $X$ is an inverse limit of the associated spaces $Y_{ij}$.  Thus (v) holds and the proof is complete.
\end{proof}

At this point we can make some general observations about local isomorphism classes of groups in $\ms{R}_{\ldlat}$ and $\ms{S}_{\ldlat}$, which are not directly questions about piecewise full groups or the class $\ms{A}$.

\begin{rmk}\label{rmk:simple} \
	\begin{enumerate}
		\item To the authors' knowledge, this is the first proof that every group in $\ms{R}_{\ldlat}$ is locally isomorphic to a group in $\ms{S}_{\ldlat}$.  Beyond the locally decomposable case, it is still an open question to determine whether or not there is a group $G \in \ms{R}$ such that $G$ is not locally isomorphic to any group in $\ms{S}$.  Theorem~\ref{thm:locally_in_A} and Lemma~\ref{lem:simple_dynamics} show that in that case $\ldlat(G)$ is trivial, which rules out many of the known constructions of groups in $\ms{R}$ such that the monolith is not in $\ms{S}$.
		\item Let $\mc{C}$ be a local isomorphism class and let $\mc{C}_0$ consist of those groups in $\mc{C}$ with trivial quasicentre.  We have shown that $\mc{C}$ contains groups in $\ms{S}_{\ldlat}$ if and only if it satisfies the following three criteria:
		\begin{enumerate}[(a)]
			\item Some (equivalently every) $U \in \mc{C}_0$ acts faithfully on its decomposition lattice.  In particular, every group in $\mc{C}$ has discrete quasicentre.
			\item Some $A \in \mc{C}_0$ is compactly generated and expansive.  (Equivalently: the group of germs associated to $\mc{C}$ is regionally expansive.)
			\item Some $B \in \mc{C}_0$ is compactly generated and has a faithful minimal micro-supported action on a compact zero-dimensional space.
		\end{enumerate}
		
		Theorem~\ref{thm:jibranch:germ} below describes a special case of these criteria: if some $U \in \mc{C}$ is a profinite branch group, then $U = B$ already satisfies (a) and (c), so all that is left to check is (b).  Example \ref{eg:iterated_wreath_same_group} provides instances of branch groups $U$ such that $\ms{L}(U)$ is regionally expansive and illustrating that $\Der(\ms{L}(U))$ can be either a finite or an infinite directed union.  On the other hand, Example~\ref{eg:branch_infinite_local_prime_content} below gives uncountably many local isomorphism classes satisfying (a) and (c), but not (b).  Since groups in $\ms{R}$ are regionally expansive, some version of criterion (b) is an unavoidable obstacle for any construction of groups in $\ms{R}$ of a given isomorphism type, and likewise for $\ms{S}$.  However, the precise conditions under which a profinite group $U$ embeds as an open subgroup of a compactly generated expansive group are mysterious at present.		
		\item A special case of the results we have so far is that given $G \in \ms{A} \cap \ms{S}$ and $G \le H \le \Aut(G)$ such that $H$ is compactly generated (where $\Aut(G)$ is given the topology such that $G$ is embedded as an open subgroup), then $\Full(H;X) \in \ms{A}$, where $X$ is the $\ms{A}$-space of $G$, and hence $\Al(H;X) \in \ms{A} \cap \ms{S}$.  However, the authors are not aware of any example where $\Full(H;X) > \Full(G;X)$, or equivalently, where $\Al(H;X) > G$.  Since $\Al(H;X) \ge \Der(H)$ in this instance, it would be enough to find an example where $\Aut(G)/G$ is not abelian, and then we could choose $H/G$ to be any finitely generated nonabelian subgroup of $\Aut(G)/G$.
	\end{enumerate}
\end{rmk}

\subsection{When is a profinite branch group locally in $\ms{S}$?}

As we saw in Section \ref{sec:examples}, profinite branch groups form the basis of several known examples of groups in $\ms{S}$, provided that there is some compactly generated expansive group commensurating the branch group.

The worst possible failure of this happens when all compactly generated subgroups of the commensurator of the branch group are \defbold{SIN} groups (small invariant neighbourhoods), i.e., have a base of identity neighbourhoods consisting of compact open normal subgroups. 
Notice that in 	Lemma \ref{lem:filtering_normal}, if the filtering family $\mc{F}$ consists of open subgroups, then $G$ is a SIN group.

For profinite branch groups, being locally isomorphic to a group in $\ms{S}$ can be read off from the commensurator:

\begin{thm}\label{thm:jibranch:germ}
	Let $U$ be a  profinite branch group and let $L = \ms{L}(U)$.  Then the following are equivalent:
	\begin{enumerate}[(i)]
		\item $L$ is regionally expansive.
		\item $\Der(L)$ is open in $L$ and is a directed union of open subgroups in $\ms{A} \cap \ms{S}_{\ldlat}$, each with $\ms{A}$-action on $\mf{S}(\ldlat(L))$.  In particular, $L \in \ms{R}_{\ldlat}$.
	\end{enumerate}
	Moreover, if $U$ is just infinite and 
	$L$ is not regionally expansive, then every compactly generated subgroup of $L$ is a SIN group.
\end{thm}
\begin{proof}
	If (ii) holds then $L$ is regionally expansive, since all groups in $\ms{R}$ are regionally expansive. 
	So suppose that (i) holds.  Since $L$ is regionally expansive, there is a compactly generated open subgroup $H$ of $L$ that is expansive; note moreover that every compactly generated open subgroup containing $H$ is expansive.
	We may assume that $H$ contains $U$.  We see via Lemma~\ref{lem:faithful_germ} that the faithful transitive action of $U$ on $\mf{S}(\ldlat(U))$ naturally extends to a faithful transitive action of $H$ on $X := \mf{S}(\ldlat(L))$, which is isomorphic to $\mf{S}(\ldlat(U))$ as a $U$-space.
	By Lemma~\ref{lem:expansive_quotient}, there is a quotient $H$-space $Y$ of $X$ on which $H$ acts faithfully and expansively. 
	However, by Lemma~\ref{lem:branch_ldlat}, in fact we must have $Y = X$.  We now extend to $K = \Full(H;X)$ and observe that $K$ naturally contains $H$ as an open subgroup.
	
	Now let $\{H_i \mid i \in I\}$ be the set of compactly generated open subgroups of $L$ containing $H$.  Then we express $L$ as a directed union $L = \bigcup_{i \in I} H_i$; for each $H_i$ we construct $K_i = \Full(H_i;X) \le L$ as before.  For each $i \in I$, we see by Corollary~\ref{cor:group_expansive_iff_alternating_compactly_gen} that $\ol{\Al(K_i;X)}$ is compactly generated, and hence by Corollary~\ref{cor:A_is_open}, in fact $\Al(K_i;X)$ is open, simple and equal to $\Der(K_i)$.  Since the action of $K_i$ on $X$ is faithful and locally decomposable, we now have $\Der(K_i) \in \ms{S}_{\ldlat}$.  
	It is now clear that (ii) holds.
	
	Finally, suppose that $U$ is just infinite and
	let $H$ be a compactly generated open subgroup of $L$ containing $U$, and let $R$ be the intersection of open normal subgroups of $H$; note that $R \cap U$ is normal in $U$ and hence either open or trivial.  Given \cite[Lemma~4.13]{Reid-JI}, there are two possibilities: either $R=\triv$, in which case $H$ is a SIN group by Lemma~\ref{lem:filtering_normal}, or else $R$ is open and is a direct product of finitely many topologically simple groups, from which it is easy to see that $H$ is expansive.  If all compactly generated open subgroups of $L$ are SIN groups, then in fact all compactly generated subgroups of $L$ are SIN groups; if some compactly generated open subgroup of $L$ is expansive, then $L$ is regionally expansive.
\end{proof}

Example \ref{eg:iterated_wreath_same_group} provides instances of branch groups $U$ such that $\ms{L}(U)$ is regionally expansive and illustrating that $\Der(L)$ can be either a finite or an infinite directed union.

\begin{ex}\label{eg:branch_infinite_local_prime_content}
	An example of a family of  profinite branch groups $U$ such that $\ms{L}(U)$ is not regionally expansive are those with infinite local prime content, meaning the $p$-Sylow subgroup of $U$ is infinite for infinitely many $p$.  Such groups cannot be embedded in any regionally expansive \tdlc group: see \cite[Proposition~4.6]{CRW-Part1}.  For example, we could take $U$ to be the iterated wreath product
	\[
	U = \varprojlim ( \Alt(k_n) \wr \dots \wr \Alt(k_2) \wr \Alt(k_1) ),
	\]
	where $k_{i+1} > k_i\geq 5$ for all $i$, which gives $2^{\aleph_0}$ different local isomorphism classes of such groups.
\end{ex}

\subsection{Further consequences of Theorem \ref{thm:locally_in_A}}

\begin{cor}\label{cor:direct_powers}
	Let $G \in \ms{R}$, and write $G^n$ for the direct product of $1 \le n < \infty$ copies of $G$.
	\begin{enumerate}[(i)]
		\item If $G$ is faithful locally decomposable, then $\ms{L}(G^n)$ has an open subgroup in $\ms{A} \cap \ms{S}_{\ldlat}$.  Indeed, if $G$ is compactly generated then there is $H \in \ms{A}$ compactly generated such that $G \wr \Sym(n)$ occurs as an open subgroup of $H$; in that case, $\Der(H)$ is open and belongs to $\ms{A} \cap \ms{S}_{\ldlat}$.
		\item If $G$ is not faithful locally decomposable and $n>1$, then no open subgroup of $\ms{L}(G^n)$ belongs to $\ms{R}$.
	\end{enumerate}
\end{cor}

\begin{proof}
	(i)
	By Lemma~\ref{lem:R_is_regional}, any sufficiently large compactly generated open subgroup $G_0$ of $G$ belongs to $\ms{R}$.  We then have a minimal compressible action of $G_0$ on a compact zero-dimensional space $X$, and hence a minimal compressible action of the compactly generated expansive group $G_0 \wr \Sym(n)$ on the disjoint union of $n$ copies of $X$.  Using the implication (vi) $\Rightarrow$ (v) in Theorem~\ref{thm:locally_in_A}, we obtain $H \in \ms{A}$ compactly generated that contains $G_0 \wr \Sym(n)$ as an open subgroup; the remaining claims about $H$ follow by Theorem~\ref{thm:A_properties}.
	
	(ii)
	Since $G$ does not act faithfully on $\ldlat(G)$, the kernel $K$ of the action of $G$ on $\ldlat(G)$ contains $\TMon(G)$.  In turn, we see that $K^n$ is the kernel of the action of $G^n$ on $\ldlat(G^n)$. In particular, $K^n$ is not discrete, so given a \tdlc group $H$ locally isomorphic to $G^n$, then $H$ does not act faithfully on $\ldlat(H)$.  On the other hand, $\ldlat(H) = \ldlat(G^n)$ is non-trivial thanks to the given direct decomposition of $G^n$.  Lemmas~\ref{lem:centraliser_lattice} and~\ref{lem:simple_dynamics} now ensure that $H \not\in \ms{R}$.
\end{proof}

Here it should be noted that many examples of groups in $\ms{R}_{\ldlat}$ (such as those that are locally isomorphic to a regular profinite branch group, or those studied in \cite{SmithDuke}) are actually locally isomorphic to one of their proper direct powers, so given $G \in \ms{R}_{\ldlat}$ it often happens that the set $\{\ms{L}(G^n) \mid n \ge 1\}$ only contains finitely many isomorphism types of groups.

We also note the ubiquity of one-ended groups in the class $\ms{S}_{\ldlat}$, which is interesting given that many examples of groups in $\ms{S}_{\ldlat}$ are obtained as groups acting on trees.  The next corollary is an immediate consequence of Theorems~\ref{thm:NL} and \ref{thm:locally_in_A}. 

\begin{cor}\label{cor:one-ended}
	Every $G \in \ms{S}_{\ldlat}$ occurs as an open subgroup of some $H \in \ms{S}_{\ldlat}$, such that $H$ has property (NL); in particular, $H$ is one-ended.
\end{cor}

Corollary~\ref{cor:one-ended} does not generalise to groups $G \in \ms{S}$.
For example, the $p$-adic Lie group $G = \mathrm{PSL}_2(\Qb_p)$ belongs to $\ms{S}$ and admits a general type proper cocompact action on the regular tree of degree $p+1$, so certainly it is infinitely-ended. 
However, $G$ is rigid in the sense of \cite[Section~1.2]{BEW}: by \cite[Corollary~0.3]{Pink}, we have $\ms{L}(G) = \Aut(G)$.
On the other hand $G$ has the Howe--Moore property, so every proper open subgroup of $G$ is compact (see for instance \cite[Proposition 3.2]{ccltv_HoweMooreprops}).
Thus $G$ is actually the only group in $\ms{S}$ in its local isomorphism class.

In the context of Theorem~\ref{thm:locally_in_A}, we obtain a restriction on how $\ms{L}(G)$ acts on its decomposition lattice.

\begin{cor}\label{cor:germ_ldlat_minimality}
	Let $G \in \ms{R}_{\ldlat}$ and let $L = \ms{L}(G)$.  Then there is no proper non-trivial $\Der(L)$-invariant subalgebra of $\ldlat(L)$, or in other words, the Stone space $\mf{S}(\ldlat(L))$ does not admit any proper non-trivial Hausdorff quotient as a $\Der(L)$-space.
\end{cor}

\begin{proof}
	Let $X = \mf{S}(\ldlat(L))$ and let $\alpha,\beta \in \ldlat(L) \setminus \{0,\infty\}$.  By Theorem~\ref{thm:locally_in_A}, there is a compactly generated open subgroup $F$ of $L$ with $\ms{A}$-action on a quotient space $Y = \mf{S}(\mc{A})$ of $X$, such that $\{\alpha,\beta\} \subseteq \mc{A}$.  We then have $\Der(F) = \Al(F;Y)$, and by Lemma~\ref{lem:A(G)_quotientspace}, $\Der(F)$ does not preserve any proper non-trivial subalgebra of $\mc{A}$; thus any $\Der(F)$-invariant subalgebra of $\ldlat(L)$ that contains $\alpha$ must also contain $\beta$.  Since the pair $(\alpha,\beta)$ was arbitrary, we deduce that there is no proper non-trivial $\Der(L)$-invariant subalgebra of $\ldlat(L)$.
\end{proof}

We say that a \tdlc group $G$ is \defbold{locally finitely decomposable} if every compact open subgroup can be expressed as a finite direct product of directly indecomposable groups.  By \cite[Corollary~4.12]{CRW-Part1}, if $G$ is first-countable and $\QZ(G)=\triv$, then $G$ is locally finitely decomposable, or equivalently, every compact open subgroup has only finitely many direct factors, if and only if $\ldlat(G)$ is countable.

Suppose $G \in \ms{S}$ is such that $G$ is not locally finitely decomposable.   
Then by Corollary~\ref{cor:germ_ldlat_minimality}
we see that $\ol{\Mon}(\ms{L}(G))$ has an uncountable orbit on $\ldlat(G)$, so $\ol{\Mon}(\ms{L}(G))$ cannot be second-countable.  We deduce the following.

\begin{cor}\label{cor:Polish:finitely_decomposable}
	Let $G \in \ms{S}$ and suppose that $\ol{\Mon}(\ms{L}(G))$ is second-countable.  Then $G$ is locally finitely decomposable.
\end{cor}

This corollary is not too surprising in light of Smith's construction of uncountably many different groups in $\ms{S}$, all of which are locally isomorphic to one another; but it does suggest more generically that given $G \in \ms{S}_{\ldlat}$, we should expect the local isomorphism class of $G$ within $\ms{S}_{\ldlat}$ to be ``large'' as soon as $G$ is locally infinitely decomposable.  
In particular we obtain the following necessary condition for a group in $\ms{S}$ to be rigid in the sense of \cite{BEW}:

\begin{cor}
	Let $G \in \ms{S}$.  Suppose that every isomorphism between open subgroups of $G$ extends uniquely to an automorphism of $G$, and let $U$ be a compact open subgroup of $G$.
	Then $U$ has only finitely many direct factors.
\end{cor}

\begin{proof}
	The hypothesis implies that $G$ is normal in $\ms{L}(G)$, and in particular $G \ge \ol{\Mon}(\ms{L}(G))$.
	Thus $\ol{\Mon}(\ms{L}(G))$ is second-countable, and hence $G$ is locally finitely decomposable by Corollary~\ref{cor:Polish:finitely_decomposable}.
\end{proof}

\section{Open questions}\label{sec:questions}

The following is a list of questions that the authors were unable to answer, or are beyond the scope of this paper and are either a natural follow-up or tangent. 
A couple  are well known (and hard), others seem more tractable and yet others are small research programmes in themselves. 
They are more or less ordered by how they would naturally appear in the main text (and this gives a clue as to where the relevant definitions are).
We hope that at least some of them will become fruitful lines of inquiry for other researchers.
 
\subsection{Boolean inverse monoids and compact generation}

Taking inspiration from the countable setting, there may be non-discrete analogues of graph groupoids, or other sources of interesting \'{e}tale groupoids that can be modified to be non-discrete, that can be used to construct new examples of piecewise full topological groups.

 \begin{prob}
 	Find  non-discrete locally decomposable inverse monoids $M\leq\PHomeo_c(X)$ that give new examples of (Polish and/or locally compact) piecewise full groups, such that $M$ is not essentially a locally decomposable group with some extra partial homeomorphisms. 
 \end{prob}

Nekrashevych's finite generation theorem for $\Al(G)$ was originally proved via \'{e}tale groupoids; for our adaptation to compact generation, we took a detour into the theory of topological inverse monoids.  However, since our applications are to piecewise full \tdlc groups, it would be enlightening to have a more direct argument.

\begin{prob}
Find a short direct proof (without going via groupoids or inverse monoids) of Corollary~\ref{cor:group_expansive_iff_alternating_compactly_gen}.
\end{prob}

If $G$ is a discrete piecewise full group and $\mc{G}$ is the associated \'{e}tale groupoid, then the homology groups $H_i(\mc{G},\Zb)$ (especially for $i=0,1$), together with the index map $I: G \rightarrow H_1(\mc{G},\Zb)$, all in the sense of Matui \cite{MatuiHomology} (following Crainic--Moerdijk \cite{CrainicMoerdijk}) form a key part of the toolkit for studying $G$.  Although it is beyond the scope of this article, it would be useful to adapt this homological theory to the setting of non-discrete topological groups, particularly in the context of trying to understand the abelian group $\Full(G)/\Al(G)$ for $G \in \ms{A}$.  (For example, note that $\Full(G)/\Al(G)$ will detect the modular function of $\Full(G)$, which is trivial in the discrete case but could be non-trivial for $G \in \ms{A}$.)

\begin{prob}
Develop appropriate analogues of $H_0(\mc{G},\Zb)$, $H_1(\mc{G},\Zb)$ and $I$, for non-discrete piecewise full groups (or for Boolean topological inverse monoids), in a way that takes account of the group topology (or the monoid topology).
\end{prob}

Corollary~\ref{cor:group_expansive_iff_alternating_compactly_gen} provides a useful condition for the alternating group of a piecewise full group to be compactly generated.  However the relationship between compact generation of the alternating group and the piecewise full group remains mysterious. 
We do not know the answer to the following questions in the discrete or the non-discrete case (and possibly the answers are different).

\begin{que}
Let $X$ be the Cantor space and let $G \le \Homeo(X)$ be minimal and equipped with a locally compact Polish group topology, such that the action of $G$ is locally decomposable.  If one of $G$, $\Al(G)$ and $\Full(G)$ is compactly generated, are the other two compactly generated?
\end{que}

In the discrete case, some progress has been made in determining whether $\Full(G)/\Al(G)$ is finitely generated, mostly using the homological methods mentioned above.
See \cite{Nekra} (in particular Section 7) and references therein. 

\subsection{Simplicity of the derived group}

There are several results in the literature showing that if a piecewise full group $G$ satisfies some additional conditions, then the derived subgroup $\Der(G)$ is simple (and therefore coincides with $\Al(G)$).  However, it appears the following question is still open in general.

\begin{que}
	Let $G$ be a minimal piecewise full group of homeomorphisms of the Cantor space.  Is $\Der(G)$ simple?
\end{que}

\subsection{Piecewise full groups based on prescribed local actions on trees}

\begin{que}[See Example \ref{eg:waltraud_full_burgermozes}]
	For which intransitive finite permutation groups $F$ do we have $\Full(U(F))=\Comm_{\Homeo(\partial T)}(U(F))$?
\end{que}

\begin{que}[See Example \ref{eg:leboudec_G(FF')}]
	For which pairs $F\le F'$ of finite permutation groups do we have $\Al(G(F,F'))= \Al(U(F))$ ? This is equivalent, by Lemma \ref{lem:alternating_to_full} to $\Full(G(F,F'))= \Full(U(F))$ and to  $\mc{BI}(G(F,F'))=\mc{BI}(U(F))$.
\end{que}

\subsection{Branch groups and compactly generated expansive \tdlc groups}

Let us now focus on \tdlc groups containing a profinite branch group $U$ as an open subgroup.  If $U$ is commensurable with its $n$-th direct power $U^n$ then, by \cite{Roever}, the $n$-ary Higman--Thompson group $V_n$ embeds in the commensurator $\Comm(U)\cong \Comm_{\Homeo(\partial T)}(U)$. 
\begin{que}\label{que:branch_product}
If $U$ is a profinite branch group commensurable with its $n$-th direct power for $n > 1$, must we have $\rist_U(m) \cong K^n$ for some level $m$ and some $K\leq U$ of finite index?
\end{que}
A positive answer to Question~\ref{que:branch_product} does not imply that $U$ is regular branch.
For example, the group $U=G_3$ constructed in \cite{KS_nonregularbranchgroup} acts on the ternary rooted tree and  satisfies that $\rist_{G_3}(m)=H_m^{3^m}$ for every $m>0$ where $G_3/H_m\cong \bZ/2^{m+1} \bZ$ and $G_3/\bigcap_{m>0}H_m\cong \bZ^3$.

If the answer to Question~\ref{que:branch_product} is positive, then $V_n \le \AAut(T)$, and the embedding $V_n\hookrightarrow \Comm(U)$ is simply conjugation inside $\Homeo(\partial T)$, where $T$ is the tree on which $U$ acts as a branch group.  Moreover, in this case the group $G = \grp{U,V_n}$ is expansive, more specifically, $\bigcap_{g \in G}gUg\inv = \triv$.
On the other hand, if the answer to Question~\ref{que:branch_product} is negative, it is not immediate that $\grp{U,V_n}$ is expansive. 
\begin{que}
Does there exist a profinite branch group $U$ such that $V_n\leq \Comm(U)$ for some $n>1$ and a non-trivial normal subgroup of $U$ that is normalised by $V_n$?
\end{que}

A full classification of $\ms{A}$ (or equivalently $\ms{S}_{\ldlat}$) up to local isomorphism seems beyond the current state of research. 
However, the question of which profinite branch groups $U$ are locally isomorphic to groups in $\ms{A}$ may be more tractable. 
We know it is equivalent to asking if $U$ is locally isomorphic to a compactly generated expansive \tdlc group; in other words, if $U$ has an open subgroup $V$ in common with a compactly generated group $G$, such that $\bigcap_{g \in G}gVg^{-1} = \triv$.
 The latter is the question of ``local obstacles to regional expansivity'', which is probably  interesting for many classes of profinite groups, for example when  $U$  is a free pro-$p$ group. 
  Put another way: what kind of profinite groups occur as vertex stabilisers of faithful vertex-transitive actions on locally finite graphs?

  As far as the authors are aware, even the following basic question is open:
	
	\begin{que}
		Are there uncountably many different local isomorphism classes of compactly generated expansive \tdlc groups?
	\end{que}
  
More specifically, the following would be good to know:

\begin{que}\label{que:branch_expansive}
	Are there uncountably many profinite branch groups $U$ (up to finite index) occurring as open subgroups of compactly generated expansive \tdlc groups?
\end{que}

For Question~\ref{que:branch_expansive}, if there are uncountably many commensurability classes, it would answer in the affirmative the open question of whether there are uncountably many local isomorphism classes of groups in $\ms{S}$.  On the other hand if there are only countably many, there is some prospect of classifying them using some natural parameters.

\subsection{Automorphisms of simple groups in $\ms{A}$}

Proposition~\ref{prop:piecewise_rigidity} puts some restrictions on the automorphism group of a piecewise full group, but we do not know the answer to the following.

\begin{que}
	Let $G$ be a simple group in $\ms{A}$.
	\begin{enumerate}[(i)]
		\item Can every automorphism of $G$ be realised as an element of $\Full(G)$?
		\item Is $\Aut(G)/\mathrm{Inn}(G)$ finitely generated?
	\end{enumerate}
\end{que}
 
\subsection{Groups containing a group in $\ms{S}_{\mc{LD}}$ as an open subgroup}

There remain two natural questions concerning the cardinality of $\ldlat(G)$ for $G \in \ms{S}_{\ldlat}$ that are not addressed by Corollary~\ref{cor:germ_ldlat_minimality}.

\begin{que}
	Does there exist $G \in \ms{S}$ (or more specifically, $G \in \ms{S}_{\ldlat}$) such that $\ldlat(G)$ is countable, but $\ol{\Mon}(\ms{L}(G))$ is not second-countable?
\end{que}

\begin{que}
	Does there exist $G \in \ms{S}$ such that $\ldlat(G)$ is uncountable, but such that there are only countably many isomorphism types of groups in $\ms{S}$ locally isomorphic to $G$?
\end{que}

Note that for the latter question: if $\ldlat(G)$ is uncountable, we know that the number of open subgroups of $\ms{L}(G)$ belonging to $\ms{S}$ is uncountable, but in some cases perhaps they fall into countably many conjugacy classes, hence only countably many isomorphism types.

	The following question is inspired by Corollary~\ref{cor:one-ended}.
	
	\begin{que}
		Let $G \in \ms{S}_{\ldlat}$.  Is $G$ (or some finite direct power of $G$) locally isomorphic to some $H \in \ms{S}_{\ldlat}$, such that $H$ has infinitely many ends?
	\end{que}
	
	Here we note that for any $G \in \ms{S}$, one can use Smith's construction \cite{SmithDuke} to obtain a group $H = U(G,G) \in \ms{S}_{\ldlat}$ with infinitely many ends, such that $H$ has an open subgroup of the form $K \rtimes G$ for $K$ compact.  So the open question here is specifically about embedding $G$ as an \emph{open} subgroup in an infinitely-ended $H \in \ms{S}_{\ldlat}$.
	
	It is suspected that there are some regular branch groups that are not locally isomorphic to any infinitely-ended group in $\ms{S}$, but there is something to prove here.

\bibliographystyle{abbrv}
\bibliography{biblio}
		
	\end{document}